\documentclass[letterpaper, 12pt, oneside]{book}

\usepackage{thesisstyle}
\usepackage[book]{MHmacros}

\usepackage[pdfauthor={Miha E.\ Habič},
    pdftitle={Joint Laver diamonds and grounded forcing axioms},
    hidelinks
]{hyperref}

\begin{document}

\frontmatter

\begin{titlepage}

\begin{center}

~\vspace{2in}

\textsc{Joint Laver diamonds and grounded forcing axioms} \\[0.5in]
{by} \\[0.5in]
\textsc{Miha E.\ Habič} 

\vspace{\fill}
A dissertation submitted to the Graduate Faculty in Mathematics in partial fulfillment of the requirements for the degree of Doctor of Philosophy, The City University of New York \\[0.25in]
2017

\end{center}

\end{titlepage}

\setcounter{page}{2}

\phantom{}\vspace{\fill}
\begin{center}
\copyright~2017\\
\textsc{Miha E.\ Habič}\\
All Rights Reserved\\
\end{center}
\begin{center}
Joint Laver diamonds and grounded forcing axioms

by

Miha E.\ Habič

\begin{singlespace}
This manuscript has been read and accepted for the Graduate Faculty in Mathematics in satisfaction of the dissertation requirement for the degree of Doctor of Philosophy.
\end{singlespace}
\end{center}

\vspace{0.75in}

\begin{tabular}{p{1.75in}p{0.5in}p{3.5in}}
~                                   & & \\
\hrulefill                          & &\hrulefill \\
Date                                  & & Joel David Hamkins\\
~                                & & Chair of Examining Committee\\
~                                   & & \\
~                                   & & \\
\hrulefill                          & &\hrulefill \\
Date                                   & & Ara Basmajian\\
~                                & & Executive Officer\\
\end{tabular}

\vspace{0.75in}

\begin{center}
\begin{tabular}{c}
Supervisory Committee: \\
Joel David Hamkins \\
Arthur W.\ Apter \\
Gunter Fuchs \\
\end{tabular}

\vspace{\fill}

\textsc{The City University of New York}
\end{center}
\begin{center}
Abstract \\
\textsc{Joint Laver diamonds and grounded forcing axioms} \\
by \\
\textsc{Miha E.\ Habič} \\[0.25in]
\end{center}

\vspace{0.25in}

\noindent Advisor: Joel David Hamkins

\vspace{0.25in}

In chapter~\ref{chap:joint} a notion of independence for 
diamonds and Laver diamonds is investigated.
A sequence of Laver diamonds for \(\kappa\) is \emph{joint} if for any sequence of targets
there is a single elementary embedding \(j\) with critical point \(\kappa\) such that
each Laver diamond guesses its respective target via \(j\).
In the case of measurable cardinals (with similar results holding for (partially) 
supercompact cardinals) I show that a single Laver diamond for \(\kappa\) 
yields a joint sequence of length \(\kappa\), and I give strict separation results
for all larger lengths of joint sequences. 
Even though the principles get strictly stronger in terms of direct implication,
I show that they are 
all equiconsistent.
This is contrasted with the case of \(\theta\)-strong cardinals where, for certain \(\theta\),
the existence of even the shortest joint Laver sequences carries nontrivial consistency
strength.
I also formulate a notion of jointness for ordinary \(\diamondsuit_\kappa\)-sequences on
any regular cardinal \(\kappa\). The main result concerning these shows that there is no separation according to length and a single \(\diamond_\kappa\)-sequence yields joint
families of all possible lengths.

In chapter~\ref{chap:grma} the notion of a \emph{grounded} forcing axiom is introduced
and explored in the case of Martin's axiom. This grounded Martin's axiom, a
weakening of the usual axiom, states that the universe is a ccc forcing extension
of some inner model and the restriction of Martin's axiom to the posets coming from
that ground model holds. I place the new axiom in the hierarchy of fragments of Martin's
axiom and examine its effects on the cardinal characteristics of the continuum. I also show that the grounded version is quite a bit more robust under mild forcing
than Martin's axiom itself.
\chapter*{Acknowledgments}

I wish to thank my advisor, Joel David Hamkins, for his constant support and guidance through
my studies and the writing of this dissertation. It is safe to say that this text would
not exist without his keen insight, openness to unusual ideas, and patience. It has been
a pleasure and a privilege to work with him and I could not have asked for a better
role model of mathematical ingenuity, rigour, and generosity.

I would also like to thank Arthur Apter and Gunter Fuchs, the other members of my
dissertation committee. Thank you both for your contributions, the many conversations, and
for taking the time to read through this text.
Thank you as well to Kameryn, Vika, Corey, the whole New York logic community, and the many other
friends at the Graduate Center. I will remember fondly the hours spent talking about
mathematics, playing games, and doing whatever other things graduate students do.

Thank you, Kaethe. You have been (and hopefully remain) my best friend, sharing my
happy moments and supporting me in the less happy ones. Thank you for
letting me explain every part of this dissertation to you multiple times,
and thank you for teaching me about mathematics that I would never have
known or understood without you. I am incredibly fortunate to have met you
and I can only hope to return your kindness and empathy in the future.

Lastly, thank you to my parents and other family members who have stood by me
and supported me through my long journey. I hope I have made
you proud. Rad vas imam.

\tableofcontents

\listoffigures

\mainmatter

\chapter*{Introduction}
\markboth{}{INTRODUCTION}
\label{chap:intro}
\addcontentsline{toc}{chapter}{Introduction}

\noindent
This dissertation consists of two largely independent chapters in which I explore some
questions in set theory related to guessing principles, and variants of forcing axioms.

Set-theoretic guessing principles are combinatorial gems, exemplifying the
compact nature of certain infinities. They find use in many
inductive or diagonalization arguments, where they allow us to anticipate a wide variety of
objects. For example, we might attempt to construct a branchless tree by using a guessing
principle to anticipate what a potential branch might look like \emph{while still building
the tree} and avoiding it. A version of this idea was essentially the
primordial application of the guessing principle \(\diamond\) by Jensen, and
since then a number of other principles have emerged: some allow us to anticipate
more complicated sets, others guarantee that the provided guesses are correct
quite often, etc. An important variant are the Laver diamond principles which arise
alongside many members of the large-cardinal hierarchy.

Chapter~\ref{chap:joint} is concerned with the question of how many guessing objects
of a certain kind can a cardinal carry. For example, how many different 
\(\diamond_\kappa\)-sequences or Laver functions can there be for a given \(\kappa\)?
To avoid trivial perturbations of a single sequence, the question
should rather ask for the number of independent such sequences, for some
notion of independence. We suggest and investigate \emph{jointness} as a suitable candidate;
a family of guessing sequences is joint if all of its members can guess their targets
simultaneously and independently of one another. We investigate jointness for
Laver diamonds on (partially) supercompact, and strong cardinals. In the case of
partially supercompact cardinals we show that, while having joint Laver families of various
different lengths requires no additional consistency strength, there are no nontrivial
implications between these principles: a cardinal may carry \(\lambda\) many joint
Laver functions but not \(\lambda^+\) many, for example. The situations is
markedly different in the case of partially strong cardinals. There, in certain cases,
having any (infinite) number of joint Laver diamonds requires consistency strength beyond 
just the large cardinal itself.
The chapter is concluded with a discussion of jointness for \(\diamond_\kappa\)-sequences.
The definition needs to be modified to make sense in the small-cardinal context, but
a natural coherence condition suggests itself, in terms of the existence of normal
uniform filters or, equivalently, generic elementary embeddings.

Chapter~\ref{chap:grma} investigates \emph{grounded} versions of forcing axioms, with
a particular focus on Martin's axiom. Forcing axioms are combinatorial statements,
generalizing the Baire category theorem, which state that certain ``generic'' objects
exist (such as a generic point in a typical application of the mentioned theorem).
They also serve to bridge the gap between set theory and other fields (particularly
topology, analysis, and algebra) by presenting a ready-to-use account of the powerful
technique of forcing. The strength of these forcing axioms can be calibrated by changing the 
meaning of the word ``generic'' in interesting ways. One might restrict or expand
the field of structures (posets) that we find these generics in and arrive at Martin's
axiom, or the proper forcing axiom, or Martin's maximum or many others.

Grounding a forcing axiom is exactly this kind of restriction. Specifically, we connect
the axiom even more closely to its forcing nature and require that the posets
under consideration come from a ground model of the universe of a suitable kind:
a ccc ground model for Martin's axiom, a proper ground model for the proper forcing axiom
etc. Our investigation centres on the grounded Martin's axiom (\grma) and its relationship to
the more well-known fragments of Martin's axiom (\ma). We show that \grma is largely
independent of the other fragments, obtaining the following diagram of implications.

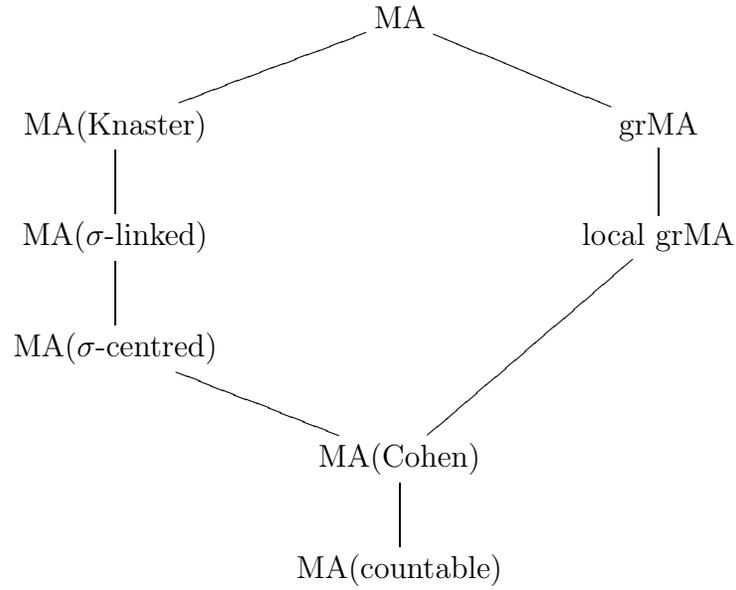
\begin{figure}[bht]
\[
\xymatrix{
&\ma\ar@{-}[dl]\ar@{-}[dr]&\\
\mak\ar@{-}[d]&&\grma\ar@{-}[d]\\
\textup{MA(\(\sigma\)-linked)}\ar@{-}[d]&&\text{local \grma}\ar@{-}[ddl]\\
\mas\ar@{-}[dr]\\
&\mathrm{MA(Cohen)}\ar@{-}[d]\\
&\mathrm{MA(countable)}}
\]
\end{figure}

\noindent
The grounded axiom is also more permissive regarding
the combinatorics of the real line. For example, it is consistent with the cardinality of
the continuum being singular and that the real line can be covered by \(\aleph_1\) many
Lebesgue-null sets. We also explore the behaviour of \grma under forcing adding a single
real. We show, using a variant of term forcing, that, in contrast with \ma, the grounded
axiom is preserved when adding a single Cohen or random real.

\chapter{Joint diamonds and Laver diamonds}
\label{chap:joint}

The concept of a Laver function, introduced for supercompact cardinals 
in~\cite{Laver1978:MakingSupercompactnessIndestructible},
is a powerful strengthening of the usual \(\diamond\)-principle to the large cardinal setting.
It is based on the observation that a large variety of large cardinal properties give rise to
different notions of ``large'' set, intermediate between stationary and club, and these
are then used to provide different guessing principles, where we require that the
sequence guesses correctly on ever larger sets. This is usually recast in terms of
elementary embeddings or extensions (if the large cardinal in question admits such a
characterization), using various ultrapower constructions. For example, in the case of a 
supercompact cardinal \(\kappa\), the usual definition states that a \emph{Laver function}
for \(\kappa\) is a function \(\ell\colon \kappa\to V_\kappa\) such that for any \(\theta\)
and any \(x\in H_{\theta^+}\) there is a \(\theta\)-supercompactness embedding
\(j\colon V\to M\) with critical point \(\kappa\) such that \(j(\ell)(\kappa)=x\) (this
ostensibly second order definition can be rendered in first order language by replacing
the quantification over arbitrary embeddings with quantification over
ultrapowers by measures on \(\mathcal{P}_\kappa(\theta)\), as in Laver's original account).

Laver functions for other large cardinals were later defined by Gitik and Shelah
(see~\cite{GitikShelah1989:IndestructibilityOfStrong}), Corazza
(see~\cite{Corazza1989:LaverSequencesExtendible}), Hamkins 
(see~\cite{Hamkins2002:StrongDiamondPrinciples}) and others. The term \emph{Laver diamonds} 
has been suggested to more strongly underline the connection between
the large and small cardinal versions.

In this chapter we examine the notion of \emph{jointness} for both ordinary and Laver
diamonds. We shall give a simple example in section~\ref{sec:defs}; for now let us just
say that a family of Laver diamonds is joint if they can guess their targets simultaneously
and independently of one another.
Section~\ref{sec:defs} also introduces some terminology that will ease later discussion.
Sections~\ref{sec:JLDSc} and~\ref{sec:JLDStrong} deal with the outright existence
or at least the consistency of joint Laver sequences for supercompact and strong
cardinals, respectively. Our results will show that in almost all cases the existence of
a joint Laver sequence of maximal possible length is simply equiconsistent
with the particular large cardinal. The exception are the \(\theta\)-strong cardinals
where \(\theta\) is a limit of small cofinality, for which we prove that additional
strength is required for even the shortest joint sequences to exist. 
We also show that there are no
nontrivial implications between the existence of joint Laver sequences of different lengths.
Section~\ref{sec:JD} considers joint \(\diamond_\kappa\)-sequences and their relation
to other known principles. Our main result there shows that \(\diamond_\kappa\) is simply
equivalent to the existence of a joint \(\diamond_\kappa\)-sequence of any possible length.

\section{Jointness: a motivating example}
\label{sec:defs}
All of the large cardinals we will be dealing with in this chapter are characterized by
the existence of elementary embeddings of the universe into certain transitive inner models 
which have that cardinal
as their critical point. We can thus speak of embeddings associated to a measurable, 
a \(\theta\)-strong, or a 17-huge cardinal \(\kappa\) etc.
Since the definitions of (joint) Laver diamonds for these various large cardinals
are quite similar, we give the following general definition as a framework to
aid future exposition.

\begin{definition}
\label{def:guessingFunctions}
Let \(j\) be an elementary embedding of the universe witnessing the largeness of
its critical point \(\kappa\) (e.g. a measurable, or a hugeness embedding)
and let \(\ell\) be a function defined on \(\kappa\).
We say that a set \(a\), the target, is \emph{guessed by \(\ell\) via \(j\)} 
if \(j(\ell)(\kappa)=a\).

If \(A\) is a set or a definable class, say that \(\ell\) is an \emph{\(A\)-guessing Laver 
function} if for any \(a\in A\) there is an embedding \(j\), witnessing the largeness of 
\(\kappa\), such 
that \(\ell\) guesses \(a\) via \(j\). If there is an \(A\)-guessing Laver function
for \(\kappa\), we shall say that \(\Ld_\kappa(A)\) holds.\footnote{Different notation has been used by different authors to denote the existence of a Laver function. We chose here to follow Hamkins~\cite{Hamkins2002:StrongDiamondPrinciples}.}
\end{definition}

To simplify the notation even more, we shall, in the coming sections, 
associate to each type of large cardinal 
considered, a default set of targets \(A\) (for example, when talking about measurable
cardinals \(\kappa\) we will mostly be interested in targets from \(A=H_{\kappa^+}\)). 
In view of this, whenever we neglect
the mention of a particular class of targets, this default \(A\) will be intended.

We will often specify the type of large cardinal embeddings we
have in mind explicitly, by writing \(\Ldmeas_\kappa\), or \(\Ldthetasc_\kappa\) etc.
This is to avoid ambiguity; for example, we could conceivably start with a supercompact
cardinal \(\kappa\) but only be interested in its measurable Laver functions.
Even so, to keep the notation as unburdened as possible, we may sometimes omit
the specific large cardinal property under consideration when it is clear from
context.
%

As a further complication, the stated definition of an \(A\)-guessing Laver function is 
second-order, since we are quantifying over all possible embeddings \(j\). This seems 
unavoidable for arbitrary \(A\). However, the default sets of targets we shall be working
with are chosen in such a way that standard factoring arguments allow us to restrict
our attention to ultrapower or extender embeddings. The most relevant definitions of
Laver functions can therefore be recast in first-order language.

Given the concept of a Laver diamond for a large cardinal \(\kappa\), one might ask
how many different Laver diamonds can there be on a given cardinal \(\kappa\).
Immediately after posing the question we realize that it is not that interesting.
Since the behaviour of Laver diamonds is determined by their restrictions to large
(in the sense of measure) subsets of \(\kappa\), we can simply take a fixed Laver function
and perturb it on an unbounded but small (nonstationary, say) subset of its domain
and preserve its guessing property. There are \(2^\kappa\) many such perturbations and
thus trivially \(2^\kappa\) many distinct Laver functions.

This answer is very much unsatisfactory. We are not interested in insignificantly modified
versions of a single Laver function. We definitely do not want to count the functions above
as distinct: they cannot even guess distinct targets! So perhaps we are interested
in collections of Laver functions whose targets needn't all be equal, given a fixed embedding 
\(j\). But even that seems
insufficient. For example, given a Laver function \(\ell\) we can define a new one
\(\ell'\) by letting \(\ell'(\xi)\) be the unique element of \(\ell(\xi)\), if one exists.
Both of these are Laver functions, their targets under a fixed \(j\) are never equal,
and yet the targets cannot be independently chosen at all. Again, one can always 
find large families of Laver functions like this, providing another unsatisfactory answer
to our original question.

Finally we come to the right idea of when to Laver functions should be deemed to be
distinct, or independent. The key property should be that their targets, under a single
embedding \(j\), can be chosen completely independently. Let us illustrate this
situation with a simple example.

Suppose \(\ell\colon\kappa\to V_\kappa\) is a supercompactness Laver function. We can
then define two functions \(\ell_0,\ell_1\) by letting \(\ell_0(\xi)\) and
\(\ell_1(\xi)\) be the first and second components, respectively, of \(\ell(\xi)\), if
this happens to be an ordered pair. These two are then easily seen to be Laver functions
themselves, but have the additional property that, given any pair of targets \(a_0,a_1\),
there is a \emph{single} supercompactness embedding \(j\) such that
\(j(\ell_0)(\kappa)=a_0\) and \(j(\ell_1)(\kappa)=a_1\). This additional trait, where
two Laver functions are, in a sense, enmeshed, we call jointness.

\begin{definition}
\label{def:jointLaverDiamonds}
Let \(A\) be a set or a definable class and let \(\kappa\) be a cardinal with a notion of 
\(A\)-guessing Laver function. A sequence \(\vec{\ell}=\langle \ell_\alpha;\alpha<\lambda\rangle\)
of \(A\)-guessing Laver functions is an \emph{\(A\)-guessing joint Laver sequence} if for any 
sequence \(\vec{a}=\langle a_\alpha;\alpha<\lambda\rangle\) of targets from \(A\) there is a single 
embedding \(j\), witnessing the largeness of \(\kappa\), such that each 
\(\ell_\alpha\) guesses \(a_\alpha\) via \(j\). If there is
an \(A\)-guessing joint Laver sequence of length \(\lambda\) for \(\kappa\), we shall say 
that \(\jLd_{\kappa,\lambda}(A)\) holds.
\end{definition}

In other words, a sequence of Laver diamonds is joint if, given any sequence of targets,
these targets can be guessed simultaneously by their respective Laver diamonds.
The guessing property of joint Laver sequences is illustrated in the following
diagram:

\begin{figure}[h]
\centering
\begin{tikzpicture}
\draw (0,0) rectangle (4,3);
\draw [help lines, black, ystep=3, xstep=0.5] (0,0) grid (3,3);
\node at (3.5,1.5) {\(\dots\)};
\node [below,right] at (4,0) {\(\lambda\)};
\node [left] at (0,3) {\(\kappa\)};
\node [below] at (2,0) {\(\vec{\ell}\)};
\node [rotate=90] at (0.75,1.5) {\(\ell_\alpha\)};
\node [rotate=90] at (2.25,1.5) {\(\ell_\beta\)};
\draw (0.75,3.35) circle [radius=0.3] node {\(a_\alpha\)};
\draw (2.25,3.35) circle [radius=0.3] node {\(a_\beta\)};
\draw [-{>[scale=1.5]}] (4.5,2) to [out=25, in=155] node[midway,above]{\(j\)}(8.5,2);
\draw (9,0) rectangle (14,4);
\draw [help lines, black, ystep=4, xstep=0.5] (9,0) grid (13,4);
\node at (13.5, 2) {\(\dots\)};
\node [below,right] at (14,0) {\(j(\lambda)\)};
\node [left] at (9,4) {\(j(\kappa)\)};
\node [left] at (9,3) {\(\kappa\)};
\draw [dashed] (9,3)--(14,3);
\node [below] at (11.5,0) {\(j(\vec{\ell})\)};
\node [rotate=90] at (9.75,1.5) {\(j(\ell_\alpha)\)};
\node at (9.8,3.2) {\(a_\alpha\)};
\node [rotate=90] at (12.25,1.5) {\(j(\ell_\beta)\)};
\node at (12.3,3.2) {\(a_\beta\)};
\end{tikzpicture}
\caption{The guessing property of joint Laver sequences}
\end{figure}
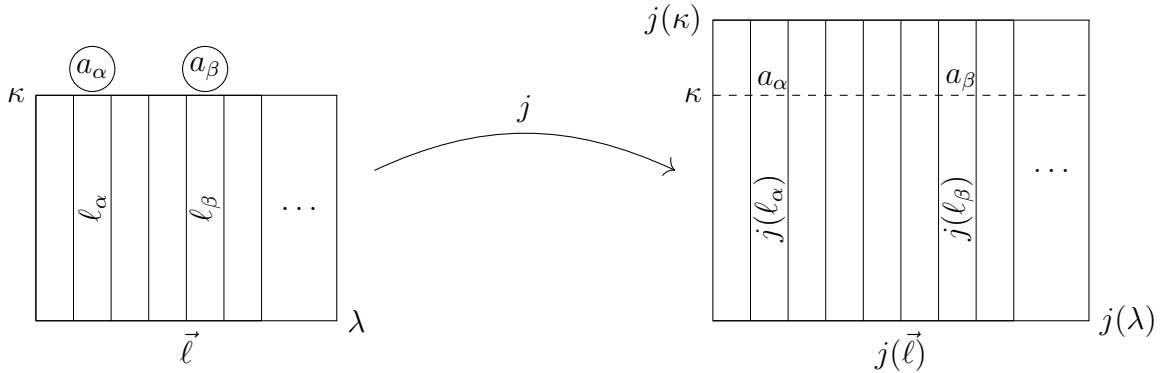

We must be careful to distinguish between and entire sequence being jointly Laver
or its members being pairwise jointly Laver. It is not difficult to find examples of
three (or four or even infinitely many) 
Laver functions that are pairwise joint but not fully so. For example,
given two joint Laver functions \(\ell_0\) and \(\ell_1\) we might define
\(\ell_2(\xi)\) to be the symmetric difference of \(\ell_0\) and \(\ell_1\). It is easy
to check that any two of these three functions can have their targets freely chosen, but
the third one is uniquely determined by the other two.

Jointness also makes sense for ordinary diamond sequences, but needs to be formulated
differently, since elementary embeddings do not (obviously) appear in that setting.
Rather, we distil jointness for Laver diamonds into a property of certain
ultrafilters and then apply this to more general filters and diamond sequences.
We explore this further in section~\ref{sec:JD}.

\section{Joint Laver diamonds for supercompact cardinals}
\label{sec:JLDSc}

\begin{definition}
A function \(\ell\colon\kappa\to V_\kappa\)
is a \(\theta\)-supercompactness Laver function for \(\kappa\) if it guesses elements of 
\(H_{\theta^+}\) via \(\theta\)-supercompactness embeddings with critical point \(\kappa\).
This also includes the case of \(\kappa\) being measurable (as this is equivalent to
it being \(\kappa\)-supercompact).

If \(\kappa\) is fully supercompact then a function \(\ell\colon\kappa\to V_\kappa\) is
a Laver function for \(\kappa\) if it is a \(\theta\)-supercompactness Laver function
for \(\kappa\) for all \(\theta\).
%
%
%
\end{definition}


Observe that there are at most \(2^\kappa\) many \(\theta\)-supercompactness
Laver functions for \(\kappa\), since there are only \(2^\kappa\) many functions
\(\kappa\to V_\kappa\). 
Since a joint Laver sequence cannot have the same function appear on two different
coordinates (as they could never guess two different targets),
his implies
that \(\lambda=2^\kappa\) is the largest cardinal for which 
there could possibly be a joint Laver sequence of length \(\lambda\). 
Bounding from the other side,
a single \(\theta\)-supercompactness Laver function already yields a joint Laver 
sequence of length \(\theta\).

\begin{proposition}
\label{prop:SCHasSomeJLDiamond}
If\/ \(\Ldthetasc_\kappa\) holds 
then there is a \(\theta\)-supercompactness joint Laver sequence for \(\kappa\) of length 
\(\lambda=\min\{\theta,2^\kappa\}\).
\end{proposition}

\begin{proof}
Let \(\ell\) be a Laver function for \(\kappa\). Fix a subset \(I\) of 
\(\mathcal{P}(\kappa)\)
of size \(\lambda\) and an enumeration \(f\colon \lambda\to I\). 
For \(\alpha<\lambda\) define \(\ell_\alpha\colon \kappa\to V_\kappa\) by
\(\ell_\alpha(\xi)=\ell(\xi)(f(\alpha)\cap\xi)\) if this makes sense and
\(\ell_\alpha(\xi)=\emptyset\) otherwise. We claim that
\(\langle \ell_\alpha;\alpha<\lambda\rangle\) is a joint Laver sequence for \(\kappa\).

To verify this let \(\vec{a}=\langle a_\alpha;\alpha<\lambda\rangle\) be a sequence of 
elements of \(H_{\theta^+}\). Then \(\vec{a}\circ f^{-1}\in H_{\theta^+}\), so by assumption 
there is a \(\theta\)-supercompactness embedding \(j\colon V\to M\) such that 
\(j(\ell)(\kappa)=\vec{a}\circ f^{-1}\).
But now observe that, by elementarity, for any \(\alpha<\lambda\)
\[
j(\ell_\alpha)(\kappa)=j(\ell)(\kappa)(j(f(\alpha))\cap\kappa)=
j(\ell)(\kappa)(f(\alpha))=a_\alpha \qedhere
\]
%
\end{proof}

Of course, if a given Laver function works for many degrees of supercompactness
then the joint Laver sequence derived above will work for those same degrees.
In particular, if \(\kappa\) is fully supercompact then this observation, combined with
Laver's original construction, gives us a supercompactness joint Laver sequence
of length \(2^\kappa\).

\begin{corollary}
\label{cor:SCHasLongJLDiamond}
If \(\kappa\) is supercompact then \(\jLdsc_{\kappa,2^\kappa}\) holds.
\end{corollary}

Proposition~\ref{prop:SCHasSomeJLDiamond} essentially shows that joint Laver sequences
of maximal length
exist automatically for cardinals with a high degree of supercompactness. Since we will
be interested in comparing the strengths of the principles \(\jLd_{\kappa,\lambda}\)
for various \(\lambda\), we will in the remainder of this section be mostly concerned
with cardinals \(\kappa\) which are not \(2^\kappa\)-supercompact (but are at least
measurable).

\subsection{Creating long joint Laver diamonds}

We now show that the existence of \(\theta\)-supercompactness joint Laver sequences
of maximal length does not require strength beyond \(\theta\)-supercompactness
itself. 
The following lemma is well-known.

\begin{lemma}
\label{lemma:ThetaSCMenasFunction}
If \(\kappa\) is \(\theta\)-supercompact then there is a function \(f\colon \kappa\to\kappa\) 
with the Menas property, i.e.\ we have \(j(f)(\kappa)>\theta\)
for some \(\theta\)-supercompactness embedding 
\(j\colon V\to M\) with critical point \(\kappa\).
\end{lemma}

\begin{theorem}
\label{thm:ForceLongJLDiamondSC}
If \(\kappa\) is \(\theta\)-supercompact then there is a forcing extension in which
\(\jLdthetasc_{\kappa,2^\kappa}\) holds.
\end{theorem}
\begin{proof}
Since \(\theta\)-supercompactness of \(\kappa\) implies its \(\theta^{<\kappa}\)-supercompactness,
we may assume that \(\theta^{<\kappa}=\theta\). Furthermore, we assume that
\(2^\theta=\theta^+\), since this may be forced without adding subsets to 
\(\mathcal{P}_\kappa\theta\). 
Fix a Menas function \(f\) as in lemma~\ref{lemma:ThetaSCMenasFunction}. 
Let \(\P_\kappa\) be the length \(\kappa\) Easton support iteration which forces with 
\(\Q_\gamma=\Add(\gamma,2^\gamma)\) at inaccessible closure points of \(f\), i.e.\ those
inaccessible \(\gamma\) for which \(f[\gamma]\subseteq\gamma\).
Finally, let \(\P=\P_\kappa*\Q_\kappa\). 
It is useful to note that forcing with \(\P\) does not change the value of \(2^\kappa\). 
Let \(G*g\subseteq\P\) be generic; we will extract a joint Laver sequence from \(g\).

If \(g(\alpha)\) is the \(\alpha\)-th subset added by \(g\), we view it as a sequence
of bits. Given any \(\xi<\kappa\) we may then view the segment of \(g(\alpha)\)
between the \(\xi\)th bit and the next marker as the Mostowski code of an element of
\(V_\kappa\) (employing bit doubling or some other coding scheme which admits
end-of-code markers). We then define \(\ell_\alpha\colon\kappa\to V_\kappa\) as
follows:
given an inaccessible \(\xi\), let \(\ell_\alpha(\xi)\) be the set coded by
\(g(\alpha)\) at \(\xi\); otherwise let \(\ell_\alpha(\xi)=\emptyset\).
We claim that \(\langle \ell_\alpha;\alpha<2^\kappa\rangle\) is a joint Laver sequence.

Let \(\vec{a}=\langle a_\alpha;\alpha<2^\kappa\rangle\) be a sequence of targets in
\(H_{\theta^+}^{V[G][g]}\). 
Let \(j\colon V\to M\) be the ultrapower embedding by a normal measure on 
\(\mathcal{P}_\kappa\theta\)
which corresponds to \(f\), i.e.\ such that \(j(f)(\kappa)>\theta\).
We will lift this embedding through the forcing \(\P\) in \(V[G][g]\).

The argument splits into two cases, depending on the size of \(\theta\). We deal first
with the easier case when \(\theta\geq 2^\kappa\). In this case the poset 
\(j(\P_\kappa)\) factors as \(j(\P_\kappa)=\P_\kappa*\Q_\kappa*\P_{\text{tail}}\).
Since \(j(f)(\kappa)>\theta\), the next stage of forcing in \(j(\P_\kappa)\) above \(\kappa\)
occurs after \(\theta\), so \(\P_{\text{tail}}\) is \(\leq\theta\)-closed in \(M[G][g]\) and 
has size  \(j(\kappa)\) there.
Since \(M\) was an ultrapower, \(M[G][g]\) has at most \(\bigl|{2^{j(\kappa)}}\bigr|\leq 
(2^\kappa)^\theta
=\theta^+\) many subsets of \(\P_{\text{tail}}\), and so we can diagonalize against all of
them to produce in \(V[G][g]\) an \(M[G][g]\)-generic \(G_{\text{tail}}\subseteq
\P_{\text{tail}}\) and lift \(j\) to \(j\colon V[G]\to M[j(G)]\), where 
\(j(G)=G*g*G_{\text{tail}}\).

Since \(M[j(G)]\) is still an ultrapower and thus closed under \(\theta\)-sequences in 
\(V[G][g]\),
we get \(j[g]\in M[j(G)]\). Since \(j(\Q_\kappa)\) is \(\leq\theta\)-directed closed in
\(M[j(G)]\) it has \(q=\bigcup j[g]\) as a condition. Since \(M[j(G)]\) has both 
the sequence of targets \(\vec{a}\) and \(j\rest 2^\kappa\), we can further extend 
\(q\) to \(q^*\) by coding \(a_\alpha\) at \(\kappa\) in \(q(j(\alpha))\)
for each \(\alpha<2^\kappa\). We again diagonalize against the
\(\bigl|{2^{2^{j(\kappa)}}}\bigr|\leq \theta^+\) many dense subsets of \(j(\Q_\kappa)\) 
in \(M[j(G)]\) below the master condition \(q^*\) to get a \(M[j(G)]\)-generic \(g^*\subseteq
j(\Q_\kappa)\) and lift \(j\) to \(j\colon V[G][g]\to M[j(G)][g^*]\). Finally, observe
that we have arranged the construction of \(g^*\) in such a way that 
\(g^*(j(\alpha))\) codes \(a_\alpha\) at \(\kappa\) 
for all \(\alpha<2^\kappa\) and, by definition,
this implies that \(j(\ell_\alpha)(\kappa)=a_\alpha\) for all \(\alpha<2^\kappa\). Thus we 
indeed have a joint Laver sequence for \(\kappa\) of length \(2^\kappa\) in \(V[G][g]\).

It remains for us to consider the second case, when \(\kappa\leq\theta<2^\kappa\). In this
situation our assumptions on \(\theta\) imply that \(2^\kappa=\theta^+\).
The poset \(j(\P_\kappa)\) factors as 
\(j(\P_\kappa)=\P_\kappa*\widetilde{\Q}_\kappa*\P_{\text{tail}}\), where 
\(\widetilde{\Q}_\kappa=\Add(\kappa,(2^\kappa)^{M[G]})\) is isomorphic but not necessarily
equal to \(\Q_\kappa\). Nevertheless, the same argument as before allows us to lift \(j\)
to \(j\colon V[G]\to M[j(G)]\) where \(j(G)=G*\widetilde{g}*G_{\text{tail}}\) and 
\(\widetilde{g}\) is the isomorphic image of \(g\).

We seem to hit a snag with the final lift through the forcing \(\Q_\kappa\), 
which has size \(2^\kappa\) and thus resists the usual approach of lifting via a master condition,
since this condition would simply be too big for the amount of closure we have. We salvage the
argument by using a technique, originally due to 
Magidor~\cite{Magidor1979:NonregularUltrafiltersCardinalityOfUltrapowers}, sometimes known as 
the ``master filter argument''.

The forcing \(j(\Q_\kappa)=\Add(j(\kappa),2^{j(\kappa)})\) has size \(2^{j(\kappa)}\) and
is \(\leq\theta\)-directed closed and \(j(\kappa)^+\)-cc in \(M[j(G)]\). Since \(M[j(G)]\)
is still an ultrapower, \(|2^{j(\kappa)}|\leq \theta^+=2^\kappa\) and so \(M[j(G)]\) has at most
\(2^\kappa\) many maximal antichains of \(j(\Q_\kappa)\). Let these be given in the
sequence \(\langle Z_\alpha;\alpha<2^\kappa\rangle\). Since each \(Z_\alpha\) has size at most
\(j(\kappa)\), it is in fact contained in some bounded part of the poset \(j(\Q_\kappa)\). 
Furthermore (and crucially), since \(j\) is an ultrapower by a measure on \(\mathcal{P}_\kappa
\theta\), it is continuous at \(2^\kappa=\theta^+\) and so there is for each \(\alpha\)
a \(\beta_\alpha<2^\kappa\) such that \(Z_\alpha\subseteq \Add(j(\kappa),2^{j(\kappa)})\rest
j(\beta_\alpha)\).
In particular, each \(Z_\alpha\) is a maximal antichain in \(\Add(j(\kappa),2^{j(\kappa)})\rest
j(\beta_\alpha)\). We will now construct in \(V[G][g]\) a descending sequence of conditions,
deciding more and more of the antichains \(Z_\alpha\), which will generate a filter,
the ``master filter'', that will allow us to lift \(j\) to \(V[G][g]\) and also (lest we forget)
witness the joint guessing property. 
%
%
%
We begin by defining the first condition \(q_0\). Consider the generic \(g\) up to
\(\beta_0\). This piece has size \(\theta\) and so \(\bigcup j[g\rest\beta_0]\) is a
condition in \(j(\Q_\kappa)\rest j(\beta_0)\). Let \(q_0'\) be the extension of
\(\bigcup j[g\rest\beta_0]\) which codes the target \(a_\alpha\) at \(\kappa\)
in \(q_0'(j(\alpha))\) for each \(\alpha<\beta_0\). This is still a condition
in \(j(\Q_\kappa)\rest j(\beta_0)\) and we can finally let \(q_0\) be any extension
of \(q_0'\) in this poset which decides the maximal antichain \(Z_0\).
Note that \(q_0\) is compatible with every condition in \(j[g]\), since we extended the
partial master condition \(\bigcup j[g\rest\beta_0]\) and made no commitments outside
\(j(\Q_\kappa)\rest j(\beta_0)\). We continue in this way inductively, constructing a descending
sequence of conditions \(q_\alpha\) for \(\alpha<\theta^+\), using the \(\theta\)-closure of
\(j(\Q_\kappa)\) and \(M[j(G)]\) to pass through limit stages. Now consider the filter
\(g^*\) generated by the conditions \(q_\alpha\). It is \(M[j(G)]\)-generic by construction
and also extends (or can easily be made to extend) \(j[g]\). We can thus lift \(j\) to
\(j\colon V[G][g]\to M[j(G)][g^*]\) and, since \(\P\) is \(\theta^+\)-cc and both 
\(G_{\text{tail}}\)
and \(g^*\) were constructed in \(V[G][g]\), the model \(M[j(G)][g^*]\) is closed under \(\theta\)-sequences
which shows that \(\kappa\) remains \(\theta\)-supercompact in \(V[G][g]\). Finally, as in
the previous case, \(g^*\) was constructed in such a way that \(j(\ell_\alpha)(\kappa)=a_\alpha\)
for all \(\alpha<2^\kappa\), verifying that these functions really do form a
joint Laver sequence for \(\kappa\).
\end{proof}

As a special case of theorem~\ref{thm:ForceLongJLDiamondSC} we can deduce the corresponding
result for measurable cardinals.

\begin{corollary}
\label{cor:ForceLongJLDiamondMeas}
If \(\kappa\) is measurable then there is a forcing extension in which there is a 
measurable joint Laver sequence for \(\kappa\) of length \(2^\kappa\).
\end{corollary}

It follows from the results of Hamkins~\cite{Hamkins2003:ApproximationAndCoveringNoNewLargeCardinals} that the forcing
\(\P\) from theorem~\ref{thm:ForceLongJLDiamondSC} does not create any measurable or
(partially) supercompact cardinals below \(\kappa\), since it admits a very low gap.
We could therefore have started with the least large cardinal \(\kappa\)
of interest and preserved its minimality through the construction.

\begin{corollary}
\label{cor:LeastSCCanHaveLongJLD}
If \(\kappa\) is the least \(\theta\)-supercompact cardinal then there
is a forcing extension where \(\kappa\) remains the least \(\theta\)-supercompact cardinal
and \(\jLdthetasc_{\kappa,2^\kappa}\) holds.
\end{corollary}

It is perhaps interesting to observe the peculiar arrangement of cardinal
arithmetic in the model produced in the above proof. We have
\(2^\theta=\theta^+\) and, if \(\theta\leq 2^\kappa\), also 
\(2^\kappa=\theta^+\). In particular, we never produced a \(\theta\)-supercompactness
joint Laver sequence of length greater than \(\theta^+\) (assuming here, of course,
that \(\theta=\theta^{<\kappa}\) is the optimal degree of supercompactness).
One has to wonder whether this is significant. Certainly the existence of
long joint Laver sequences does not imply much about cardinal arithmetic, since,
for example, if \(\kappa\) is indestructibly supercompact, we can manipulate the
value of \(2^\kappa\) freely, while maintaining the existence of a supercompactness
joint Laver sequence of length \(2^\kappa\).
On the other hand, even in the case of measurable \(\kappa\), the consistency 
strength of \(2^\kappa>\kappa^+\) is
known to exceed that of \(\kappa\) being measurable. The following question is therefore
natural:
\begin{question}[open]
If \(\kappa\) is \(\theta\)-supercompact and \(2^\kappa>\theta^+\), is there a forcing
extension preserving these facts in which there is a 
joint Laver sequence for \(\kappa\) of length \(2^\kappa\)?
\end{question}

We next show that the existence of joint Laver sequences is preserved under mild 
forcing. This will be
useful later when we separate different lengths of these sequences.

\begin{lemma}
\label{lemma:JLDiamondSCPreservation}
Let \(\kappa\) be \(\theta\)-supercompact and suppose that \(\jLdthetasc_{\kappa,\lambda}\) holds. 
Let \(\P\) be a poset such that either
\begin{enumerate}
\item \(\lambda\geq \theta^{<\kappa}\) and \(\P\) is \(\leq\lambda\)-distributive, or 
\item \(|\P|\leq \kappa\) and many \(\theta\)-supercompactness embeddings
with critical point \(\kappa\) lift through \(\P\).
\end{enumerate}
Then forcing with \(\P\) preserves \(\jLdthetasc_{\kappa,\lambda}\).
\end{lemma}

We were intentionally vague in item (2) of the lemma. The hypotheses will definitely
be satisfied if every \(\theta\)-supercompactness embedding should lift through \(\P\),
as is very often the case with forcing iterations up to the large cardinal \(\kappa\).
The proof, however, will show that it is in fact sufficient that at least one ground model
embedding associated to each target of the existing joint Laver sequence lift
through \(\P\). Furthermore, while the restriction \(|\P|\leq\kappa\) in (2) is
necessary for full generality, it can in fact be relaxed to \(|\P|\leq\theta\) for
a large class of forcings.

\begin{proof}
Under the hypotheses of (1) every ultrapower embedding by a measure on 
\(\mathcal{P}_\kappa\theta\) lifts to the extension by \(\P\) and no elements
of \(H_{\theta^+}\) or \(\lambda\)-sequences of such are added, so any ground model
joint Laver sequence of length \(\lambda\) retains this property in the extension.

Now suppose that the hypotheses of (2) hold, let \(\langle \ell_\alpha;\alpha<\lambda\rangle\)
be a joint Laver sequence for \(\kappa\) and let \(G\subseteq \P\) be generic. We may
also assume that the universe of \(\P\) is a subset of \(\kappa\).
Define functions \(\ell'_\alpha\colon \kappa\to V_\kappa[G]\) by
\(\ell'_\alpha(\xi)=\ell_\alpha(\xi)^{G\cap\xi}\) if this makes sense and \(\ell'_\alpha(\xi)=
\emptyset\) otherwise. We claim that \(\langle \ell'_\alpha;\alpha<\lambda\rangle\) 
is a joint Laver sequence in \(V[G]\).

Let \(\vec{a}\in V[G]\) be a \(\lambda\)-sequence of targets in \(H_{\theta^+}\).
Since \(\P\) is small we can find names \(\dot{a}_\alpha\) for these targets
in \(H_{\theta^+}\) and the sequence of these names is in \(V\).
Let \(j\colon V\to M\) be a \(\theta\)-supercompactness embedding with critical point 
\(\kappa\)
which lifts through \(\P\) and which satisfies \(j(\ell_\alpha)(\kappa)=\dot{a}_\alpha\)
for each \(\alpha\). It then follows that \(j(\ell'_\alpha)(\kappa)=a_\alpha\) in \(V[G]\),
verifying the joint Laver diamond property there.
\end{proof}

\subsection{Separating joint Laver diamonds by length}

We next aim to show that it is consistent that there is a measurable Laver function 
for \(\kappa\) but no joint Laver sequences of length \(\kappa^+\).
The following proposition expresses the key observation
for our solution, connecting the question to the number of normal measures problem.

\begin{proposition}
\label{prop:JLDiamondHasManyMeasures}
If there is a \(\theta\)-supercompactness joint Laver sequence for \(\kappa\)
of length \(\lambda\) then there
are at least \(2^{\theta\cdot \lambda}\) many normal measures on \(\mathcal{P}_\kappa\theta\).
\end{proposition}
\begin{proof}
The point is that any normal measure on \(\mathcal{P}_\kappa\theta\) realizes a single
\(\lambda\)-sequence of elements of \(H_{\theta^+}\) via the joint Laver sequence
and there are \(2^{\theta\cdot\lambda}\) many such sequences of targets.
\end{proof}

\begin{theorem}
\label{thm:SeparateShortLongJLDiamondMeas}
If \(\kappa\) is measurable then there is a forcing extension in which
there is a measurable Laver function for \(\kappa\) but no measurable joint Laver sequences
of length \(\kappa^+\).
\end{theorem}

\begin{proof}
After forcing as in the proof of theorem~\ref{thm:ForceLongJLDiamondSC}, 
if necessary, we may assume that \(\kappa\) has a Laver
function. A result of Apter--Cummings--Hamkins \cite{ApterCummingsHamkins2007:LargeCardinalsFewMeasures} then shows that
\(\kappa\) still carries a Laver function in the extension by
\(\P=\Add(\omega,1)*\Coll\bigl(\kappa^+,2^{2^\kappa}\bigr)\) 
but only carries \(\kappa^+\) many normal measures there.
Proposition~\ref{prop:JLDiamondHasManyMeasures} now implies that
there cannot be a joint Laver sequence of length \(\kappa^+\) in the extension.
\end{proof}

We can push this result a bit further to get a separation between any two desired lengths
of joint Laver sequences. To state the sharpest result we need to introduce a new
notion.

\begin{definition}
Let \(\kappa\) be a large cardinal supporting a notion of Laver diamond and
\(\lambda\) a cardinal. We say that a sequence
\(\vec{\ell}=\langle \ell_\alpha;\alpha<\lambda\rangle\)
is an \emph{almost-joint Laver sequence} if
\(\vec{\ell}\rest\gamma\) is a joint Laver sequence for any \(\gamma<\lambda\).
We say that \(\jLd_{\kappa,<\lambda}\) holds if there is an almost-joint Laver sequence of 
length \(\lambda\).
\end{definition}

\begin{theorem}
\label{thm:SeparateBetterShortLongJLDiamondMeas}
Suppose \(\kappa\) is measurable and let \(\lambda\) be a regular cardinal satisfying
\(\kappa<\lambda\leq 2^\kappa\).
If\/ \(\jLdmeas_{\kappa,<\lambda}\) holds then there is a forcing extension
preserving this in which \(\jLdmeas_{\kappa,\lambda}\) fails.
\end{theorem}

\begin{proof}
We imitate the proof of theorem~\ref{thm:SeparateShortLongJLDiamondMeas} but force
instead with \(\P=\Add(\omega,1)*\Coll\bigl(\lambda,2^{2^\kappa}\bigr)\).
The analysis based on \cite{ApterCummingsHamkins2007:LargeCardinalsFewMeasures} now
shows that the final extension has at most \(\lambda\) many normal measures on \(\kappa\)
and thus there can be no joint Laver sequences of length \(\lambda\) there by 
proposition~\ref{prop:JLDiamondHasManyMeasures}. That \(\jLdmeas_{\kappa,<\lambda}\) still
holds
follows from (the proof of) lemma~\ref{lemma:JLDiamondSCPreservation}: 
part (2) implies that, by guessing names, the \(\jLdmeas_{\kappa,<\lambda}\)-sequence
from the ground model gives rise to one in the intermediate Cohen extension.
Part (1) then shows that each of the initial segments of this sequence remains
a joint Laver sequence in the final extension.
\end{proof}

We can also extend these results to \(\theta\)-supercompact cardinals without too much
effort.

\begin{theorem}
\label{thm:SeparateShortLongJLDiamondThetaSC}
If \(\kappa\) is \(\theta\)-supercompact, \(\theta\) is regular, and 
\(\theta^{<\kappa}=\theta\) then there is a forcing extension in
which \(\Ldthetasc_\kappa\) holds but \(\jLdthetasc_{\kappa,\theta^+}\) fails.
\end{theorem}

Of course, the theorem is only interesting when \(\kappa\leq\theta<2^\kappa\), in which
case the given separation is best possible in view of 
proposition~\ref{prop:SCHasSomeJLDiamond}.

\begin{proof}
We may assume by prior forcing, as in theorem~\ref{thm:ForceLongJLDiamondSC}, that we have a 
Laver function for \(\kappa\).
We now force with
\(\P=\Add(\omega,1)*\Coll\bigl(\theta^+,2^{2^\theta}\bigr)\) to get an extension
\(V[g][G]\). By the results of \cite{ApterCummingsHamkins2007:LargeCardinalsFewMeasures}
the extension \(V[g][G]\) has at most \(\theta^+\) many normal measures on \(\mathcal{P}_\kappa\theta\)
and therefore there are no joint Laver sequences for \(\kappa\) of length \(\theta^+\)
there by proposition~\ref{prop:JLDiamondHasManyMeasures}. It remains to see that
there is a Laver function in \(V[g][G]\). Let
\(\ell\) be a Laver function in \(V\) and define
\(\ell'\in V[g][G]\) by \(\ell'(\xi)=\ell(\xi)^g\) if \(\ell(\xi)\) is an \(\Add(\omega,1)\)-%
name and \(\ell'(\xi)=\emptyset\) otherwise. For a given \(a\in H_{\theta^+}^{V[g][G]}=
H_{\theta^+}^{V[g]}\) we can select an \(\Add(\omega,1)\)-name \(\dot{a}\in H_{\kappa^+}^V\)
and find a \(\theta\)-supercompactness embedding \(j\colon V\to M\) such that
\(j(\ell)(\kappa)=\dot{a}\). The embedding \(j\) lifts to \(j\colon V[g][G]\to M[g][j(G)]\)
since the Cohen forcing was small and the collapse forcing was \(\leq\theta\)-closed.
But then clearly \(j(\ell')(\kappa)=\dot{a}^g=a\), so \(\ell'\) is a Laver function.
\end{proof}

\begin{theorem}
Suppose \(\kappa\) is \(\theta\)-supercompact and let \(\lambda\) be a regular cardinal
satisfying \(\theta^{<\kappa}<\lambda\leq 2^\kappa\).
If\/ \(\jLdthetasc_{\kappa,<\lambda}\) holds then there is a forcing extension
preserving this in which \(\jLdthetasc_{\kappa,\lambda}\) fails.
\end{theorem}

\begin{proof}
The relevant forcing is \(\Add(\omega,1)*\Coll\bigl(\lambda,2^{2^{\theta^{<\kappa}}}\bigr)\).
Essentially the argument from theorem~\ref{thm:SeparateBetterShortLongJLDiamondMeas}
then finishes the proof.
\end{proof}

A question remains about the principles \(\jLd_{\kappa,<\lambda}\), whether they are
genuinely new or whether they reduce to other principles.

\begin{question}[open]
Let \(\kappa\) be \(\theta\)-supercompact and \(\lambda\leq 2^\kappa\). 
Is \(\jLdthetasc_{\kappa,<\lambda}\) equivalent
to \(\jLdthetasc_{\kappa,\gamma}\) holding for all \(\gamma<\lambda\)?
\end{question}

An almost-joint Laver sequence definitely gives instances of
joint Laver diamonds at each particular \(\gamma\).
The reverse implication is particularly interesting in the case when \(\lambda=\mu^+\) is
a successor cardinal. This is because simply rearranging the functions in a joint
Laver sequence of length \(\mu\) gives joint Laver sequences of any length
shorter than \(\mu^+\). The question is thus asking whether \(\jLd_{\kappa,\mu}\)
suffices for \(\jLd_{\kappa,<\mu^+}\).
The restriction to \(\lambda\leq 2^\kappa\) is necessary to avoid the following
triviality.

\begin{proposition}
\(\jLdthetasc_{\kappa,<(2^\kappa)^+}\) fails for every cardinal \(\kappa\)
\end{proposition}

\begin{proof}
Any potential \(\jLdthetasc_{\kappa,<(2^\kappa)^+}\)-sequence must necessarily
have the same function appear on at least two coordinates. But then any initial segment
of this sequence cannot be joint, since it cannot guess distinct targets on those two
coordinates.
\end{proof}

An annoying feature of the models produced in the preceding theorems is that in
all of them
the least \(\lambda\) for which there is no joint Laver sequence of length 
\(\lambda\) is \(\lambda=2^\kappa\). One has to wonder whether this is significant.

\begin{question}
\label{ques:FM}
Is it relatively consistent that there is a \(\theta\)-supercompact cardinal \(\kappa\),
for some \(\theta\), such that \(\Ldthetasc_{\kappa}\) holds
and \(\jLdthetasc_{\kappa,\lambda}\) fails for some \(\lambda\) satisfying 
\(\lambda<2^\kappa\)?
\end{question}

To satisfy the listed conditions, GCH must fail at \(\kappa\) (since we must have
\(\kappa<\lambda<2^\kappa\) by proposition~\ref{prop:SCHasSomeJLDiamond}).
We can therefore expect that achieving the situation described in the question will require
some additional consistency strength. 

In the case of a measurable \(\kappa\) the answer to the question is positive:
we will show in theorem~\ref{thm:FMDoesntHaveLongjLd} that, starting from sufficient
large cardinal hypotheses, we can produce a model where \(\kappa\) is measurable and
has a measurable Laver function but no joint Laver sequences of length \(\kappa^+<2^\kappa\).
The proof relies on an argument
due to Friedman and Magidor~\cite{FriedmanMagidor2009:NumberNormalMeasures}
which facilitates the simultaneous control of the number of measures
at \(\kappa\) and the value of the continuum function at \(\kappa\) and \(\kappa^+\).

Let us briefly give a general setup for the argument 
of~\cite{FriedmanMagidor2009:NumberNormalMeasures} that will allow us to carry out
our intended modifications without repeating too much of the work done there.

Fix a cardinal \(\kappa\) and suppose GCH holds up to and including \(\kappa\).
Furthermore suppose that \(\kappa\) is the critical point of an
elementary embedding \(j\colon V\to M\) satisfying the following properties:
\begin{itemize}
\item \(j\) is an extender embedding, i.e.\ every element of \(M\) has the form
\(j(f)(\alpha)\) for some function \(f\) defined on \(\kappa\) and some \(\alpha<j(\kappa)\);

\item \((\kappa^{++})^M=\kappa^{++}\);

\item there is a function \(f\colon\kappa\to V\), such that \(j(f)(\kappa)\) is, in
\(V\) (and therefore also in \(M\)), a sequence of \(\kappa^{++}\) many disjoint
stationary subsets of \(\kappa^{++}\cap\operatorname{Cof}_{\kappa^+}\).
\end{itemize}

Given this arrangement, Friedman and Magidor define a forcing iteration \(\P\) of length
\(\kappa+1\) (with carefully chosen support) 
which forces at each inaccessible stage \(\gamma\leq\kappa\) with
\(\operatorname{Sacks}^*(\gamma,\gamma^{++})*\operatorname{Sacks}_{\mathrm{id}^{++}}(\gamma)
*\operatorname{Code}(\gamma)\). Here the conditions in 
\(\operatorname{Sacks}_{\mathrm{id}^{++}}\) are perfect trees of height \(\gamma\), splitting
on a club, and in which every splitting node of height \(\delta\) has \(\delta^{++}\)
many successors; furthermore \(\operatorname{Sacks}^*(\gamma,\gamma^{++})\) is a
large product of versions of \(\operatorname{Sacks}(\gamma)\) where splitting is restricted
to a club of singular cardinals and \(\operatorname{Code}(\gamma)\) is a certain
\(\leq\gamma\)-distributive notion of forcing coding information about the stage
\(\gamma\) generics into the stationary sets given by \(f(\gamma)\).

Let \(G\subseteq\P\) be generic. 
In the interest of avoiding a full analysis of the forcing notion \(\P\),
we list some of the properties of the extension \(V[G]\) that we will use
axiomatically to derive the subsequent results (the interested reader is encouraged
to see~\cite{FriedmanMagidor2009:NumberNormalMeasures} for proofs):

\begin{enumerate}[label=(\roman*)]
\item \label{FMlist:PreservesInacc}
\(\P\) preserves cardinals and cofinalities, and increases the values of
the continuum function by at most two cardinal steps. 
In particular, the inaccessible cardinals
of \(V\) remain inaccessible in \(V[G]\);

\item \label{FMlist:GCH}
We have \(2^\kappa=\kappa^{++}\) in \(V[G]\);

\item \label{FMlist:SacksProperty}
\(\P\) has the \(\kappa\)-Sacks property: for any function 
\(f\colon \kappa\to\mathrm{Ord}\) in \(V[G]\) there is a function \(h\in V\) such that
\(f(\alpha)\in h(\alpha)\) for all \(\alpha\) and \(|h(\alpha)|\leq\alpha^{++}\);

\item \label{FMlist:SelfEncoding}
The generic \(G\) is self-encoding in a strong way: in \(V[G]\) there is a unique 
\(M\)-generic for \(j(\P)_{<j(\kappa)}\) extending \(j[G_{<\kappa}]\);

\item \label{FMlist:TuningFork}
If \(S_\kappa\) is the generic added by 
\(\operatorname{Sacks}_{\mathrm{id}^{++}}(\kappa)\) within \(\P\), 
then \(\bigcap j[S_\kappa]\) is a \emph{tuning fork}: the union of \(\kappa^{++}\) many 
branches, all of which split off exactly at level \(\kappa\) and all of which are
generic over \(j(G_{<\kappa})\);

\item \label{FMlist:Lifting}
In \(V[G]\) there are exactly \(\kappa^{++}\) many \(M\)-generics for
\(j(\P)\) extending \(j[G]\), corresponding to the \(\kappa^{++}\) many branches in
\(\bigcap j[S_\kappa]\). In particular, there are exactly \(\kappa^{++}\) many lifts
\(j_\alpha\) of \(j\) in \(V[G]\), distinguished by \(j_\alpha(S_\kappa)(\kappa)=\alpha\)
for \(\alpha<\kappa^{++}\).
\end{enumerate}

\begin{proposition}
\label{prop:FMHasLd}
In the above setup, the iteration \(\P\) adds a measurable Laver function for \(\kappa\).
\end{proposition}

\begin{proof}
Let \(G\subseteq\P\) be generic. As we stated in item~\ref{FMlist:PreservesInacc}, for any \(\alpha<\kappa^{++}\) there is a
lift \(j_\alpha\colon V[G]\to M[j_\alpha(G)]\) of \(j\) such that 
\(j_\alpha(S_\kappa)(\kappa)=\alpha\), where \(S_\kappa\) is the Sacks subset of \(\kappa\)
added by the \(\kappa\)th stage of \(G\). This shows that 
\(\bar{\ell}(\gamma)=S_\kappa(\gamma)\) is a \(\kappa^{++}\)-guessing measurable Laver
function for \(\kappa\).

Note that all of the subsets of \(\kappa\) in \(M[j_\alpha(G)]\) (and \(V[G]\)) 
appear already
in \(M[G]\). Let \(\vec{e}=\langle e_\alpha;\alpha<\kappa^{++}\rangle\) be
an enumeration of \(H_{\kappa^+}^{M[G]}\) in \(M[G]\) and let \(\dot{e}\in M\) be
a name for \(\vec{e}\). We can write \(\dot{e}=j(F)(\kappa)\) for some function \(F\),
defined on \(\kappa\). Now define a function \(\ell\colon\kappa\to V_\kappa\) in \(V[G]\) by
\(\ell(\gamma)=(F(\gamma)^G)(\bar{\ell}(\gamma))\). This is, in fact, our desired Laver
function; given an arbitrary element of \(H_{\kappa^+}^{V[G]}\), we can find it in the 
enumeration \(\vec{e}\). If \(\alpha\) is its index, then
\[
j_\alpha(\ell)(\kappa) =
\bigl(j_\alpha(F)(\kappa)^{j_\alpha(G)}\bigr) (j_\alpha(\bar{\ell})(\kappa)) =
\bigl(j(F)(\kappa)^{j_\alpha(G)}\bigr) (\alpha) =
\vec{e}(\alpha)= e_\alpha
\qedhere
\]
\end{proof}

\begin{theorem}
\label{thm:FMDoesntHaveLongjLd}
Suppose \(\kappa\) is \((\kappa+2)\)-strong and assume that \(V=L[\vec{E}]\) is the
minimal extender model witnessing this. Then there is a forcing extension in which
\(2^\kappa=\kappa^{++}\), the cardinal \(\kappa\) remains measurable, 
\(\kappa\) carries a measurable Laver function but there are no measurable
joint Laver sequences for \(\kappa\) of length \(\kappa^+\).
\end{theorem}

This finally answers question~\ref{ques:FM} in the positive.

\begin{proof}
Let \(j\colon V\to M\) be the ultrapower embedding by the top extender of \(\vec{E}\),
the unique extender witnessing the \((\kappa+2)\)-strongness of \(\kappa\). In particular,
every element of \(M\) has the form \(j(f)(\alpha)\) for some \(\alpha<j(\kappa)\), and
\(M\) computes \(\kappa^{++}\) correctly. Furthermore, \(V\) has a canonical
\(\diamond_{\kappa^{++}}(\operatorname{Cof}_{\kappa^+})\)-sequence, which is definable over
\(H_{\kappa^{++}}\). Since \(H_{\kappa^{++}}\in M\), this same sequence is also in \(M\)
and, by definability, is of the form \(j(\bar{f})(\kappa)\) for some function \(\bar{f}\).
By having this diamond sequence guess the singletons \(\{\xi\}\) for \(\xi<\kappa^{++}\),
we obtain a sequence of \(\kappa^{++}\) many disjoint stationary subsets of
\(\kappa^{++}\cap\operatorname{Cof}_{\kappa^+}\), and this sequence itself has the form
\(j(f)(\kappa)\) for some function \(f\). We are therefore in a situation where
the definition of the Friedman--Magidor iteration we described above makes sense.
But first, we need to carry out some preliminary forcing.

Let \(g\subseteq\Add(\kappa^+,\kappa^{+3})\) be generic. Since this Cohen poset is
\(\leq\kappa\)-distributive, the embedding \(j\) lifts (uniquely) to an embedding
\(j\colon V[g]\to M[j(g)]\).\footnote{The lifted embedding will not be
a \((\kappa+2)\)-strongness embedding and, in fact, \(\kappa\) is no longer 
\((\kappa+2)\)-strong in \(V[g]\). Nevertheless, the residue of strongness will suffice
for our argument.}
Let us examine the lifted embedding \(j\). It is still an extender embedding.
Additionally, since GCH holds in \(V\), the forcing \(\Add(\kappa^+,\kappa^{+3})\)
preserves cardinals, cofinalities, and stationary subsets of \(\kappa^{++}\).
Together this means that \(M[j(g)]\) computes \(\kappa^{++}\) correctly and the stationary
sets given by the sequence \(j(f)(\kappa)\) above remain stationary. Therefore we may
still define the Friedman--Magidor iteration \(\P\) over \(V[g]\).

Let \(G\subseteq\P\) be generic over \(V[g]\). We claim that \(V[g][G]\) is the model
we want. We have \(2^\kappa=\kappa^{++}\) in the extension, by item~\ref{FMlist:GCH}
of our list,
and proposition~\ref{prop:FMHasLd} implies that \(\kappa\) is
measurable in \(V[g][G]\) and \(\Ldmeas_\kappa\) holds there. So it remains for us to see
that \(\jLdmeas_{\kappa,\kappa^+}\) fails. By proposition~\ref{prop:JLDiamondHasManyMeasures}
it suffices to show that \(\kappa\) does not carry \(2^{\kappa^+}=\kappa^{+3}\) many
normal measures in \(V[g][G]\).

Let \(U^*\in V[g][G]\) be a normal measure on \(\kappa\) and let
\(j^*\colon V[g][G]\to N[j^*(g)][j^*(G)]\) be its associated ultrapower embedding.
This embedding restricts to \(j^*\colon V\to N\). Since \(V\) is the core model
from the point of view of \(V[g][G]\), the 
embedding \(j^*\) arises as the ultrapower map associated to a normal iteration
of extenders on the sequence \(\vec{E}\).

We first claim that the first extender applied in this iteration is the top extender of
\(\vec{E}\). Let us write \(j^*=j_1\circ j_0\), where \(j_0\colon V\to N_0\) results from the
first applied extender. Clearly \(j_0\) has critical point \(\kappa\). 
Now suppose first that \(j_0(\kappa)<\kappa^{++}\). Of course,
\(j_0(\kappa)\) is inaccessible in \(N_0\) and, since \(N\) is an inner model of \(N_0\),
also in \(N\). But \(j_0(\kappa)\) is not inaccessible in \(N[j^*(g)][j^*(G)]\), since
\(2^\kappa=\kappa^{++}\) there. This is a contradiction, since passing from \(N\)
to \(N[j^*(g)][j^*(G)]\) preserves inaccessibility, by item~\ref{FMlist:PreservesInacc}
of our list.

It follows that we must have \(j_0(\kappa)\geq\kappa^{++}\). We will argue that
the extender \(E\) applied to get \(j_0\) witnesses the \((\kappa+2)\)-strongness of 
\(\kappa\), so it must be the top extender of \(\vec{E}\) and \(j_0=j\). 
Using a suitable indexing of
\(\vec{E}\), the extender \(E\) has index \((j_0(\kappa)^+)^{N_0}>\kappa^{++}\),
and the coherence of the extender sequence implies that the sequences in \(V\) and in
\(N_0\) agree up to \(\kappa^{++}\). By the acceptability of these extender models
it now follows that \(H_{\kappa^{++}}^V=H_{\kappa^{++}}^{N_0}\) or, equivalently,
\(V_{\kappa+2}\in N_0\).

Finally, we claim that the iteration giving rise to \(j^*\) ends after one step, meaning
that \(j^*=j\). Suppose to the contrary that \(j_1\) is nontrivial. By the normality of
the iteration, the critical point of \(j_1\) must be some \(\lambda>\kappa\).
We can find a function \(h\in V[g][G]\), defined on \(\kappa\), such that
\(j^*(h)(\kappa)=\lambda\), since \(j^*\) is given by a measure ultrapower of
\(V[g][G]\). By the \(\kappa\)-Sacks property of \(\P\) (see item \ref{FMlist:SacksProperty}) 
we can cover the function
\(h\) by a function \(\bar{h}\in V[g]\); in fact, since the forcing to add \(g\)
was \(\leq\kappa\)-closed, we have \(\bar{h}\in V\).
Now
\[
\lambda=j^*(h)(\kappa)\in j^*(\bar{h})(\kappa) = j_1(j(\bar{h}))(j_1(\kappa))
= j_1(j(\bar{h})(\kappa))
\]
and \(A=j(\bar{h})(\kappa)\) has cardinality at most \(\kappa^{++}\) in \(M\).
In particular, since \(\kappa^{++}<\lambda\), we have \(\lambda\in j_1(A)=j_1[A]\),
which is a contradiction, since \(\lambda\) was the critical point of \(j_1\).

We can conclude that any embedding \(j^*\) arising from a normal measure on \(\kappa\)
in \(V[g][G]\) is a lift of the ground model \((\kappa+2)\)-strongness embedding \(j\).
But there are exactly \(\kappa^{++}\) many such lifts: the lift to \(V[g]\) is unique,
and there are \(\kappa^{++}\) many possibilities for the final lift to \(V[g][G]\),
according to item~\ref{FMlist:Lifting}.
Therefore there are only \(\kappa^{++}\) many normal measures on \(\kappa\)
in \(V[g][G]\).
\end{proof}

Ben Neria and Gitik have recently announced that the consistency strength required to
achieve the failure of GCH at a measurable cardinal carrying a unique normal measure
is exactly that of a measurable cardinal \(\kappa\) with \(o(\kappa)=\kappa^{++}\)
(see~\cite{BenNeriaGitik:UniqueMeasureWithoutGCH}).
Their method is flexible enough to allow us to incorporate it into our proof of
theorem~\ref{thm:FMDoesntHaveLongjLd}, reducing the consistency strength hypothesis required there from a \((\kappa+2)\)-strong cardinal \(\kappa\) to just \(o(\kappa)=\kappa^{++}\).
We have chosen to present the proof based on the original Friedman--Magidor argument
since it avoids some complications arising from using the optimal hypothesis.

\subsection{(Joint) Laver diamonds and the number of normal measures}

The only method of controlling the existence of (joint) Laver diamonds we have seen
is by controlling the number of large cardinal measures, relying on the rough bound
given by proposition~\ref{prop:JLDiamondHasManyMeasures}. One has to wonder whether
merely the existence of sufficiently many measures guarantees the existence of
(joint) Laver diamonds. We focus on the simplest form of the question, concerning
measurable cardinals.

\begin{question}[open]
\label{q:MeasuresGiveLavers}
Suppose \(\kappa\) is measurable and there are at least \(2^\lambda\) many normal measures
on \(\kappa\) for some \(\lambda\geq\kappa\). 
Does there exist a measurable joint Laver sequence for \(\kappa\) of length \(\lambda\)?
\end{question}

As the special case when \(\lambda=\kappa\), the question includes the possibility
that having \(2^\kappa\) many normal measures, the minimum required, suffices to give
the existence of a measurable Laver function for \(\kappa\).
Even in this very special case
it seems implausible that simply having enough measures
would automatically yield a Laver function. Nevertheless, in all of the examples
of models obtained by forcing and in which we have control over the number of measures
that we have seen, Laver functions have existed. 
On the other hand, Laver functions and joint Laver sequences also exist in
canonical inner models that have sufficiently many measures. These models carry
long Mitchell-increasing sequences of normal measures that we can use to obtain
ordinal-guessing Laver functions. We can then turn these into actual Laver functions
by exploiting the coherence and absoluteness of these models.

\begin{definition}
Let \(A\) be a set (or class) of ordinals and let \(\bar{\ell}\) be an \(A\)-guessing
Laver function for some large cardinal \(\kappa\). Let \(\triangleleft\) be some
well-order (one arising from an \(L\)-like inner model, for example). 
We say that \(\triangleleft\) is \emph{suitable} for \(\bar{\ell}\) if,
for any \(\alpha\in A\), there is an elementary embedding \(j\), witnessing the largeness of
\(\kappa\), such that \(\bar{\ell}\) guesses \(\alpha\) via \(j\) and \(j(\triangleleft)\rest 
(\alpha+1) = \triangleleft \rest (\alpha+1)\); that is, the well-orders \(j(\triangleleft)\)
and \(\triangleleft\) agree on their first \(\alpha+1\) many elements.

If \(\mathcal{J}\) is a class of elementary embeddings witnessing the largeness of 
\(\kappa\),
we say that \(\triangleleft\) is \emph{supersuitable for \(\mathcal{J}\)} if 
\(j(\triangleleft)\rest j(\kappa)=\triangleleft\rest j(\kappa)\) for any \(j\in\mathcal{J}\).
\end{definition}

We could, for example, take the class \(\mathcal{J}\) to consist of all ultrapower embeddings
by normal measures on \(\kappa\) or, more to the point, all ultrapower embeddings arising
from a fixed family of extenders. We should also note that, for the notion to make sense,
the order type of \(\triangleleft\) must be quite high: at least \(\sup A\) in the case
of well-orders suitable for an \(A\)-guessing Laver function and at least
\(\sup_{j\in\mathcal{J}} j(\kappa)\) for supersuitable well-orders (the latter would also
make sense if the order type of \(\triangleleft\) were smaller than \(\kappa\), but that
case is not of much interest).

Clearly any supersuitable well-order is suitable for any ordinal-guessing Laver function
\(\bar{\ell}\), provided that the class \(\mathcal{J}\) includes the embeddings via which
\(\bar{\ell}\) guesses its targets. The following obvious lemma describes the way in which 
suitable well-orders will be used to turn ordinal-guessing Laver functions into set-guessing 
ones.

\begin{lemma}
\label{lemma:SetGuessingFromOrdinalGuessing}
Let \(A\) be a set (or class) of ordinals and let \(\bar{\ell}\) be an
\(A\)-guessing Laver function for some large cardinal \(\kappa\). Let \(\triangleleft\)
be a well-order such that \(\operatorname{otp}(\triangleleft)\subseteq A\).
If \(\triangleleft\) is suitable for \(\bar{\ell}\), then there is a \(B\)-guessing Laver 
function for \(\kappa\), where \(B\) is the field of \(\triangleleft\).
\end{lemma}

\begin{proof}
We can define a \(B\)-guessing Laver function by simply letting
\(\ell(\xi)\) be the \(\bar{\ell}(\xi)\)th element of \(\triangleleft\). Then,
given a target \(b\in B\), we can find its index \(\alpha\) in the well-order 
\(\triangleleft\) and an embedding \(j\) such that \(j(\bar{\ell})(\kappa)=\alpha\).
Since \(\triangleleft\) is suitable for \(\bar{\ell}\), the orders \(\triangleleft\) and
\(j(\triangleleft)\) agree on their \(\alpha\)th element and so \(\ell\) guesses \(b\)
via \(j\).
\end{proof}

It follows from the above lemma that in any model with a sufficiently supersuitable
well-order, being able to guess ordinals suffices to be able to guess arbitrary sets.

\begin{lemma}
Let \(X\) be a set (or class) of ordinals and let \(\mathcal{J}\) be a class of elementary
embeddings of \(L[X]\) with critical point \(\kappa\) such that 
\(j(X)\cap j(\kappa) = X\cap j(\kappa)\) for any
\(j\in\mathcal{J}\).
Then \(\leq_X\), the canonical order of \(L[X]\), is supersuitable for \(\mathcal{J}\).
\end{lemma}

\begin{proof}
This is obvious; the order \(\leq_X\rest j(\kappa)\) is definable in \(L_{j(\kappa)}[X]\),
but by our coherence hypothesis this structure is just the same as \(L_{j(\kappa)}[j(X)]\).
\end{proof}

We are mostly interested in this lemma in the case when \(X=\vec{E}\) is an extender sequence
and \(L[\vec{E}]\) is an extender model in the sense of~\cite{Zeman2002:InnerModels}.
In particular, we want \(\vec{E}\) to be \emph{acceptable} (a technical condition which 
implies enough condensation properties in \(L[\vec{E}]\) to conclude
\(H_\lambda^{L[\vec{E}]}=L_\lambda[\vec{E}]\)), \emph{coherent} (meaning that if
\(j\colon L[\vec{E}]\to L[\vec{F}]\) is an ultrapower by the \(\alpha\)th extender of
\(\vec{E}\) then \(\vec{F}\rest (\alpha+1)= \vec{E}\rest\alpha\)), and to use Jensen indexing
(meaning that the index of an extender \(E\) on \(\vec{E}\) with critical point \(\kappa\) 
is \(j_E(\kappa)^+\), as computed in the ultrapower).

\begin{corollary}
\label{cor:SupersuitableOrderInExtenderModels}
Let \(V=L[\vec{E}]\) be an extender model.
Then the canonical well-order is supersuitable for the class of ultrapower embeddings
by the extenders on the sequence \(\vec{E}\).
\end{corollary}

\begin{proof}
This is immediate from the preceding lemma and the fact that our extender sequences are
coherent and use Jensen indexing.
\end{proof}

\begin{theorem}
\label{thm:LaverFunctionsInExtenderModels}
Let \(V=L[\vec{E}]\) be an extender model. Let \(\kappa\) be a
cardinal such that every normal measure on \(\kappa\) appears on the sequence \(\vec{E}\). 
If \(o(\kappa)\geq\kappa^+\) then \(\Ldmeas_\kappa\) holds.
Moreover, if \(o(\kappa)=\kappa^{++}\) then \(\jLdmeas_{\kappa,\kappa^+}\),
and even \(\Ldmeas_\kappa(H_{\kappa^{++}})\), holds.
\end{theorem}

In particular, the above theorem implies \(\Ldmeas_\kappa\) holds in the least inner model
with the required number of measures and the same holds for \(\jLdmeas_{\kappa,\kappa^+}\).
This provides further evidence that the answer to question~\ref{q:MeasuresGiveLavers}, which
remains open, might turn out to be positive

\begin{proof}
We can argue for the two cases more or less uniformly: let 
\(\lambda\in\{\kappa^+,\kappa^{++}\}\) such that \(\lambda\leq o(\kappa)\).
The function \(\bar{\ell}(\xi)=o(\xi)\) is a \(\lambda\)-guessing measurable Laver
function for \(\kappa\).  By the acceptability of
\(\vec{E}\) we have that \(H_{\lambda}=L_{\lambda}[\vec{E}]\). The canonical well-order
\(\leq_{\vec{E}}\cap L_\lambda[\vec{E}]\) has order type \(\lambda\) and, by
corollary~\ref{cor:SupersuitableOrderInExtenderModels}, is supersuitable for the class
of ultrapower embeddings by normal measures on \(\kappa\). It follows that
\(\leq_{\vec{E}}\) is suitable for \(\bar{\ell}\), so, by 
lemma~\ref{lemma:SetGuessingFromOrdinalGuessing}, there is an \(H_{\lambda}\)-guessing
measurable Laver function for \(\kappa\).

To finish the proof we still need to produce a joint measurable Laver sequence
for \(\kappa\), in the case that \(o(\kappa)=\kappa^{++}\). This is done in exactly the same
way as in proposition~\ref{prop:SCHasSomeJLDiamond}; one simply uses the 
\(H_{\kappa^{++}}\)-guessing Laver function to guess the whole sequence of targets
for a joint Laver sequence.
\end{proof}

We should also mention that, if we restrict to a smaller set of targets, having
enough normal measures \emph{does} give us Laver functions.

\begin{lemma}
\label{lemma:DiscreteMeasures}
Let \(\kappa\) be a regular cardinal and \(\gamma\leq\kappa\)
and suppose that \(\langle \mu_\alpha;\alpha<\gamma\rangle\) is a sequence of distinct normal
measures on \(\kappa\). Then there is a sequence \(\langle X_\alpha;\alpha<\gamma\rangle\)
of pairwise disjoint subsets of \(\kappa\) such that \(X_\alpha\in\mu_\beta\) if
and only if \(\alpha=\beta\).
\end{lemma}

\begin{proof}
We prove the lemma by induction on \(\gamma\). In the base step, \(\gamma=1\),
we simply observe that, since \(\mu_0\neq\mu_1\), we must have a set \(X_0\subseteq\kappa\)
such that \(X_0\in\mu_0\) and \(\kappa\setminus X_0\in \mu_1\).

The successor step proceeds similarly. 
Suppose that the lemma holds for sequences of length \(\gamma\) and fix a sequence
of measures \(\langle \mu_\alpha;\alpha<\gamma+1\rangle\). By the induction hypothesis
we can find pairwise disjoint sets \(\langle Y_\alpha;\alpha<\gamma\rangle\) such
that each \(Y_\alpha\) picks out a unique measure among those with indices below
\(\gamma\). Again, since \(\mu_\gamma\) is distinct from all of the other measures,
we can find sets \(Z_\alpha\in \mu_\gamma\setminus \mu_\alpha\) for each \(\alpha<\gamma\).
Then the sets \(X_\alpha=Y_\alpha\setminus Z_\alpha\) for \(\alpha<\gamma\) and
\(X_\gamma=\bigcap_\alpha Z_\alpha\) are as required.

In the limit step suppose that the lemma holds for all \(\delta<\gamma\).
We can then fix sequences \(\langle X_\alpha^\delta;\alpha<\delta\rangle\) for each 
\(\delta<\gamma\)
as above. The argument proceeds slightly differently depending on whether \(\gamma=\kappa\)
or not. If \(\gamma<\kappa\) we can simply let 
\(X_\alpha=\bigcap_{\alpha<\delta<\gamma} X_\alpha^\delta\in\mu_\alpha\). 
If, on the other hand, we have \(\gamma=\kappa\) then first let
\(Y_\alpha=\diag_{\alpha<\delta<\kappa}X_\alpha^\delta\in \mu_\alpha\).
Observe that the \(Y_\alpha\) are almost disjoint: \(Y_\alpha\cap Y_\beta\) is
bounded in \(\kappa\) for any \(\alpha,\beta<\kappa\).
Now consider 
\[X_\alpha=Y_\alpha\setminus \bigcup_{\beta<\alpha}(Y_\alpha\cap Y_\beta)\]
for \(\alpha<\kappa\). Since \(Y_\alpha\cap Y_\beta\) is bounded for all \(\beta<\alpha\),
we still have \(X_\alpha\in \mu_\alpha\). Furthermore, we obviously have
\(X_\alpha\cap X_\beta=\emptyset\) for \(\beta<\alpha\) and this implies that the
\(X_\alpha\) are pairwise disjoint.
\end{proof}

\begin{theorem}
\label{thm:OrdGuessingIffEnoughMeasures}
Let \(\kappa\) be a measurable cardinal and \(\gamma<\kappa^+\) an ordinal.
There is a \(\gamma\)-guessing measurable Laver function for \(\kappa\) if and only if
there are at least \(|\gamma|\) many normal measures on \(\kappa\).
\end{theorem}

\begin{proof}
First suppose that \(\Ldmeas_\kappa(\gamma)\) holds. Then, just as in
proposition~\ref{prop:JLDiamondHasManyMeasures}, each target \(\alpha<\gamma\)
requires its own embedding \(j\) via which it is guessed and this gives us
\(|\gamma|\) many distinct normal measures.

Conversely, suppose that we have at least \(|\gamma|\) many normal measures on
\(\kappa\). We can apply lemma~\ref{lemma:DiscreteMeasures} to find a sequence
of pairwise disjoint subsets distinguishing these measures. By reorganizing
the measures and the distinguishing sets we may assume that they are given in sequences
of length \(\gamma\). We now have normal measures \(\langle \mu_\alpha;\alpha<\gamma\rangle\)
and sets \(\langle X_\alpha;\alpha<\gamma\rangle\) such that \(\mu_\alpha\) is the unique
measure concentrating on \(X_\alpha\); we may even assume that the \(X_\alpha\) partition
\(\kappa\). 
Let \(f_\alpha\) for
\(\alpha<\gamma\) be the representing functions for \(\alpha\), that is,
\(j(f_\alpha)(\kappa)=\alpha\) for any ultrapower embedding \(j\) by a normal measure on
\(\kappa\). We can now define a \(\gamma\)-guessing Laver function \(\ell\) by letting
\(\ell(\xi)=f_\alpha(\xi)\) where \(\alpha\) is the unique index such that 
\(\xi\in X_\alpha\). This function indeed guesses any target \(\alpha<\gamma\):
simply let \(j\colon V\to M\) be the ultrapower by \(\mu_\alpha\).
Since \(\mu_\alpha\) concentrates on \(X_\alpha\) we have 
\(j(\ell)(\kappa)=j(f_\alpha)(\kappa)=\alpha\).
\end{proof}

\begin{corollary}
\label{cor:SmallGuessingIffEnoughMeasures}
Let \(\kappa\) be a measurable cardinal and fix a subset \(A\subseteq H_{\kappa^+}\)
of size at most \(\kappa\). Then there is an \(A\)-guessing measurable Laver function for
\(\kappa\) if and only if there are at least \(|A|\) many normal measures on \(\kappa\).
\end{corollary}

\begin{proof}
The forward direction follows just as before: each target in \(A\) gives its own
normal measure on \(\kappa\). Conversely, if there are at least \(|A|\) many normal
measures on \(\kappa\) then, by theorem~\ref{thm:OrdGuessingIffEnoughMeasures},
there is an \(|A|\)-guessing measurable Laver function \(\bar{\ell}\). Fix a bijection
\(f\colon |A|\to A\). We may assume, moreover, that \(A\subseteq\mathcal{P}(\kappa)\). 
Then we can define an \(A\)-guessing Laver function \(\ell\) by letting 
\(\ell(\xi)=f(\bar{\ell}(\xi))\cap\xi\). This definition works: to guess \(f(\alpha)\) we let
\(\bar{\ell}\) guess \(\alpha\) via some \(j\). Then 
\(j(\ell)(\kappa)=j(f(\alpha))\cap\kappa=f(\alpha)\).
\end{proof}

Lemma~\ref{lemma:DiscreteMeasures} can be recast in somewhat different language,
giving it, and the subsequent results, a more topological flavour.

Given a cardinal \(\kappa\) let \(\mathcal{M}(\kappa)\) be the set of normal measures on
\(\kappa\). We can topologize \(\mathcal{M}(\kappa)\) by having, for each 
\(X\subseteq\kappa\), a basic neighbourhood \([X]=\{\mu\in \mathcal{M}(\kappa); X\in\mu\}\)
(this is just the topology induced on \(\mathcal{M}(\kappa)\) by the Stone topology on
the space of ultrafilters on \(\kappa\)). Lemma~\ref{lemma:DiscreteMeasures} can now
be restated to say that any subspace of \(\mathcal{M}(\kappa)\) of size at most \(\kappa\)
is discrete and, moreover, the \emph{discretizing family} of subsets of \(\kappa\)
witnessing this can be taken to be pairwise disjoint. One might thus hope to show
the existence of Laver functions by exhibiting even larger discrete subspaces of
\(\mathcal{M}(\kappa)\). In pursuit of that goal we obtain the following generalization of 
corollary~\ref{cor:SmallGuessingIffEnoughMeasures}.

\begin{theorem}
Let \(\kappa\) be a measurable cardinal and \(A\subseteq\mathcal{P}(\kappa)\). 
Then \(\Ldmeas_\kappa(A)\) holds if and only if
there are for each \(a\in A\) a set \(S_a\subseteq\kappa\) and a normal
measure \(\mu_a\) on \(\kappa\) such that \(\{\mu_a;a\in A\}\) is discrete in \(\mathcal{M}(\kappa)\),
as witnessed by \(\{S_a;a\in A\}\), and \(S_a\cap S_b\subseteq \{\xi;
a\cap\xi=b\cap\xi\}\).
\end{theorem}

\noindent Note that we could have relaxed our hypothesis to \(A\subseteq H_{\kappa^+}\)
by working with Mostowski codes.

\begin{proof}
Assume first that \(\ell\) is a measurable \(A\)-guessing Laver function for \(\kappa\).
Then we can let \(S_a=\{\xi;\ell(\xi)=a\cap\xi\}\). Obviously we have \(j(\ell)(\kappa)=a\)
if and only if the measure derived from \(j\) concentrates on \(S_a\). It follows
that the measures \(\mu_a\) derived this way form a discrete subspace of
\(\mathcal{M}(\kappa)\) and we obviously have
\(S_a\cap S_b\subseteq \{\xi;a\cap \xi=b\cap\xi\}\).

Conversely, assume we have such a discrete family of measures \(\mu_a\) 
and a discretizing family of sets \(S_a\). We can define an \(A\)-guessing
measurable Laver function \(\ell\) by letting \(\ell(\xi)=a\cap \xi\) where
\(a\) is such that \(\xi\in S_a\). This is well defined by the coherence condition
imposed upon the \(S_a\) and it is easy to see that \(\ell\) satisfies the guessing property.
\end{proof}

This topological viewpoint presents a number of questions which might suggest an
approach to question~\ref{q:MeasuresGiveLavers}. For example, 
it is unclear whether, given a discrete family of normal measures one can find an almost
disjoint discretizing family as in the above theorem. Even more pressingly, we do not know
whether it is possible for \(\mathcal{M}(\kappa)\) to have no discrete subspaces of size
\(\kappa^+\) (while itself having size at least \(\kappa^+\)).

\subsection{Laver trees}

Thus far we have thought of joint Laver diamonds as simply matrices or sequences of Laver
diamonds. To better facilitate the reflection properties required for the usual forcing
iterations using prediction, we would now like a different representation.
A reasonable attempt seems to be trying to align the joint Laver sequence
with the full binary tree of height \(\kappa\).

\begin{definition}
Let \(\kappa\) be a large cardinal supporting a notion of Laver diamond.
A \emph{Laver tree}\footnote{Not to be confused with particularly bushy trees used as conditions in Laver forcing.} 
for \(\kappa\) is a labelling of the binary tree such that the labels along the
branches form a joint Laver sequence. More precisely, a Laver tree is a function \(D\colon 
\funcs{<\kappa}{2}\to V\)
such that for any sequence of targets \(\langle a_s;s\in\funcs{\kappa}{2}\rangle\)
there is an elementary embedding \(j\), witnessing the largeness \(\kappa\),
such that \(j(D)(s)=a_s\) for all \(s\in \funcs{\kappa}{2}\).

Given an \(I\subseteq\funcs{\kappa}{2}\), an \emph{\(I\)-Laver tree} for \(\kappa\) is
a function \(D\) as above, satisfying the same guessing property but only for sequences of
targets indexed by \(I\).
\end{definition}

The guessing property of a Laver tree is illustrated in the following diagram:

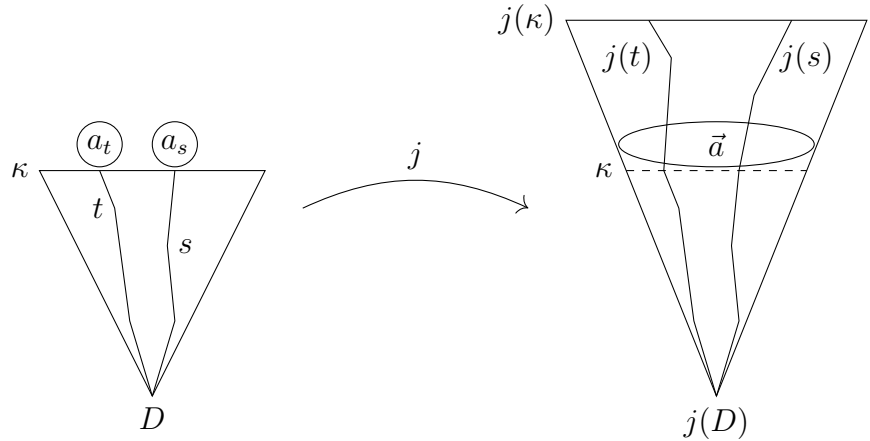
\begin{figure}[ht]
\centering
\begin{tikzpicture}
\draw (0,0)--(3,0)--(1.5,-3)--cycle;
\node [below] at (1.5,-3) {\(D\)};
\node [left] at (0,0) {\(\kappa\)};
\draw (1.5,-3)--(1.8,-2)--(1.7,-1)--(1.8,0);
\node [right] at (1.7,-1) {\(s\)};
\draw (1.8,0.35) circle [radius=0.3] node {\(a_s\)};
\draw (1.5,-3)--(1.2,-2)--(1,-0.5)--(0.8,0);
\node [left] at (1,-0.5) {\(t\)};
\draw (0.8, 0.35) circle [radius=0.3] node {\(a_t\)};
\draw [-{>[scale=1.5]}] (3.5,-0.5) to [out=25, in=155] node[midway,above]{\(j\)} (6.5,-0.5);
\draw (7,2)--(11,2)--(9,-3)--cycle;
\node [below] at (9,-3) {\(j(D)\)};
\node [left] at (7,2) {\(j(\kappa)\)};
\node at (7.5,0) {\(\kappa\)};
\draw [dashed] (7.8,0)--(10.2,0);
\draw (9,0.35) ellipse (1.3 and 0.3) node {\(\vec{a}\)};
\draw (9,-3)--(9.3,-2)--(9.2,-1)--(9.3,0)--(9.5,1)--(10,2);
\node [right] at (9.7,1.5) {\(j(s)\)};
\draw (9,-3)--(8.7,-2)--(8.5,-0.5)--(8.3,0)--(8.4,1.5)--(8.1,2);
\node [left] at (8.3,1.5) {\(j(t)\)};
\end{tikzpicture}
\caption{The guessing property of Laver trees}
\end{figure}

If the degree of supercompactness of \(\kappa\) is sufficiently large then Laver trees 
are nothing new.

\begin{proposition}
\label{prop:FullyTreeableThetaSCJLDExist}
Suppose \(\kappa\) is \(\theta\)-supercompact and \(\theta\geq 2^\kappa\). Then
a \(\theta\)-supercompactness Laver tree for \(\kappa\) exists if and only if a 
\(\theta\)-supercompactness Laver function for \(\kappa\) does
(if and only if\/ \(\jLdthetasc_{\kappa,2^\kappa}\) holds).
\end{proposition}

\begin{proof}
The forward implication is trivial, so we focus on the reverse implication. Let
\(\ell\) be a Laver function. For any \(t\in\funcs{<\kappa}{2}\) define
\(D(t)=\ell(|t|)(t)\) if this makes sense and \(D(t)=\emptyset\) otherwise. 
We claim this defines a Laver tree for \(\kappa\). Indeed, let 
\(\vec{a}=\langle a_s;s\in\funcs{\kappa}{2}\rangle\) be a sequence of targets. Since
\(\theta\geq 2^\kappa\) we get \(\vec{a}\in H_{\theta^+}\), so there is a 
\(\theta\)-supercompactness embedding \(j\) such that \(j(\ell)(\kappa)=\vec{a}\).
Therefore, given any \(s\in\funcs{\kappa}{2}\), we have
\(
j(D)(s)=j(\ell)(\kappa)(s)=a_s
\)
\end{proof}

In other situations, however, the existence of a Laver tree
can have strictly higher consistency strength than merely a 
\(\theta\)-supercompact cardinal.

\begin{definition}
Let \(X\) be a set and \(\theta\) a cardinal.
A cardinal \(\kappa\) is \emph{\(X\)-strong with closure \(\theta\)} if there is an
elementary embedding \(j\colon V\to M\) with critical point \(\kappa\) such that
\(\funcs{\theta}{M}\subseteq M\) and \(X\in M\).
\end{definition}

\begin{proposition}
\label{prop:FullyTreeableThetaSCJLDStronger}
Suppose \(\kappa\) is \(\theta\)-supercompact and there is a \(\theta\)-supercompactness 
Laver tree 
for \(\kappa\). Then \(\kappa\) is \(X\)-strong with closure
\(\theta\) for any \(X\subseteq H_{\theta^+}\) of size at most \(2^\kappa\).
\end{proposition}

\begin{proof}
Suppose \(D\colon \funcs{<\kappa}{2}\to V_\kappa\) is a Laver tree
and fix an \(X\subseteq H_{\theta^+}\) of size at most \(2^\kappa\).
Let \(f\colon \funcs{\kappa}{2}\to X\) enumerate \(X\). We can then find a 
\(\theta\)-supercompactness embedding \(j\colon V\to M\) with critical point \(\kappa\)
such that \(j(D)(s)=f(s)\) for all \(s\in\funcs{\kappa}{2}\). In particular,
\(X=j(D)[\funcs{\kappa}{2}]\) is an element of \(M\), as required.
\end{proof}

If \(2^\kappa\leq\theta\) then \(X\)-strongness with closure \(\theta\) for
all \(X\subseteq H_{\theta^+}\) of size \(2^\kappa\) 
amounts to just \(\theta\)-supercompactness and 
proposition~\ref{prop:FullyTreeableThetaSCJLDExist} gives the full equivalence of 
Laver functions
and Laver trees. But if \(\theta< 2^\kappa\) then \(X\)-strongness with closure \(\theta\)
can have additional consistency strength. For example, 
we might choose \(X\) to be a normal measure on \(\kappa\) to see that \(\kappa\) must
have nontrivial Mitchell rank (by iterating this idea we can even deduce that
\(o(\kappa)=(2^\kappa)^+\)). In the typical scenario
where \(2^\kappa=2^\theta=\theta^+\), we can also reach higher and choose \(X\) to be a 
normal measure on \(\mathcal{P}_\kappa\theta\) and see that \(\kappa\) must also have 
nontrivial \(\theta\)-supercompactness Mitchell rank. 
We can use this observation to show
that there might not be any Laver trees, even in the presence of very long joint Laver
sequences.

\begin{theorem}
\label{thm:SeparateLongAndTreeableSCJLD}
Suppose GCH holds and let \(\kappa\) be \(\theta\)-supercompact where either 
\(\theta=\kappa\) or \(\cf(\theta)>\kappa\). Then there is a cardinal-preserving
forcing extension in which \(\kappa\) remains \(\theta\)-supercompact, has a 
\(\theta\)-supercompactness joint
Laver sequence of length \(2^\kappa\), but is also the least measurable cardinal.
In particular, \(\theta<2^\kappa\) and \(\kappa\) has no \(\theta\)-supercompactness
Laver trees in the extension.
\end{theorem}

\begin{proof}
We may assume by prior forcing
as in the proof of theorem~\ref{thm:ForceLongJLDiamondSC}, 
that \(\kappa\) has a Laver function. 
Additionally, by performing either Magidor's iteration of Prikry forcing
(see~\cite{Magidor1976:IdentityCrises}) or applying an argument due to
Apter and Shelah (see~\cite{ApterShelah1997:MenasResultBestPossible}), depending
on whether \(\theta=\kappa\) or not, we may assume that, in addition to being
\(\theta\)-supercompact, \(\kappa\) is also the least measurable cardinal.\footnote{Of
course, if \(\theta>\kappa\), this arrangement requires a strong failure of GCH at 
\(\kappa\). In fact, \(2^\kappa=\theta^+\) in the Apter--Shelah model.}

We now apply corollary~\ref{cor:LeastSCCanHaveLongJLD} and
arrive at a model where \(\kappa\) carries a \(\theta\)-supercompactness joint Laver sequence 
of length \(2^\kappa\), but is also the least measurable. It follows that there can be no 
(\(\theta\)-supercompactness or even measurable) Laver trees for \(\kappa\),
since, by the discussion above, their existence would imply that \(\kappa\) has nontrivial Mitchell rank, implying that there are many measurables below \(\kappa\).
\end{proof}

Proposition~\ref{prop:FullyTreeableThetaSCJLDStronger} can be improved slightly to give a 
jump in consistency strength even for \(I\)-Laver trees where \(I\) is not the whole
set of branches.
A simple modification of the proof given there yields the following result, together
with the corresponding version of theorem~\ref{thm:SeparateLongAndTreeableSCJLD}.

\begin{theorem}
\label{thm:DefinableITreeableThetaSCJLDStronger}
Suppose \(\kappa\) is \(\theta\)-supercompact and there is a \(\theta\)-supercompactness 
\(I\)-Laver tree for \(\kappa\)
for some \(I\subseteq\funcs{\kappa}{2}\) of size \(2^\kappa\). 
If \(I\) is definable (with parameters) over \(H_{\kappa^+}\)
then \(\kappa\) is \(X\)-strong with closure \(\theta\) for any 
\(X\subseteq H_{\theta^+}\) of size at most \(2^\kappa\).
\end{theorem}

The above theorem notwithstanding, 
we shall give a construction which shows that the 
existence of an \(I\)-Laver tree does not
yield additional consistency strength, provided that we allow \(I\) to be sufficiently
foreign to \(H_{\kappa^+}\). 
The argument will rely on being able to surgically alter
a Cohen subset of \(\kappa^+\) in a variety of ways. To this end we fix some
notation beforehand.

\begin{definition}
Let \(f\) and \(g\) be functions.
The \emph{graft} of \(f\) onto \(g\) is the function \(g\wr f\), defined on \(\dom(g)\) by
\[(g\wr f)(x)=\begin{cases}f(x);&x\in\dom(g)\cap\dom(f)\\ 
g(x);& x\in\dom(g)\setminus\dom(f)\end{cases}\]
\end{definition}

\noindent Essentially, the graft replaces the values of \(g\) with those of \(f\) on their 
common domain.

\begin{theorem}
\label{thm:ThinSubsets}
Let \(\lambda\) be a regular cardinal and assume \(\diamond_\lambda\) holds.
Suppose \(M\) is a transitive model of \textup{ZFC} (either set- or class-sized) such that
\(\lambda\in M\) and \(M^{<\lambda}\subseteq M\) and \(|\mathcal{P}(\lambda)^M|=\lambda\).
Then there are an unbounded  set \(I\subseteq\lambda\) and a function \(g\colon\lambda\to H_\lambda\)
such that, given any \(f\colon I\to H_\lambda\), the graft \(g\wr f\) is generic for
\(\Add(\lambda,1)\) over \(M\).\footnote{Here we take the forcing-equivalent version of
\(\Add(\lambda,1)\) which adds a function \(g\colon\lambda\to H_\lambda\) by initial segments.}
\end{theorem}

\noindent The hypothesis of \(\diamond_\lambda\) is often automatically satisfied.
Specifically, our assumptions about \(M\) imply that \(2^{<\lambda}=\lambda\).
If \(\lambda=\kappa^+\) is a successor, this gives
\(2^\kappa=\kappa^+\) which already implies \(\diamond_\lambda\) by
a result of Shelah~\cite{Shelah2010:Diamonds}.

\begin{proof}[Proof]
Let \(\langle f_\alpha;\alpha<\lambda\rangle\), with 
\(f_\alpha\colon \alpha\to H_{|\alpha|}\), be a \(\diamond_\lambda\)-sequence and
fix an enumeration \(\langle D_\alpha;\alpha<\lambda\rangle\) of the open dense subsets of
\(\Add(\lambda,1)\) in \(M\). We shall construct by recursion a descending sequence of
conditions \(p_\alpha\in\Add(\lambda,1)\) and an increasing sequence of sets
\(I_\alpha\) as approximations to \(g\) and \(I\). Specifically, we shall use
\(\diamond_\lambda\) to guess pieces of any potential function \(f\) and ensure along the
way that the modified conditions \(p_\alpha\wr f\) meet all of the listed dense sets.

Suppose we have built the sequences \(\langle p_\alpha;\alpha<\gamma\rangle\) and
\(\langle I_\alpha;\alpha<\gamma\rangle\) for some \(\gamma<\lambda\).
Let \(I_\gamma^*=\bigcup_{\alpha<\gamma} I_\alpha\).
Let \(p_\gamma^*\in M\) be an extension of \(\bigcup_{\alpha<\gamma} p_\alpha\) 
such that \(I_\gamma^*\subseteq\dom(p_\gamma^*)\in\lambda\); such an extension exists
in \(M\) by our assumption on the closure of \(M\).

Let us briefly summarize the construction. We shall surgically modify the condition 
\(p_\gamma^*\) by grafting the function given by \(\diamond_\lambda\) onto it.
We shall then extend this modified condition to meet one of our dense sets, after
which we will undo the surgery. We will be left with a condition \(p_\gamma\)
which is one step closer to ensuring that the result of one particular grafting
\(g\wr f\) is generic. At the same time we also extend \(I_\gamma^*\) by
adding a point beyond the domains of all the conditions constructed so far.

More precisely, let  \(\tilde{p}^*_\gamma=p^*_\gamma \wr (f_\gamma\rest I^*_\gamma)\).
This is still a condition in \(M\).
Let \(\widetilde{p}_\gamma\) be any extension of this condition
inside \(D_{\eta_\gamma}\), where \(\eta_\gamma\) is the least such that
\(\tilde{p}^*_\gamma\notin D_{\eta_\gamma}\), and satisfying 
\(\dom(\widetilde{p}_\gamma)\in\lambda\). Finally, we undo the initial graft and set
\(p_\gamma=\widetilde{p}_\gamma \wr (p^*_\gamma\rest I^*_\gamma)\). Note that
we have \(p_\gamma\leq p^*_\gamma\). We also extend our approximation to \(I\) with
the first available point, letting
\(I_\gamma=I_\gamma^*\cup\{\min(\lambda\setminus\dom(p_\gamma))\}\).

Once we have completed this recursive construction we can set
\(I=\bigcup_{\gamma<\lambda}I_\gamma\) and \(g=\bigcup_{\gamma<\lambda}p_\gamma\).
Let us check that these do in fact have the desired properties.

Let \(f\colon I\to H_\lambda\) be a function. We need to show that \(g\wr f\) is generic over
\(M\). Using \(\diamond_\lambda\), we find that there are stationarily many \(\gamma\)
such that \(f_\gamma=f\rest\gamma\).
Note also that there are club many \(\gamma\) such that \(I_\gamma^*\subseteq\gamma\)
is unbounded, and together this means that 
\(S=\set{\gamma}{f_\gamma\rest I_\gamma^*=f\rest(I\cap\gamma)}\) is stationary.
The conditions \(\widetilde{p}_\gamma\) for \(\gamma\in S\) extend each other
and we have \(\bigcup_{\gamma\in S}\widetilde{p}_\gamma=g\wr f\). Furthermore, since
the sets \(D_\alpha\) are open, the construction of \(\widetilde{p}_\gamma\) ensures
that eventually these conditions will meet every such dense set, showing that \(g\wr f\)
really is generic.
\end{proof}

The construction in the above proof is quite flexible and can be modified to make the
set \(I\) generic in various ways as well (for example, we can arrange for \(I\) to be
Cohen or dominating over \(M\) etc.).

\begin{theorem}
\label{thm:ForceITreeableThetaSCJLD}
If \(\kappa\) is \(\theta\)-supercompact then there is a forcing extension in which there is
a \(\theta\)-supercompactness \(I\)-Laver tree for \(\kappa\) for some 
\(I\subseteq\funcs{\kappa}{2}\) of size \(2^\kappa\).
\end{theorem}

\begin{proof}
If \(\theta\geq 2^\kappa\) then even a single Laver function for \(\kappa\) gives rise
to a full Laver tree, by proposition~\ref{prop:FullyTreeableThetaSCJLDExist},
and we can force the existence of a Laver function by theorem~\ref{thm:ForceLongJLDiamondSC}.
We thus focus on the remaining case when \(\kappa\leq\theta<2^\kappa\).

We make similar simplifying assumptions as in theorem~\ref{thm:ForceLongJLDiamondSC}.
Just as there we assume that \(\theta=\theta^{<\kappa}\).
Furthermore, we may assume that \(2^\theta=\theta^+\), since this can be forced
without adding subsets to \(\mathcal{P}_\kappa\theta\) and affecting the 
\(\theta\)-supercompactness of \(\kappa\). Note that these cardinal arithmetic hypotheses
imply that \(2^\kappa=\theta^+\).

Let \(\P\) be the length \(\kappa\) Easton support iteration
which adds, in a recursive fashion, a labelling of the tree \(\funcs{<\kappa}{2}\)
of the extension. Specifically, let \(\P\)
force with \(\Q_\gamma=\Add(2^\gamma,1)\) at each inaccessible \(\gamma<\kappa\) stage
\(\gamma\). Let \(G\subseteq\P\) be generic and let \(G_\gamma\) be the piece added at stage 
\(\gamma\).
Using suitable coding, we can see each \(G_\gamma\), in \(V[G]\),
as a function \(G_\gamma\colon \funcs{\gamma}{2}\to H_{\gamma^+}\);
in particular, we should note that the iteration \(\P\) does not add any nodes to the
tree \((\funcs{\leq\gamma}{2})^{V[G]}\) at stage \(\gamma\) or later, so that \(G_\gamma\)
really does label the whole level \(\funcs{\gamma}{2}\).
Thus \(G\) induces a map \(D\colon \funcs{<\kappa}{2}\to V_\kappa[G]\), by 
extending the \(G_\gamma\) in any way we like to the entire tree. 
We shall show that \(D\) is an \(I\)-Laver tree for some specifically chosen \(I\).

Fix a
\(\theta\)-supercompactness embedding \(j\colon V\to M\) in \(V\). Note that
\(M[G]^{\theta}\subseteq M[G]\) in \(V[G]\) as well, since the forcing \(\P\) is
\(\theta^+\)-cc. Furthermore, in \(V[G]\) we still have \(2^\theta=\theta^+\), which
implies \(\diamond_{\theta^+}\) by a result of Shelah~\cite{Shelah2010:Diamonds}.
Now apply theorem~\ref{thm:ThinSubsets} to \(M[G]\) and \(\lambda=\theta^+\) to obtain
an \(I\subseteq \funcs{\kappa}{2}\) of size \(\theta^+\) and a function 
\(g\colon \funcs{\kappa}{2}\to H_{\theta^+}\)
such that for any \(f\colon I\to H_{\theta^+}\) in \(V[G]\), 
the graft \(g\wr f\) is generic over
\(M[G]\). We claim that \(D\) is an \(I\)-Laver tree.

To check the guessing property, fix a sequence of targets 
\(\vec{a}=\langle a_s;s\in I\rangle\) in \(V[G]\). We shall lift the embedding \(j\)
to \(V[G]\). Let us write \(j(\P)=\P*\Q_\kappa*\Ptail\).
We know that \(g\wr \vec{a}\) is \(M[G]\)-generic for \(\Q_\kappa\), so we only need to find
the further generic for \(\Ptail\). We easily see that
\(M[G][g\wr \vec{a}]^{\theta}\subseteq M[G][g\wr \vec{a}]\) in \(V[G]\), that \(\Ptail\) is
\(\leq\theta\)-closed in that model, and that \(M[G][g\wr \vec{a}]\) only has
\(\theta^+\)-many subsets of \(\Ptail\). We can thus diagonalize against these
dense sets in \(\theta^+\)-many steps and produce a generic \(\Gtail\) for
\(\Ptail\). Putting all of this together, we can lift \(j\) to
\(j\colon V[G]\to M[j(G)]\) in \(V[G]\), where \(j(G)=G*(g\wr \vec{a})*\Gtail\). 
Now consider \(j(D)\). This is
exactly the labelling of the tree \(\funcs{<j(\kappa)}{2}\) in \(M[j(G)]\)
given by \(j(G)\). Furthermore, for any \(s\in I\), we have 
\(j(D)(s)=(g\wr \vec{a})(s)=a_s\), verifying the guessing property.
\end{proof}

Given a Laver tree for \(\kappa\), it is easy to produce a joint Laver sequence of length
\(2^\kappa\) from it by just reading the labels along each branch of the Laver tree.
The resulting sequence then exhibits a large degree of coherence.
We might wonder about the possibility of reversing this process, starting with a joint
Laver sequence and attempting to fit it into a tree. A sequence for which this can
be done might be called treeable. But, taken literally, this notion is not very robust.
For example, all functions in a treeable joint Laver sequence must have the same value at
0. This means that we could take a treeable sequence and modify it in an inessential way
to destroy its treeability. To avoid such trivialities, we relax the definition to only
ask for coherence modulo bounded perturbations.

\begin{definition}
\label{def:treeability}
Let \(\kappa\) be a regular cardinal and \(\vec{f}=\langle f_\alpha;\alpha<\lambda\rangle\)
a sequence of functions defined on \(\kappa\). The sequence \(\vec{f}\) is
\emph{treeable} if there are a bijection \(e\colon \lambda\to \funcs{\kappa}{2}\)
and a tree labelling \(D\colon \funcs{<\kappa}{2}\to V\) such that, for all \(\alpha<\lambda\),
we have \(D(e(\alpha)\rest\xi)=f_\alpha(\xi)\) for all but boundedly many \(\xi<\kappa\).

Given an \(I\subseteq\funcs{\kappa}{2}\), the sequence \(\vec{f}\) is \emph{\(I\)-treeable}
if the above holds for a bijection \(e\colon \lambda\to I\).
\end{definition}

\begin{lemma}
\label{lemma:GenericNotTreeable}
Let \(\kappa\) be a regular cardinal and assume that \((2^{<\kappa})^+\leq 2^\kappa\). 
Let \(G\subseteq\Add(\kappa,2^\kappa)\) be generic. Then \(G\) is not treeable.
\end{lemma}

\begin{proof}
Let us write \(G=\langle g_\alpha;\alpha<2^\kappa\rangle\) as a sequence of its slices.
Now suppose that this
sequence were treeable and let \(\dot{e}\) and \(\dot{D}\) be names for the indexing function
and the labelling of \(\funcs{<\kappa}{2}\), respectively. Our cardinal arithmetic
assumption implies that the name \(\dot{D}\) only
involves conditions from a bounded part of the poset \(\Add(\kappa,2^\kappa)\), so we
may assume that the labelling \(D\) exists already in the ground model.
Let \(p\) be an arbitrary condition and \(\alpha<\kappa\). Since we assumed that
\(G\) was forced to be treeable, there is a name \(\dot{\gamma}\)
for an ordinal, beyond which \(G_\alpha\) agrees with \(D\rest f(\alpha)\).
By strengthening \(p\) if necessary, we may assume that the value of \(\dot{\gamma}\) has
been decided. We now inductively construct a countable descending sequence of conditions 
below \(p\), deciding longer and longer initial segments of \(\dot{e}(\alpha)\),
in such a way that, for some \(\delta>\gamma\), their union \(p^*\leq p\) decides
\(\dot{e}(\alpha)\rest\delta\) but does not decide \(G_\alpha(\delta)\).
Then \(p^*\) can be further extended to a condition forcing
\(G_\alpha(\delta)\neq D(\dot{e}(\alpha)\rest\delta)\), which contradicts the fact
that \(p\) forces that \(G\) is treeable.
\end{proof}

\begin{corollary}
\label{cor:NontreeableSCJLD}
If \(\kappa\) is \(\theta\)-supercompact then there is a forcing extension in which
there is a nontreeable \(\theta\)-supercompactness joint Laver sequence for \(\kappa\) of 
length \(2^\kappa\).
\end{corollary}

\begin{proof}
The joint Laver sequence constructed in theorem~\ref{thm:ForceLongJLDiamondSC} was
added by forcing with \(\Add(\kappa,2^\kappa)\), so it is not treeable by
lemma~\ref{lemma:GenericNotTreeable}.
\end{proof}

\section{Joint Laver diamonds for strong cardinals}
\label{sec:JLDStrong}
\begin{definition}
A function \(\ell\colon\kappa\to V_\kappa\) is
a \(\theta\)-strongness Laver function if it guesses elements of \(V_\theta\)
via \(\theta\)-strongness embeddings with critical point \(\kappa\).

If \(\kappa\) is fully strong then a function \(\ell\colon\kappa\to V_\kappa\) is a
Laver function for \(\kappa\) if it is a \(\theta\)-strongness Laver function
for \(\kappa\) for all \(\theta\).
\end{definition}

As in the supercompact case, \(2^\kappa\) is the largest possible cardinal length of a
\(\theta\)-strongness joint Laver sequence for \(\kappa\), just because there are only
\(2^\kappa\) many functions \(\ell\colon\kappa\to V_\kappa\).

The set of targets \(V_\theta\) is a bit unwieldy and lacks some basic closure
properties, particularly in the case when \(\theta\) is a successor ordinal.
The following lemma shows that, modulo some coding, we can recover a good deal of closure
under sequences.

\begin{lemma}
\label{lemma:FlatCoding}
Let \(\theta\) be an infinite ordinal and let \(I\in V_\theta\) be a set.
If \(\theta\) is  successor ordinal or \(\cf(\theta)>|I|\) then \(V_\theta\) is closed
under a coding scheme for sequences indexed by \(I\). Moreover, this coding is
\(\Delta_0\)-definable.
\end{lemma}

\begin{proof}
If \(\theta=\omega\) then the \(I\) under consideration are finite. Since \(V_\omega\)
is already closed under finite sequences we need only deal with \(\theta>\omega\).

Fix in advance a simply definable flat pairing function \([\cdot,\cdot]\)
(flat in the sense that any infinite \(V_\alpha\) is closed under it; the Quine--Rosser
pairing function will do).

Let \(\vec{a}=\langle a_i; i\in I\rangle\) be a sequence of elements of \(V_\theta\).
For each \(i\in I\) we can find an (infinite) ordinal \(\theta_i<\theta\) such that
\(a_i\cup \{i\}\subseteq V_{\theta_i}\). Now let \(\widetilde{a}_i=\{[i,b];b\in a_i\}
\subseteq V_{\theta_i}\) and finally define
\(\widetilde{a}=\bigcup_{i\in I}\widetilde{a}_i\). We see that 
\(\widetilde{a}\subseteq V_{\sup_i \theta_i}\) and, under our hypotheses,
\(\sup_i\theta_i<\theta\). It follows that
\(\widetilde{a}\in V_{(\sup_i\theta_i)+1} \subseteq V_\theta\) as required.
\end{proof}

\begin{proposition}
\label{prop:ThetaStrJLDiamond}
Let \(\kappa\) be \(\theta\)-strong with \(\kappa+2\leq\theta\) and let 
\(\lambda\leq 2^\kappa\) be a cardinal. If there is a \(\theta\)-strongness Laver 
function for \(\kappa\)
and \(\theta\) is either a successor ordinal or \(\lambda<\cf(\theta)\) then there is a
\(\theta\)-strongness joint Laver sequence of length \(\lambda\) for \(\kappa\).
\end{proposition}

\noindent In particular, if \(\theta\) is a successor then a single \(\theta\)-strongness
Laver function already yields a joint Laver sequence of length \(2^\kappa\), the maximal
possible.

\begin{proof}
We aim to imitate the proof of proposition~\ref{prop:SCHasSomeJLDiamond}. 
%
To that end, fix an \(I\subseteq\mathcal{P}(\kappa)\) of size \(\lambda\) and a
bijection \(f\colon \lambda\to I\).
If \(\ell\) is a Laver function for \(\kappa\), we define a 
joint Laver sequence by letting \(\ell_\alpha(\xi)\), for each \(\alpha<\lambda\),
be the element of \(\ell(\xi)\) with index \(f(\alpha)\cap \xi\) in the coding scheme
described in lemma~\ref{lemma:FlatCoding}.

It is now easy to verify that the functions \(\ell_\alpha\) form a joint Laver sequence:
given a sequence of targets \(\vec{a}\), we can replace it, by using \(f\) and 
lemma~\ref{lemma:FlatCoding}, with a coded version \(\widetilde{a}\in V_\theta\).
We can then use \(\ell\) to guess \(\widetilde{a}\) and the \(\theta\)-strongness
embedding obtained this way will witness the joint guessing property of the \(\ell_\alpha\).
\end{proof}

Again, as in the supercompact case, if the Laver diamond we started with works for
several different \(\theta\) then the joint Laver sequence derived above
will also work for those same \(\theta\). In particular, if \(\kappa\) is strong
then combining the argument from proposition~\ref{prop:ThetaStrJLDiamond} with
the Gitik--Shelah construction of a strongness Laver function in
\cite{GitikShelah1989:IndestructibilityOfStrong} gives an analogue of corollary~\ref{cor:SCHasLongJLDiamond}.

\begin{corollary}
\label{cor:StrHasJLD}
If \(\kappa\) is strong then there is a strongness joint Laver sequence for \(\kappa\)
of length \(2^\kappa\).
\end{corollary}

Proposition~\ref{prop:ThetaStrJLDiamond} implies that in most cases (that is, for most 
\(\theta\)) we do not need to do any additional work beyond ensuring that there is a 
\(\theta\)-strongness Laver function for \(\kappa\) to automatically also find the longest 
possible joint Laver sequence. For example, if
\(\theta\) is a successor or if \(\cf(\theta)>2^\kappa\) then a single \(\theta\)-strongness
Laver function yields a joint Laver sequence of length \(2^\kappa\).
To gauge the consistency strength of the existence of \(\theta\)-strongness
joint Laver sequences for \(\kappa\) we should therefore only focus on the consistency
strength required for a single Laver diamond, and, separately, 
on \(\theta\) of low cofinality.

The following lemma is implicit in many arguments involving the preservation of
strongness; we include a proof for completeness.

\begin{lemma}
\label{lemma:NamesInVAlpha}
Let \(\kappa\) be a cardinal and suppose \(\P\subseteq V_\kappa\) is a
poset. Let \(G\subseteq \P\) be generic and assume that \(\kappa\) is a \(\beth\)-fixed
point in \(V[G]\). If \(\alpha\geq\kappa+1\) is an ordinal then
\(V_\alpha\) contains (codes for) names for all elements of \(V[G]_\alpha=V_\alpha^{V[G]}\).
\end{lemma}

\begin{proof}
We argue by induction on \(\alpha\); the limit case is clear. 
For the base case \(\alpha=\kappa+1\) notice that,
since \(\kappa\) is a \(\beth\)-fixed point in \(V[G]\), 
names for elements of \(V[G]_{\kappa+1}\) can be recovered from names for subsets of 
\(\kappa\). But nice names for subsets of \(\kappa\) are essentially just subsets
of \(\kappa\times \P\) and are thus elements of \(V_{\kappa+1}\).

For the successor case, observe that nice names for subsets of \(V[G]_\alpha\)
can be represented by subsets of \(V_\alpha\times\P\) since, by the induction hypothesis,
\(V_\alpha\) has names for elements of \(V[G]_\alpha\). Therefore \(V_{\alpha+1}\) has
all of these nice names for elements of \(V[G]_{\alpha+1}\).
\end{proof}


%

In the following theorem we give a forcing construction of a \(\theta\)-strongness
Laver function.
Combined with proposition~\ref{prop:ThetaStrJLDiamond} this will give us better bounds on 
the consistency strength of the existence of \(\theta\)-strongness joint Laver
sequences of various lengths. For this we need a preliminary lemma,
analogous to lemma~\ref{lemma:ThetaSCMenasFunction}.

\begin{lemma}
\label{lemma:ThetaStrMenasFunction}
If \(\kappa\) is \(\theta\)-strong then there is a function \(f\colon \kappa\to\kappa\)
with the Menas property; i.e.\ we have \(j(f)(\kappa)=\theta\) for some
\(\theta\)-strongness embedding \(j\colon V\to M\) with critical point \(\kappa\).
\end{lemma}

\begin{theorem}
\label{thm:ThetaStrLDForcing}
Let \(\kappa\) be \(\theta\)-strong with \(\kappa+2\leq\theta\). 
Then there is a forcing extension in which there is a \(\theta\)-strongness Laver function for \(\kappa\).
\end{theorem}

\begin{proof}
We may assume by prior forcing that GCH holds.
Fix a Menas function \(f\) as in lemma~\ref{lemma:ThetaStrMenasFunction} and let
\(\P\) be the length \(\kappa\) Easton support iteration which forces with
\(\Q_\gamma=\Add(\beth_{f(\gamma)}^+,1)\) at inaccessible \(\gamma\). If \(G\subseteq\P\)
is generic we can extract a Laver function \(\ell\) from it by letting
\(\ell(\xi)\) be the set whose Mostowski code appears as the length \(\beth_{f(\xi)}\)
initial segment of \(G_\xi\) if \(\xi\) is inaccessible and \(\ell(\xi)=\emptyset\)
otherwise.

To see that this definition works let \(A\in V_\theta^{V[G]}\). Since 
\(\P\subseteq V_\kappa\), lemma~\ref{lemma:NamesInVAlpha} implies that 
\(A\) has a name \(\dot{A}\in V_\theta\).

Now fix a \(\theta\)-strongness embedding \(j\colon V\to M\) such that \(j(f)(\kappa)=
\theta\). We may additionally assume that this embedding is an ultrapower given by
a \((\kappa,V_\theta)\)-extender, i.e.\ that \(M=\{j(h)(a);a\in V_\theta,
h\colon V_\kappa\to V\}\). We will lift this embedding through the forcing \(\P\).

The poset \(j(\P)\) factors as \(j(\P)=\P*\P_{\text{tail}}\), where
\(M[G]\vDash \text{``\(\P_{\text{tail}}\) is \(\leq\beth_\theta\)-closed''.}\)
Let \(X=\{j(h)(\kappa,\dot{A});h\colon V_\kappa\to V\}\) be the Skolem hull of
\(\{\kappa,\dot{A}\}\cup\operatorname{ran}(j)\) in \(M\). 
Then \(X^\kappa\subseteq X\) in \(V\),
\(\P_\kappa\in X\) and \(\P_{\text{tail}}\in X[G]\). Furthermore, since \(\theta=j(f)(\kappa)
\in X\), we also have \(V_\theta\in X\). 
Now observe that, since \(X[G]\) is essentially an
ultrapower by a measure on \(\kappa\), \(\P_{\text{tail}}\) has
size \(j(\kappa)\) in \(X[G]\), and GCH holds, \(X[G]\) has at most \(\kappa^+\) many maximal antichains
for \(\P_{\text{tail}}\), counted in \(V[G]\). In addition, as \(X[G]^\kappa\subseteq X[G]\)
in \(V[G]\) and \(\P_{\text{tail}}\) is \(\leq\kappa\)-closed in \(X[G]\), we may
diagonalize against these maximal antichains and build a \(X[G]\)-generic \(G_{\text{tail}}^*
\subseteq\P_{\text{tail}}\cap X[G]\). Since we put \(A,V_\theta\in X[G]\), we may
also ensure in this construction that the length \(\beth_\theta\) initial segment of
\(G_\kappa\) codes \(A\).

To finish the proof we show that \(G^*_{\text{tail}}\) generates an
\(M[G]\)-generic filter \(G_{\text{tail}}\). To see this let \(D\in M[G]\) be an open
dense subset of \(\P_{\text{tail}}\). Choosing a name \(\dot{D}\in M\) for \(D\),
we can find a function \(\vec{D}\) and an \(a\in V_\theta\) such that \(j(\vec{D})(a)=
\dot{D}\). We now let \(\overline{D}=\bigcap_b j(\vec{D}(b)^G\) where the intersection
runs over those \(b\in V_\theta\) for which \(j(\vec{D}(b)^G\) is a dense open subset of
\(\P_{\text{tail}}\). Clearly \(\overline{D}\subseteq D\) and \(\overline{D}\) is a dense
subset of \(\P_{\text{tail}}\), as this poset is \(\leq\beth_\theta\)-closed in \(M[G]\).
But now notice that
\(\overline{D}\) was defined using only parameters in \(X[G]\) (recall that we had
\(V_\theta\in X\)). Thus, since \(X[G]\prec M[G]\), we get \(\overline{D}\in X[G]\).
The filter \(G^*_{\text{tail}}\) being \(X[G]\)-generic, the intersection
\(G^*_{\text{tail}}\cap \overline{D}\) is nonempty, which shows that \(G^*_{\text{tail}}\)
meets \(D\) and generates an \(M[G]\)-generic filter.

We may thus lift the embedding \(j\) to \(j\colon V[G]\to M[j(G)]\), where
\(j(G)=G*G_{\text{tail}}\). It now follows from lemma~\ref{lemma:NamesInVAlpha}
that the lift of \(j\)
witnesses the \(\theta\)-strongness of \(\kappa\) in \(V[G]\). Finally, we have arranged
matters to ensure that \(j(\ell)(\kappa)=A\), witnessing the guessing property.
\end{proof}

\begin{corollary}
Let \(\kappa\) be \(\theta\)-strong with \(\kappa+2\leq\theta\). 
If \(\theta\) is either a successor ordinal or \(\cf(\theta)\geq\kappa^+\) then
there is a forcing extension in which \(\jLdthetastr_{\kappa,2^\kappa}\) holds.
\end{corollary}
%
Of course, the question of the necessity of the hypotheses of 
proposition~\ref{prop:ThetaStrJLDiamond} remains very attractive.

\begin{question}
Suppose there is a \(\theta\)-strongness Laver function for \(\kappa\) (with \(\theta\)
possibly being a limit of low cofinality).
Is there a \(\theta\)-strongness joint Laver sequence of length \(\kappa\)? Or even of length 
\(\omega\)?
\end{question}

We give a partial answer to this question.
In contrast to the supercompact case, some restrictions are in fact necessary to allow for the
existence of joint Laver sequences for \(\theta\)-strong cardinals. 
The existence of even the
shortest of such sequences can surpass the existence of a \(\theta\)-strong cardinal in
consistency strength.

To give a better lower bound on the consistency strength required, we introduce
a notion of Mitchell rank for \(\theta\)-strong cardinals, inspired by 
Carmody~\cite{Carmody2015:Thesis}.

\begin{definition}
Let \(\kappa\) be a cardinal and \(\theta\) an ordinal. 
The \emph{\(\theta\)-strongness Mitchell order} \(\triangleleft\) for \(\kappa\) is defined 
on the set of \((\kappa,V_\theta)\)-extenders, by letting \(E\triangleleft F\) if
\(E\) is an element of the (transitive collapse of the) ultrapower of \(V\) by \(F\).
\end{definition}

Unsurprisingly, this Mitchell order has properties analogous to those of the usual
Mitchell order on normal measures on \(\kappa\) or the \(\theta\)-supercompactness Mitchell
order on normal fine measures on
\(\mathcal{P}_\kappa\theta\), as studied by Carmody. In particular, the \(\theta\)-strongness
Mitchell order is well-founded and gives rise to a notion of a
\emph{\(\theta\)-strongness Mitchell rank}. Having \(\theta\)-strongness Mitchell
rank at least 2 implies that many cardinals below \(\kappa\) have reflected
versions of \(\theta\)-strongness; for example, if \(\kappa\) has 
\((\kappa+\omega)\)-strongness Mitchell rank at least 2, then there are stationarily
many cardinals \(\lambda<\kappa\) which are \((\lambda+\omega)\)-strong (and much more is true).

We should mention a bound on the \(\theta\)-strongness Mitchel rank of a cardinal \(\kappa\).
If \(j\colon V\to M\) is the ultrapower by a \((\kappa,V_\theta)\)-extender then any
\((\kappa,V_\theta)\)-extenders in \(M\) appear in \(V_{j(\kappa)}^M\). It follows
that these extenders are represented by a function \(f\colon V_\kappa\to V_\kappa\) and
a seed \(a\in V_\theta\). In particular, there are at most \(\beth_\theta\)
many such extenders, counted in \(V\). Any given extender therefore has at most
\(\beth_\theta\) many predecessors in the Mitchell order, so the highest possible
\(\theta\)-strongness Mitchell rank of \(\kappa\) is \(\beth_\theta^+\).

\begin{theorem}
Let \(\kappa\) be a \(\theta\)-strong cardinal, where \(\theta\) is a limit ordinal,
and \(\cf(\theta)\leq\kappa<\theta\). 
If there is a \(\theta\)-strongness joint Laver sequence for \(\kappa\) of length 
\(\cf(\theta)\)
then \(\kappa\) has maximal \(\theta\)-strongness Mitchell rank. 
\end{theorem}

\begin{proof}
We first show that the existence of a short \(\theta\)-strongness joint Laver sequence 
implies a certain degree
of hypermeasurability for \(\kappa\). Let \(\vec{\ell}=
\langle \ell_\alpha;\alpha<\cf(\theta)\rangle\) be the
joint Laver sequence. If \(\vec{a}=\langle a_\alpha;\alpha<\cf(\theta)\rangle\) 
is any sequence of
targets in \(V_\theta\) there is, by definition, a \(\theta\)-strongness embedding
\(j\colon V\to M\) with critical point \(\kappa\) such that \(j(\ell_\alpha)(\kappa)=a_\alpha\). 
But we can recover \(\vec{\ell}\) from \(j(\vec{\ell})\) as an initial segment, since
\(\vec{\ell}\) is so short. Therefore we actually get the whole sequence \(\vec{a}\in M\),
just by evaluating that initial segment at \(\kappa\).
Now consider any \(a\subseteq V_\theta\). We can resolve
\(a\) into a \(\cf(\theta)\)-sequence of elements \(a_\alpha\) of \(V_\theta\) such that 
\(a=\bigcup_\alpha a_\alpha\).
Our argument then implies that \(a\) is an element of \(M\).

Now let \(E\) be an arbitrary \((\kappa,V_\theta)\)-extender.
Since \(E\) can be represented as a family of measures on \(\kappa\) indexed by \(V_\theta\),
it is coded by a subset of \(V_\theta\) 
(using the coding scheme from lemma~\ref{lemma:FlatCoding}, for example). 
Applying the argument
above, there is a \(\theta\)-strongness embedding \(j\colon V\to M\) with critical point
\(\kappa\) such that \(E\in M\). It follows that \(\kappa\) is \(\theta\)-strong in \(M\),
giving \(\kappa\) nontrivial \(\theta\)-strongness Mitchell rank in \(V\).

The argument in fact yields more: given any collection of at most
\(\beth_\theta\) many \((\kappa,V_\theta)\)-extenders, we can code the whole family
by a subset of \(V_\theta\) and, again, obtain an extender whose ultrapower contains
the entire collection we started with. It follows that, given any family of at most
\(\beth_\theta\) many extenders, we can find an extender which is above all of them
in the \(\theta\)-strongness Mitchell order. Applying this fact inductively now
shows that \(\kappa\) must have maximal \(\theta\)-strongness Mitchell rank.
\end{proof}


Just as in the case of \(\theta\)-supercompactness we can also consider 
\(\theta\)-strongness Laver trees. In view of propositions~\ref{prop:ThetaStrJLDiamond}
and \ref{prop:FullyTreeableThetaSCJLDExist} it is not surprising that again,
for most \(\theta\), a single \(\theta\)-strongness Laver diamond yields a
full Laver tree.

\begin{proposition}
\label{prop:ThetaStrLaverTree}
Suppose \(\kappa\) is \(\theta\)-strong where \(\kappa+2\leq\theta\) and
\(\theta\) is either a successor ordinal or \(\cf(\theta)>2^\kappa\). Then a
\(\theta\)-strongness Laver tree for \(\kappa\) exists if and only if a
\(\theta\)-strongness Laver function for \(\kappa\) does (if and only if\/
\(\jLdthetasc_{\kappa,2^\kappa}\) holds).
\end{proposition}

\begin{proof}
We follow the proof of proposition~\ref{prop:FullyTreeableThetaSCJLDExist}.
Note that, since \(\theta\geq\kappa+2\), we have \(\funcs{\kappa}{2}\in V_\theta\),
so \(V_\theta\) is closed under sequences indexed by \(\funcs{\kappa}{2}\) via the coding
scheme given by lemma~\ref{lemma:FlatCoding}.
If \(\ell\) is a \(\theta\)-strongness Laver function for \(\kappa\) we define
a Laver tree by letting \(D(t)\) be the element with index \(t\) in the sequence
coded by \(\ell(|t|)\), if this makes sense. It is now easy to check that this truly is
a Laver tree: given any sequence of targets \(\vec{a}\) we simply use the Laver function
\(\ell\) to guess it (or, rather, its code), and the embedding \(j\) obtained this way
will witness the guessing property for \(D\).
\end{proof}

\section{Joint diamonds}
\label{sec:JD}
Motivated by the joint Laver sequences of the previous sections we now apply
the jointness concept to smaller cardinals. Of course, since we do not have any 
elementary embeddings of the universe with critical point \(\omega_1\), say,
we need a reformulation that will make sense in this setting as well.

As motivation, consider a measurable Laver function \(\ell\) and let \(a\subseteq \kappa\). 
By definition there
is an elementary embedding \(j\colon V\to M\) such that \(j(\ell)(\kappa)=a\). Let \(U\)
be the normal measure on \(\kappa\) derived from this embedding. Observe that \(a\)
is represented in the ultrapower by the function \(f_a(\xi)=a\cap \xi\) and thus, by
Łoś's theorem, we conclude that \(\ell(\xi)=a\cap\xi\) for \(U\)-almost all \(\xi\).
Therefore \(\ell\) is (essentially) a special kind of \(\diamond_\kappa\)-sequence,
guessing not just on stationary sets but on sets of measure one. 
Similarly, if we are dealing with a joint Laver sequence 
\(\langle \ell_\alpha;\alpha<\lambda\rangle\) there is
for every sequence \(\langle a_\alpha;\alpha<\lambda\rangle\) of subsets of \(\kappa\)
a normal measure on \(\kappa\) with respect to which each of the sets
\(\{\xi<\kappa;\ell_\alpha(\xi)=a_\alpha\cap \xi\}\) has measure one.

This understanding of jointness seems amenable to transfer to smaller cardinals.
There are still no normal measures on \(\omega_1\), but perhaps we can weaken
that requirement slightly.

Recall that a filter on a regular cardinal \(\kappa\) is \emph{normal} if it is closed under
diagonal intersections, and \emph{uniform} if it extends the cobounded filter. 
It is a standard result that the club filter on \(\kappa\) is the least normal
uniform filter on \(\kappa\) (in fact, a normal filter is uniform if and only if it extends the
club filter). It follows that any subset of \(\kappa\) contained in a (proper) normal uniform 
filter is stationary. Conversely, if \(S\subseteq\kappa\) is stationary, then it is
easy to check that \(S\), together with the club filter, generates a normal uniform
filter on \(\kappa\). Altogether, we see that a set is stationary if and only if it is contained
in a normal uniform filter. This observation suggest an analogy between Laver functions
and \(\diamond_\kappa\)-sequences: in the same way that Laver functions guess their
targets on positive sets with respect to some large cardinal measure, 
\(\diamond_\kappa\)-sequences guess their targets on positive sets with respect to some normal uniform
filter. Extending the analogy, in the same way that a joint Laver sequence is
a collection of Laver functions that guess sequences of targets on positive sets
with respect to a \emph{common} large cardinal measure (corresponding to the single
embedding \(j\)), a collection of \(\diamond_\kappa\)-sequences will be joint
if they guess sequences of targets on positive sets with respect to a common
normal uniform filter.

We will adopt the following terminology: if \(\kappa\) is a cardinal then a
\emph{\(\kappa\)-list} is a function \(d\colon \kappa\to\mathcal{P}(\kappa)\) with
\(d(\alpha)\subseteq\alpha\).

\begin{definition}
Let \(\kappa\) be an uncountable regular cardinal. A \(\jd_{\kappa,\lambda}\)-sequence is a
sequence \(\vec{d}=\langle d_\alpha;\alpha<\lambda\rangle\) of \(\kappa\)-lists 
such that for every sequence \(\langle a_\alpha;\alpha<\lambda\rangle\) of subsets of 
\(\kappa\) there is a (proper) normal uniform filter \(\mathcal{F}\) on \(\kappa\) 
such that for every \(\alpha\) the \emph{guessing set}
\(S_\alpha=S(d_\alpha,a_\alpha)=\{\xi<\kappa;d_\alpha(\xi)=a_\alpha\cap \xi\}\) 
is in \(\mathcal{F}\).
\end{definition}

An alternative, apparently simpler attempt at defining jointness would be to require
that all the \(\kappa\)-lists in the sequence guess their respective targets on
the same stationary set. However, a straightforward diagonalization argument shows that,
with this definition, there are no \(\jd_{\kappa,\kappa}\)-sequences at all. 
Upon reflection, the definition we gave above corresponds more closely to the one in the
case of Laver diamonds and, hopefully, is not vacuous.

We will not use the following proposition going forward, but it serves to give
another parallel between \(\diamond_\kappa\)-sequences and Laver diamonds. It turns out that
\(\diamond_\kappa\)-sequences can be seen as Laver functions, except that they work with
generic elementary embeddings.\footnote{In the terminology introduced by Gitman and
Schindler in~\cite{GitmanSchindler:VirtualLargeCardinals}, we could say that 
\(\diamond_\kappa\)-sequences are \emph{virtual} Laver functions.}

\begin{proposition}
\label{prop:JDisGenericLaver}
Let \(\kappa\) be an uncountable regular cardinal and let \(d\) be a \(\kappa\)-list.
Then \(d\) is a \(\diamond_\kappa\)-sequence if and only if there is, for any \(a\subset\kappa\),
a generic elementary embedding \(j\colon V\to M\) with critical point \(\kappa\)
and \(M\) wellfounded up to \(\kappa+1\) such that \(j(d)(\kappa)=a\).
\end{proposition}

\begin{proof}
Suppose \(d\) is a \(\diamond_\kappa\)-sequence and fix a target \(a\subseteq\kappa\).
Let \(S(d,a)=\{\xi<\kappa;d(\xi)=a\cap\xi\}\) be the guessing set. By our discussion above
there is a normal uniform filter \(\mathcal{F}\) on \(\kappa\) with \(S\in\mathcal{F}\).
Let \(G\) be a generic ultrafilter extending \(\mathcal{F}\) and \(j\colon V\to M\)
the generic ultrapower by \(G\). Then \(M\) is wellfounded up to \(\kappa^+\) and
\(\kappa=[\mathrm{id}]_G\). Since \(S\in G\), Łoś's theorem now implies that
\(j(d)(\kappa)=a\).

Conversely, fix a target \(a\subseteq\kappa\) and suppose that there is a
generic embedding \(j\) with the above properties. We can replace \(j\) with the
induced normal ultrapower embedding and let \(U\) be the derived ultrafilter in
the extension. Since \(j(d)(\kappa)=a\) it follows that \(S(d,a)\in U\). But since
\(U\) extends the club filter, \(S(d,a)\) must be stationary.
\end{proof}

Similarly to the above proposition, a sequence of \(\kappa\)-lists is joint
if they can guess any sequence of targets via a single generic elementary embedding.

We would now like to find a ``bottom up'' criterion deciding when a collection of subsets of 
\(\kappa\) (namely, some guessing sets) is contained in a normal uniform filter.
The following key lemma gives such a criterion, which is completely analogous to the
finite intersection property characterizing containment in a filter. 

\begin{definition}
Let \(\kappa\) be an uncountable regular cardinal. A family \(\mathcal{A}\subseteq
\mathcal{P}(\kappa)\) has the \emph{diagonal intersection property} if
for any function \(f\colon\kappa\to\mathcal{A}\) the diagonal intersection
\(\diag_{\alpha<\kappa}f(\alpha)\) is stationary.
\end{definition}

\begin{lemma}
\label{lemma:DiagonalIntersectionProperty}
Let \(\kappa\) be uncountable and regular and let \(\mathcal{A}\subseteq\mathcal{P}(\kappa)\).
The family \(\mathcal{A}\) is contained in a normal uniform filter on \(\kappa\) if and only if
\(\mathcal{A}\) satisfies the diagonal intersection property.
\end{lemma}

\begin{proof}
The forward direction is clear, so let us focus on the converse. Consider the family
of sets
\[
E=\Bigl\{C\cap \diag_{\alpha<\kappa}f(\alpha);f\in\funcs{\kappa}{\mathcal{A}}, 
C\subseteq\kappa\text{ club}\Bigr\}
\]
We claim that \(E\) is directed under diagonal intersections: any diagonal intersection of
\(\kappa\) many elements of \(E\) includes another element of \(E\). To see this take
\(C_\alpha\cap \diag_{\beta<\kappa}f_\alpha(\beta)\in E\) for \(\alpha<\kappa\).
Let \(\langle\cdot,\cdot\rangle\) be a pairing function and define 
\(F\colon \kappa\to\lambda\) by \(F(\langle\alpha,\beta\rangle)=f_\alpha(\beta)\).
A calculation then shows that
\[
\diag_{\alpha<\kappa}(C_\alpha\cap\diag_{\beta<\kappa}f_\alpha(\beta))
=\diag_\alpha C_\alpha\cap\diag_\alpha\diag_\beta f_\alpha(\beta)
\supseteq \diag_\alpha C_\alpha\cap D\cap \diag_\alpha F(\alpha)
\]
where \(D\) is the club of closure points of the pairing function.

It follows that closing \(E\) under supersets yields a normal uniform filter on \(\kappa\).
By considering constant functions \(f\) we also see that every \(a\in \mathcal{A}\) is in 
this filter.
\end{proof}

Lemma~\ref{lemma:DiagonalIntersectionProperty} will be the crucial tool for verifying
\(\jd_{\kappa,\lambda}\). More specifically, we shall often apply the following corollary.

\begin{corollary}
\label{cor:JDIffShortSubsequenceJD}
A sequence \(\vec{d}=\langle d_\alpha;\alpha<\lambda\rangle\)
is a \(\jd_{\kappa,\lambda}\)-sequence if and only if every subsequence of length \(\kappa\) is a
\(\jd_{\kappa,\kappa}\)-sequence.
\end{corollary}

\begin{proof}
The forward implication is obvious; let us check the converse. Let
\(\vec{a}=\langle a_\alpha;\alpha<\lambda\rangle\) be a sequence of targets and let
\(S_\alpha\) be the corresponding guessing sets. By 
lemma~\ref{lemma:DiagonalIntersectionProperty} we need to check that the
family \(\mathcal{S}=\{S_\alpha;\alpha<\lambda\}\)
satisfies the diagonal intersection property. So fix a function 
\(f\colon \kappa\to\mathcal{S}\) and let \(r=\{\alpha;S_\alpha\in f[\kappa]\}\).
By our assumption \(\vec{d}\rest r\) is a \(\jd_{\kappa,|r|}\)-sequence, so
\(f[\kappa]\) is contained in a normal uniform filter and, in particular,
\(\diag_\alpha f(\alpha)\) is stationary.
\end{proof}

This characterization leads to fundamental differences between joint diamonds and
joint Laver diamonds. While the definition of joint diamonds was inspired by large cardinal
phenomena, the absence of a suitable analogue of the diagonal intersection property in the
large cardinal setting provides for some very surprising results. 

\begin{definition}
Let \(\kappa\) be an uncountable regular cardinal. A \emph{\(\diamond_\kappa\)-tree} is a
function \(D\colon \funcs{<\kappa}{2}\to \mathcal{P}(\kappa)\) such that for any sequence
\(\langle a_s;s\in\funcs{\kappa}{2}\rangle\) of subsets of \(\kappa\)
there is a (proper) normal uniform filter on \(\kappa\) containing all the guessing sets
\(S_s=S(D,a_s)=\{\xi<\kappa;D(s\rest\xi)=a_s\cap\xi\}\).
\end{definition}

This definition clearly imitates the definition of Laver trees. We also have a
correspondence in the style of proposition~\ref{prop:JDisGenericLaver}: a
\(\diamond_\kappa\)-tree acts like a Laver tree for \(\kappa\) using
generic elementary embeddings.

The following theorem, the main result of this section, shows that, in complete contrast
to our experience with joint Laver diamonds, \(\diamond_\kappa\) already implies all of
its stronger, joint versions.

\begin{theorem}
\label{thm:DiamondEquivalentToJD}
Let \(\kappa\) be an uncountable regular cardinal.
The following are equivalent:
\begin{enumerate}
\item \(\diamond_\kappa\)
\item \(\jd_{\kappa,\kappa}\)
\item \(\jd_{\kappa,2^\kappa}\)
\item there exists a \(\diamond_\kappa\)-tree.
\end{enumerate}
\end{theorem}

\begin{proof}
For the implication \((1)\Longrightarrow (2)\), let 
\(d\colon\kappa\to\mathcal{P}(\kappa)\) be a \(\diamond_\kappa\)-sequence and fix
a bijection \(f\colon \kappa\to\kappa\times\kappa\). Define
\[
d_\alpha(\xi)= (f[d(\xi)])_\alpha\cap\alpha =
\{\eta<\alpha; (\alpha,\eta)\in f[d(\xi)]\}
\]
We claim that \(\langle d_\alpha;\alpha<\kappa\rangle\)
is a \(\jd_{\kappa,\kappa}\)-sequence.

To see this take a sequence of targets \(\langle a_\alpha;\alpha<\lambda\rangle\) 
and let \(S_\alpha=\{\xi<\kappa;d_\alpha(\xi)=a_\alpha\cap\xi\}\) be the guessing sets.
The set
\[
T=\biggl\{
\xi<\kappa; f^{-1}\biggl[\bigcup_{\alpha<\kappa}\{\alpha\}\times a_\alpha\biggr]\cap\xi 
= d(\xi) \biggr\}
\]
is stationary in \(\kappa\). Let \(\mathcal{F}\) be the filter generated by the club filter 
on \(\kappa\)
together with \(T\). This is clearly a proper filter and, since a set \(Y\) is 
\(\mathcal{F}\)-positive if and only if \(Y\cap T\) is stationary, 
Fodor's lemma immediately implies that it is also normal. 
We claim that we have \(S_\alpha\in \mathcal{F}\) for all \(\alpha<\kappa\),
so that \(\mathcal{F}\) witnesses the defining property of a 
\(\jd_{\kappa,\kappa}\)-sequence.
Since \(f[\xi]=\xi\times\xi\) for club many 
\(\xi<\kappa\), the set
\[
T'= \biggl\{
\xi<\kappa; d(\xi)=f^{-1}\biggl[
\bigcup_{\alpha<\kappa}\{\alpha\}\times a_\alpha\biggr]\cap f^{-1}[\xi\times\xi]
\biggr\}
\]
is just the intersection of \(T\) with some club and is therefore in \(\mathcal{F}\). 
But now observe that
\begin{align*}
T' &= \biggl\{
\xi<\kappa; d(\xi)=f^{-1}\biggl[
\bigcup_{\alpha<\xi}\{\alpha\}\times(a_\alpha\cap\xi)
\biggr]\biggr\}\\
&= \{\xi<\kappa;\forall\alpha<\xi\colon a_\alpha\cap\xi = (f[d(\xi)])_\alpha=
d_\alpha(\xi)\} 
= \diag_{\alpha<\kappa}S_\alpha
\end{align*}
We see that \(T'\in \mathcal{F}\) is, modulo a bounded set, contained in each \(S_\alpha\) and can thus
conclude, since \(\mathcal{F}\) is uniform, that \(S_\alpha\in F\) for all \(\alpha<\kappa\).

Instead of proving \((2)\Longrightarrow (3)\) it will be easier to show 
\((2)\Longrightarrow (4)\) directly. Since the implications 
\((4)\Longrightarrow (3)\Longrightarrow (1)\) are obvious, this will finish the proof.

Fix a \(\jd_{\kappa,\kappa}\)-sequence \(\vec{d}\).
We proceed to construct the \(\diamond_\kappa\)-tree \(D\) level by level;
in fact the only meaningful work will take place at limit levels.
At a limit stage \(\gamma<\kappa\) we shall let the first \(\gamma\) many 
\(\diamond_\kappa\)-sequences
anticipate the labels and their positions. 
Concretely, consider the sets \(d_\alpha(\gamma)\) for
\(\alpha<\gamma\). For each \(\alpha\) we interpret \(d_{2\alpha}(\gamma)\) as a node on
the \(\gamma\)-th level of \(\funcs{<\kappa}{2}\) and let \(D(d_{2\alpha}(\gamma))=
d_{2\alpha+1}(\gamma)\), provided that there is no interference between the different
\(\diamond_\kappa\)-sequences. If it should happen that for some \(\alpha\neq \beta\)
we get \(d_{2\alpha}(\gamma)=d_{2\beta}(\gamma)\) but 
\(d_{2\alpha+1}(\gamma)\neq d_{2\beta+1}(\gamma)\)
we scrap the whole level and move on with the construction
higher in the tree.
At the end we extend \(D\) to be defined on the nodes of \(\funcs{<\kappa}{2}\) that
were skipped along the way in any way we like.

We claim that the function \(D\) thus constructed is a \(\diamond_\kappa\)-tree. 
To check this let us fix a sequence of targets 
\(\vec{a}=\langle a_s;s\in\funcs{\kappa}{2}\rangle\)
and let \(S_s\) be the guessing sets.
By lemma~\ref{lemma:DiagonalIntersectionProperty} it now suffices to check that
\(\diag_{\alpha<\kappa}S_{s_\alpha}\) is stationary for any sequence of branches 
\(\langle s_\alpha;\alpha<\kappa\rangle\).

For \(\alpha<\kappa\) let 
\(T_{2\alpha}=\{\xi;s_\alpha^{-1}[\{1\}]\cap\xi=d_{2\alpha}(\xi)\}\)
and \(T_{2\alpha+1}=\{\xi;A^{s_\alpha}\cap\xi=d_{2\alpha+1}(\xi)\}\).
Since our construction was guided by a \(\jd_{\kappa,\kappa}\)-sequence, 
there is a normal uniform
filter on \(\kappa\) which contains every \(T_\alpha\). In particular,
\(T=\diag_{\alpha<\kappa}T_\alpha\) is stationary. By a simple bootstrapping argument there
is a club \(C\) of limit ordinals \(\gamma\) such that all \(s_\alpha\rest\gamma\) for 
\(\alpha<\gamma\) are distinct. Let \(\gamma\in C\cap T\). We now have
\(s_\alpha^{-1}[\{1\}]\cap\gamma=d_{2\alpha}(\gamma)\) and
\(a_{s_\alpha}\cap\gamma=d_{2\alpha+1}(\gamma)\) for all \(\alpha<\gamma\).
But this means precisely that the construction of \(D\) goes through at level \(\gamma\)
and that \(\gamma\in\bigcap_{\alpha<\gamma}S_{s_\alpha}\) and it follows that
\(\diag_{\alpha<\kappa}S_{s_\alpha}\) is stationary.
\end{proof}

We can again consider the treeability of joint diamond sequences, as we did in
definition~\ref{def:treeability}. We get the following analogue of 
corollary~\ref{cor:NontreeableSCJLD}.

\begin{theorem}
\label{thm:NontreeableJDConsistent}
If \(\kappa\) is an uncountable regular cardinal and GCH holds
then after forcing with \(\Add(\kappa,2^\kappa)\) there is a nontreeable 
\(\jd_{\kappa,2^\kappa}\)-sequence.
\end{theorem}

\begin{proof}
Let \(\P=\Add(\kappa,2^\kappa)\) and \(G\subseteq\P\) generic; 
we refer to the \(\alpha\)-th subset added by \(G\) as \(G_\alpha\). 
We will show that the generic \(G\), seen as a sequence of
\(2^\kappa\) many \(\diamond_\kappa\)-sequences in the usual way, is a nontreeable 
\(\jd_{\kappa,2^\kappa}\)-sequence.

Showing that \(G\) is a \(\jd_{\kappa,2^\kappa}\)-sequence requires only minor modifications to
the usual proof that a Cohen subset of \(\kappa\) codes a \(\diamond_\kappa\)-sequence.
Thus, we view each \(G_\alpha\) as a sequence, defined on \(\kappa\), with
\(G_\alpha(\xi)\subseteq\xi\). Fix a sequence \(\langle \dot{a}_\alpha;\alpha<2^\kappa
\rangle\) of names for subsets of \(\kappa\), a name \(\dot{f}\) for a function
from \(\kappa\) to \(2^\kappa\) and a name \(\dot{C}\) for a club in \(\kappa\)
as well as a condition \(p\in\P\). We will find a condition \(q\leq p\)
forcing that \(\dot{C}\cap\diag_{\alpha<\kappa}S_{\dot{f}(\alpha)}\) is nonempty, where
\(S_{\dot{f}(\alpha)}\) names the set \(\{\xi<\kappa;\dot{a}_{\dot{f}(\alpha)}\cap\xi=
G_\alpha(\xi)\}\);
this will show that \(G\) codes a \(\jd_{\kappa,2^\kappa}\)-sequence by 
lemma~\ref{lemma:DiagonalIntersectionProperty}.

We build the condition \(q\) in \(\omega\) many steps. To start with, let \(p_0=p\) and
let \(\gamma_0\) be an ordinal such that \(\dom(p_0)\subseteq 2^\kappa\times\gamma_0\).
We now inductively find ordinals \(\gamma_n\), sets \(B^\alpha_n\subseteq\gamma_n\),
functions \(f_n\) and a descending sequence of
conditions \(p_n\) satisfying \(\dom(p_n)\subseteq 2^\kappa\times\gamma_n\) and
\(p_{n+1}\forces \gamma_n\in\dot{C}\) as well as 
\(p_{n+1}\forces \dot{f}\rest\gamma_n=f_n\) and \(p_{n+1}\forces \dot{a}_{f_n(\alpha)}=
B^\alpha_n\) for \(\alpha<\gamma_n\).
Let \(\gamma=\sup_n \gamma_n\) and
\(p_\omega=\bigcup_n p_n\) and \(f_\omega=\bigcup_n f_n\) and \(B^\alpha_\omega=
\bigcup_n B_n^\alpha\). 
The construction of these ensures that \(\dom(p_\omega)
\subseteq 2^\kappa\times\gamma\) and \(p_\omega\) forces that
\(\dot{f}\rest\gamma=f_\omega\) and \(\dot{a}_{\dot{f}(\alpha)}\cap\gamma=B^\alpha_\omega\)
for \(\alpha<\gamma\) as well as \(\gamma\in\dot{C}\). 
To obtain the final condition \(q\) we now simply extend \(p_\omega\)
by placing the code of \(B^\alpha_\omega\) on top of the \(f(\alpha)\)-th column
for all \(\alpha<\gamma\). It now follows immediately that \(q\forces\gamma
\in\dot{C}\cap\diag_{\alpha<\kappa}S_{\dot{f}(\alpha)}\).

It remains to show that the generic \(\jd_{\kappa,2^\kappa}\)-sequence is not treeable.
This follows from lemma~\ref{lemma:GenericNotTreeable}. 
\end{proof}

In the case of Laver diamonds we were able to produce models with quite long
joint Laver sequences but no Laver trees simply on consistency strength grounds (see
theorem~\ref{thm:SeparateLongAndTreeableSCJLD}). In other words, we have models where
there are long joint Laver sequences, but none of them are treeable. The situation seems
different for ordinary diamonds, as theorem~\ref{thm:DiamondEquivalentToJD} tell us
that treeable joint diamond sequences exist as soon as a single diamond sequence
exists. While theorem~\ref{thm:NontreeableJDConsistent} shows that it is at least
consistent that there are nontreeable such sequences, we should ask whether this is 
simply always the case.

\begin{question}[open]
Is it consistent for a fixed \(\kappa\) that every \(\jd_{\kappa,2^\kappa}\)-sequence is
treeable? Is it consistent that all \(\jd_{\kappa,2^\kappa}\)-sequences are treeable for
all \(\kappa\)?
\end{question}

\renewcommand{\c}{\mathfrak{c}}
\chapter{The grounded Martin's axiom}
\label{chap:grma}

The standard Solovay-Tennenbaum proof of the consistency of Martin's axiom with
a large continuum starts by choosing a suitable cardinal \(\kappa\) and then
proceeds in an iteration of length \(\kappa\) by forcing with ccc posets of size
less than \(\kappa\), and not just those in the ground model but also those
arising in the intermediate extensions. To ensure that all of the potential ccc posets 
are considered, some bookkeeping device is usually employed.

Consider now the following reorganization of the argument. Instead of iterating
for \(\kappa\) many steps we build a length \(\kappa^2\) iteration by first dealing
with the \(\kappa\) many small posets in the ground model, then the small posets in
that extension, and so on. The full length \(\kappa^2\) iteration can be seen as a length 
\(\kappa\) iteration, all of whose iterands are themselves length \(\kappa\) iterations.

In view of this reformulation, we can ask what happens if we halt this
construction after forcing with the first length \(\kappa\) iteration, when we have, in 
effect, ensured that Martin's axiom holds for posets from the ground model. What 
combinatorial consequences of Martin's axiom follow already from this weaker principle? 
We aim in this chapter to answer these questions (at least partially).

\section{Forcing the grounded Martin's axiom}

\begin{definition}
%
The \emph{grounded Martin's axiom} (\grma) asserts that \(V\) is a ccc forcing
extension of some ground model \(W\) and \(V\) satisfies the conclusion of
Martin's axiom for posets \(\Q\in W\) which are still ccc in \(V\).
\end{definition}

To be clear, in the definition we require that \(V\) have \(\mathcal{D}\)-generic
filters for any family \(\mathcal{D}\in V\) of fewer than \(\c^V\) dense
subsets of \(\Q\).
We should also note that, while the given definition is second-order, \grma is in fact 
first-order expressible, using the result of Reitz \cite{Reitz2007:GroundAxiom} that the 
ground models of the universe are uniformly definable.

%
%

If Martin's axiom holds, we may simply take \(W=V\) in the definition,
which shows that the grounded Martin's axiom is implied by Martin's axiom.
In particular, by simply performing the usual Martin's axiom iteration,
\(\grma+\c=\kappa\) can be forced from any model where \(\kappa\) is regular
and satisfies \(2^{<\kappa}=\kappa\). It will be shown in theorem~\ref{thm:CanonicalgrMA},
however, that the grounded Martin's axiom is strictly weaker than Martin's axiom.
Ultimately we shall see that \grma retains some of the interesting combinatorial
consequences of Martin's axiom (corollary~\ref{cor:grMALargeCharacts}), while 
also being more robust with respect to mild forcing 
(theorems~\ref{thm:CohenPreservesgrMA} and \ref{thm:RandomGivesgrMA}).

As in the case of Martin's axiom, a key property of the grounded Martin's axiom is
that it is equivalent to its restriction to posets of small size.

\begin{lemma}
\label{lemma:grMASmallPosets}
The grounded Martin's axiom is equivalent to its restriction to posets of size less 
than continuum, i.e.\ to the following principle:

{\newlength{\bit}%
\setlength{\bit}{\linewidth}%
\addtolength{\bit}{-5em}%
\vspace{\baselineskip}%
\parbox{\bit}{The universe \(V\) is a ccc forcing extension of some ground
model \(W\) and \(V\) satisfies the conclusion of Martin's axiom for
posets \(\Q\in W\) of size less than \(\c^V\) which are still ccc in \(V\).}%
\hspace*{2em}\textup{(\(\ast\))}%
\vspace*{\baselineskip}}
\end{lemma}

\begin{proof}
Assume \(V\) satisfies (\(\ast\)) and let \(\Q\in W\) be a poset which is ccc in \(V\) and 
\(\D=\set{D_\alpha}{\alpha<\kappa}\in V\)
a family of \(\kappa<\c^V\) many dense subsets of \(\Q\). Let \(V=W[G]\) for some
\(W\)-generic \(G\subseteq \P\) and let \(\dot{D}_\alpha\in W\) be \(\P\)-names for the
\(D_\alpha\). Choose \(\theta\) large enough so that \(\P,\Q\) and all of the 
\(\dot{D}_\alpha\) are in \(H_{\theta^+}^W\).
We can then find an \(X\in W\) of size at most \(\kappa\) such that 
\(X\prec H_{\theta^+}^W\) and \(X\) contains \(\P,\Q\) and the 
\(\dot{D}_\alpha\).

Now let \[X[G]=\set{\tau^G}{\text{\(\tau \in X\) is a \(\P\)-name}}\in W[G]\]
We can verify the Tarski-Vaught criterion to show that \(X[G]\) is an
elementary substructure of \(H_{\theta^+}^W[G]=H_{\theta^+}^{W[G]}\). Specifically,
suppose that \(H_{\theta^+}^W[G]\models\exists x\colon\varphi(x,\tau^G)\)
for some \(\tau\in X\). Let \(S\) be the set of conditions \(p\in\P\) which force
\(\exists x\colon\varphi(x,\tau)\). Since \(S\) is definable from the parameters
\(\P\) and \(\tau\) we get \(S\in X\). Let \(A\in X\) be an antichain, maximal
among those contained in \(S\). By mixing over \(A\) we can obtain a name
\(\sigma\in X\) such that \(p\forces \varphi(\sigma,\tau)\) for
any \(p\in A\) and it follows that
\(H_{\theta^+}^W[G]\models\varphi(\sigma^G,\tau^G)\), which completes the verification.

Now let \(\Q^*=\Q\cap X[G]\) and \(D_\alpha^*=D_\alpha\cap X[G]\). Then
\(\Q^*\prec \Q\) and it follows that \(\Q^*\) is ccc (in \(W[G]\)) and that 
\(D_\alpha^*\) is dense in \(\Q^*\) for any \(\alpha<\kappa\). Furthermore, \(\Q^*\) has
size at most \(\kappa\). Finally, since \(\P\) is ccc,
the filter \(G\) is \(X\)-generic and so \(\Q^*=\Q\cap X[G]=\Q\cap X\) is an element
of \(W\). If we now apply \((\ast)\) to \(\Q^*\) and \(\mathcal{D}^*=\set{D_\alpha^*}{\alpha<\kappa}\), we find in \(W[G]\) a filter \(H\subseteq \Q^*\) intersecting every \(D_\alpha^*\), and
thus every \(D_\alpha\). Thus \(H\) generates a \(\mathcal{D}\)-generic filter on \(\Q\).
\end{proof}

%
%

The reader has likely noticed that the proof of 
lemma~\ref{lemma:grMASmallPosets} is somewhat more involved than the proof of
the analogous result for Martin's axiom. The argument there hinges on the
straightforward observation that elementary subposets of ccc posets are themselves
ccc. While that remains true in our setting, of course, matters are made more difficult 
since we require that all of our posets come from a ground model that may not
contain the dense sets under consideration.
It is therefore not at all clear that taking
appropriate elementary subposets will land us in the ground model and a slightly more
elaborate argument is needed.

Let us point out a deficiency in the definition of \grma. As we have described 
it, the principle posits the existence of a ground model for the universe, a 
considerable global assumption. On the other hand, lemma~\ref{lemma:grMASmallPosets} 
suggests that the 
operative part of the axiom is, much like Martin's axiom, a statement about \(H_\c\). 
This discrepancy allows for some undesirable phenomena. For example, 
Reitz~\cite{Reitz2007:GroundAxiom} shows that it is possible to perform arbitrarily 
closed class forcing over a given model and obtain a model of the ground axiom,
the assertion that the universe is not a nontrivial set forcing extension over any
ground model at all. This 
implies that there are models which have the same \(H_\c\) as a model of the
grounded Martin's axiom but which fail to satisfy it simply because they have
no nontrivial ground models at all. 
To avoid this situation we can weaken the definition of \grma in a technical way.

\begin{definition}
The \emph{local grounded Martin's axiom} asserts that there are a cardinal
\(\kappa\geq\c\) and a transitive \(\mathrm{ZFC}^-\) model 
\(M\subseteq H_{\kappa^+}\) such that
\(H_{\kappa^+}\) is a ccc forcing extension of \(M\) and \(V\) satisfies the
conclusion of Martin's axiom for posets \(\Q\in M\) which are still ccc in \(V\).
\end{definition}

Of course, if the grounded Martin's axiom holds, over the ground model \(W\) via
the forcing notion \(\P\), then its 
local version holds as well. We can simply take \(\kappa\) to be large enough
so that \(M=H_{\kappa^+}^W\) contains \(\P\) and that \(M[G]=H_{\kappa^+}^V\).
One should view the local version of the axiom as capturing all of the relevant
combinatorial effects of \grma (which, as we have seen, only involve \(H_\c\)),
while disentangling it from the structure of the universe higher up.

We now aim to give a model where the Martin's axiom fails but the grounded version
holds. The idea is to imitate the Solovay-Tennenbaum argument, but to only use ground
model posets in the iteration. While it is then relatively clear that \grma will hold
in the extension, a further argument is needed to see that \ma itself fails. 
The key will be a kind of product analysis, given in the next few lemmas. We will show
that the iteration of ground model posets, while not exactly
a product, is close enough to a product to prevent Martin's axiom from holding in
the final extension by a result of Roitman. An extended version of this argument will
also yield the consistency of \grma with a singular continuum.

\begin{lemma}
\label{lemma:TwoStepIterationDecidesccc}
Let \(\Q_0\) and \(\R\) be posets and \(\tau\) a \(\Q_0\)-name for a poset
such that \(\Q_0*\tau\) is ccc and
\(\Q_0\forces\textup{``if \(\check{\R}\) is ccc then \(\tau=\check{\R}\)''}\).
Furthermore, suppose that \(\Q_0*\tau\forces\textup{``\(\check{\Q}_0\) is ccc''}\). Then either
\(\Q_0\forces\textup{``\(\check{\R}\) is ccc''}\) or \(\Q_0\forces\textup{``\(\check{\R}\) is not ccc''}\).
\end{lemma}

\begin{proof}
Suppose the conclusion fails, so that there are conditions \(q_0,q_1\in\Q_0\) which
force \(\check{\R}\) to either be or not be ccc, respectively. It follows that
\(\Q_0\rest q_1\times\R\) is not ccc. Switching the factors, there must be a condition
\(r\in\R\) forcing that \(\check{\Q}_0\rest\check{q}_1\) is not ccc. Now let \(G*H\)
be generic for \(\Q_0*\tau\) with \((q_0,\check{r})\in G*H\) (note that 
\((q_0,\check{r})\) is really a condition since \(q_0\forces\tau=\check{\R}\)).
Consider the extension \(V[G*H]\). On the one hand \(\Q_0\) must be ccc there, since
this was one of the hypotheses of our statement, but on the other hand \(\Q_0\rest q_1\) 
is not ccc there since \(r\in H\) forces this.
\end{proof}

\begin{lemma}
\label{lemma:FSIterationFactorsAsProduct}
Let \(\P=\langle \P_\alpha,\tau_\alpha;\alpha<\gamma\rangle\), with \(\gamma>0\), be a finite-support ccc
iteration such that for each \(\alpha\) there is some poset \(\Q_\alpha\) for which
\[
\P_\alpha\forces\textup{``if \(\check{\Q}_\alpha\) is ccc then \(\tau_\alpha=\check{\Q}_\alpha\) and \(\tau_\alpha\) is trivial otherwise''}
\]
Furthermore assume that \(\P\forces\textup{``\(\check{\Q}_0\) is ccc''}\). Then
\(\P\) is forcing equivalent to the product \(\Q_0\times\overline{\P}\) for some
poset \(\overline{\P}\).
\end{lemma}

Before we give the (technical) proof, let us provide some intuition for this lemma.
We can define the iteration \(\overline{\P}\) in the same way as \(\P\) (i.e.\ using the
same \(\Q_\alpha\)) but skipping the first step of forcing.
The idea is that, by lemma~\ref{lemma:TwoStepIterationDecidesccc}, the posets which
appear in the iteration \(\P\) do not depend on the first stage of forcing \(\Q_0\).
We thus expect that generics \(G\subseteq\P\) will correspond exactly to generics
\(H\times\overline{G}\subseteq \Q_0\times\overline{\P}\), since the first stage \(G_0\)
of the generic \(G\) does not affect the choice of posets in the rest of
the iteration.
 
\begin{proof}
We show the lemma by induction on \(\gamma\), the length of the iteration \(\P\).
In fact, we shall work with a stronger induction hypothesis. Specifically, we shall show
that for each \(\alpha<\gamma\) there is a poset \(\overline{\P}_\alpha\) such that
\(\Q_0\times\overline{\P}_\alpha\) embeds densely into \(\P_\alpha\), that
the \(\overline{\P}_\alpha\) form the initial segments of a finite-support iteration
and that the dense embeddings extend one another. For the
purposes of this proof we shall take all two-step iterations to be in the style of
Kunen, i.e.\ the conditions in \(\P*\tau\) are pairs \((p,\sigma)\) such that
\(p\in\P\) and \(\sigma\in\dom(\tau)\) and \(p\forces\sigma\in\tau\). Furthermore
we shall assume that the \(\tau_\alpha\) are full names. These assumptions make no
difference for the statement of the lemma, but ensure that certain embeddings will in
fact be dense.

Let us start with the base case \(\gamma=2\), when \(\P=\Q_0*\tau_1\). 
By lemma~\ref{lemma:TwoStepIterationDecidesccc} whether or not \(\check{\Q}_1\) is ccc
is decided by every condition in \(\Q_0\). But then, by our assumption on \(\tau_1\),
if \(\Q_0\) forces \(\check{\Q}_1\) to be ccc then \(\tau_1\) is forced to be equal
to \(\check{\Q}_1\) and \(\Q_0\times\Q_1\) embeds densely into \(\P\), and otherwise
\(\tau_1\) is forced to be trivial and \(\Q_0\) embeds densely into \(\P\). Depending
on which is the case, we can thus take either \(\overline{\P}_2=\Q_1\) or 
\(\overline{\P}_2=1\).

For the induction step let us assume that the stronger induction hypothesis holds for 
iterations of length \(\gamma\) and show that it holds for iterations of length 
\(\gamma+1\).
Let us write \(\P=\P_\gamma*\tau_\gamma\). By the induction hypothesis there is
a \(\overline{\P}_\gamma\) such that \(\Q_0\times\overline{\P}_\gamma\) embeds densely
into \(\P_\gamma\).

Before we give the details, let us sketch the string of equivalences that will yield
the desired conclusion. We have
\[
\P\equiv (\Q_0\times\overline{\P}_\gamma)*\bar{\tau}_\gamma \equiv
(\overline{\P}_\gamma\times\Q_0)*\bar{\tau}_\gamma \equiv
\overline{\P}_\gamma*(\check{\Q}_0*\bar{\tau}_\gamma)
\]
Here \(\bar{\tau}_\gamma\) is the \(\Q_0\times\overline{\P}_\gamma\)-name (or
\(\overline{\P}_\gamma\times\Q_0\)-name) resulting from pulling back the 
\(\P_\gamma\)-name \(\tau_\gamma\) along the dense embedding provided by the induction
hypothesis. We will specify what exactly we mean by 
\(\check{\Q}_0*\bar{\tau}_\gamma\) later.

We can apply the base step of the induction to the iteration 
\(\check{\Q}_0*\bar{\tau}_\gamma\)
in \(V^{\overline{\P}_\gamma}\) and obtain a \(\overline{\P}_\gamma\)-name
\(\dot{\R}\) such that \(\check{\Q}_0\times\dot{\R}\) is forced to densely embed into
\(\check{\Q}_0*\bar{\tau}_\gamma\). We can then continue the chain above with
\[
\overline{\P}_\gamma*(\check{\Q}_0*\bar{\tau}_\gamma) \equiv
\overline{\P}_\gamma*(\check{\Q}_0\times\dot{\R}) \equiv
\overline{\P}_\gamma*(\dot{\R}\times\check{\Q}_0) \equiv
\overline{\P}_{\gamma+1}\times\Q_0
\]
where \(\overline{\P}_{\gamma+1}=\overline{\P}_\gamma*\dot{\R}\). While this is 
apparently enough to finish the successor step for the bare statement of the lemma,
we wish to preserve the stronger induction hypothesis, and this requires a bit more work.

We first pick a specific \(\overline{\P}_\gamma\)-name for \(\check{\Q}_0*\bar{\tau}_\gamma\).
Let
\[
\tau=\set{((\check{q}_0,\rho),\bar{p})}{q_0\in\Q_0, \rho\in\dom(\bar{\tau}_\gamma),\bar{p}\in\overline{\P}_\gamma,(\bar{p},q_0)\forces \rho\in\bar{\tau}_\gamma}
\]
Then \(\overline{\P}_\gamma\forces \tau=\check{\Q}_0*\tau_\gamma\). Next we pin down
\(\dot{\R}\). Note that \(\overline{\P}_\gamma\) forces that \(\check{\Q}_0\)
decides whether \(\check{\Q}_\gamma\) is ccc or not by 
lemma~\ref{lemma:TwoStepIterationDecidesccc}. Let 
\(A=A_0\cup A_1\subseteq\overline{\P}_\gamma\) be a maximal antichain such that each
\(\bar{p}\in A_0\) forces \(\check{\Q}_\gamma\) to not be ccc and each \(\bar{p}\in A_1\)
forces it to be ccc. Now let
\[
\dot{\R}=\set{(\check{1},\bar{p})}{\bar{p}\in A_1}\cup \set{(\check{q},\bar{p})}{q\in\Q_\gamma,\bar{p}\in A_0}
\]
Observe that \(\dot{\R}\) has the properties we require of it: it is forced by 
\(\overline{\P}_\gamma\) that \(\dot{\R}=\check{\Q}_\gamma\) if \(\check{\Q}_0\) forces
that \(\check{\Q}_\gamma\) is ccc, and \(\dot{\R}\) is trivial otherwise, and that
\(\check{\Q}_0\times \dot{\R}\) embeds densely into \(\tau\).

Finally, let us define \(\overline{\P}_{\gamma+1}=\overline{\P}_\gamma*\dot{\R}\).
We can now augment the equivalences given above with dense embeddings:

\begin{itemize}
\item The embedding \(\overline{\P}_{\gamma+1}\times\Q_0\hookrightarrow
\overline{\P}_\gamma*(\check{\Q}_0\times\dot{\R})\) is clear.

\item To embed \(\overline{\P}_\gamma*(\check{\Q}_0\times\dot{\R})\) into 
\(\overline{\P}_\gamma*\tau\) we can send \((\bar{p},(\check{q}_0,\rho))\) to \((\bar{p},
(\check{q}_0,\rho'))\) where \(\rho'\) is some element of \(\dom(\bar{\tau}_\gamma)\) for which
\((\bar{p},q_0)\forces \rho=\rho'\) (this is where the fullness of \(\tau_\gamma\) is
needed).

\item With our specific choice of the name \(\tau\) we in fact get an isomorphism
between \(\overline{\P}_\gamma*\tau\) and 
\((\overline{\P}_\gamma\times\Q_0)*\bar{\tau}_\gamma\), given by sending
\((\bar{p},(\check{q}_0,\rho))\) to \(((\bar{p},q_0),\rho)\).

\item The final embedding from \((\overline{\P}_\gamma\times\Q_0)*\bar{\tau}_\gamma\) into 
\(\P\) is given by the induction hypothesis.
\end{itemize}

After composing these embeddings, we notice that the first three steps essentially fixed 
the \(\overline{\P}_\gamma\) part of the condition and the last step fixed the
\(\tau_\gamma\) part. It follows that the embedding 
\(\overline{\P}_{\gamma+1}\times\Q_0\hookrightarrow \P_{\gamma+1}\) we constructed
extends the embedding
\(\overline{\P}_\gamma\times\Q_0\hookrightarrow \P_\gamma\) given by the induction
hypothesis. This completes the successor step of the induction.

We now look at the limit step of the induction. The induction hypothesis gives us
for each \(\alpha<\gamma\) a poset \(\overline{\P}_\alpha\) and a dense embedding
\(\overline{\P}_\alpha\times\Q_0\hookrightarrow \P_\alpha\) and we also know that
the \(\overline{\P}_\alpha\) are the initial segments of a finite-support iteration
and that the dense embeddings extend each other. If we now let \(\overline{\P}_\gamma\)
be the direct limit of the \(\overline{\P}_\alpha\), we can easily find a dense embedding
of \(\overline{\P}_\gamma\times\Q_0\) into \(\P_\gamma\). Specifically, given a condition
\((\bar{p},q)\in\overline{\P}_\gamma\times\Q_0\), we can find an \(\alpha<\gamma\)
such that \(\bar{p}\) is essentially a condition in \(\overline{\P}_\alpha\), since
\(\overline{\P}_\gamma\) is the direct limit of these. Now we can map
\((\bar{p},q)\) using the stage \(\alpha\) dense embedding, landing in
\(\P_\alpha\) and interpreting this as an element of \(\P_\gamma\). This map
is independent of the particular choice of \(\alpha\) since all the dense embeddings
extend one another and it is itself a dense embedding since all the previous stages
were.
\end{proof}

\begin{theorem}
\label{thm:CanonicalgrMA}
Let \(\kappa>\omega_1\) be a cardinal of uncountable cofinality satisfying 
\(2^{<\kappa}=\kappa\). Then there is a ccc forcing extension that satisfies 
\(\grma + \lnot\ma + \c=\kappa\).
\end{theorem}

\begin{proof}
Fix a well-order \(\triangleleft\) of \(H_\kappa\) of length \(\kappa\); 
this can be done since
\(2^{<\kappa}=\kappa\). We can assume, without loss of generality, that the least element
of this order is the poset \(\Add(\omega,1)\). 
We define a length \(\kappa\) finite-support iteration
\(\P\) recursively: at stage \(\alpha\) we shall force with the next poset with respect 
to the order \(\triangleleft\) if that is ccc at that stage and with
trivial forcing otherwise. Let \(G\) be \(\P\)-generic.

Notice that any poset \(\Q\in H_\kappa\) occurs, up to isomorphism, unboundedly often in 
the well-order \(\triangleleft\). Specifically, we can first find an isomorphic copy 
whose universe is a set of ordinals bounded
in \(\kappa\) and then simply move this universe higher and higher up. In particular,
isomorphic copies of Cohen forcing \(\mathrm{Add}(\omega,1)\) appear unboundedly
often. Since these are ccc in a highly robust way (being
countable), they will definitely be forced with in the iteration
\(\P\). Therefore we have at least \(\kappa\) many reals in the extension \(V[G]\). 
Since the forcing is ccc of size \(\kappa\) and \(\kappa\) has uncountable 
cofinality, a nice-name argument shows that the continuum equals \(\kappa\) in the 
extension.

To see that \ma fails in \(V[G]\), notice that \(\P\) is exactly the type of iteration
considered in lemma~\ref{lemma:FSIterationFactorsAsProduct}. The lemma then implies
that \(\P\) is equivalent to \(\overline{\P}\times\Add(\omega,1)\) for some 
\(\overline{\P}\). Therefore the extension
\(V[G]\) is obtained by adding a Cohen real to some intermediate model.
But, as CH fails in the final extension, Roitman has shown 
in~\cite{Roitman1979:AddingRandomOrCohenRealEffectMA} that Martin's axiom
must also fail there.

Finally, we show that the grounded Martin's axiom holds in \(V[G]\) with
\(V\) as the ground model. Before we consider the general case let us look at
the easier situation when \(\kappa\) is regular. 
Thus, let \(\Q\in V\) be a poset of size less than
\(\kappa\) which is ccc in \(V[G]\) and \(\mathcal{D}\in V[G]\) a family of fewer
than \(\kappa\) many dense subsets of \(\Q\). We can assume without loss of 
generality that \(\Q\in H_\kappa\). Code the elements of \(\mathcal{D}\) into a
single set \(D\subseteq \kappa\) of size \(\lambda<\kappa\). Since \(\P\) is ccc, 
the set \(D\) has a nice \(\P\)-name \(\dot{D}\) of size \(\lambda\). Since the iteration 
\(\P\) has finite support, the set \(D\) appears before the end of the iteration.
As we have argued 
before, up to isomorphism, the poset \(\Q\) appears \(\kappa\) many times in the
well-order \(\triangleleft\).
Since \(\Q\) is ccc in \(V[G]\) and \(\P\) is ccc, posets isomorphic to \(\Q\) will
be forced with unboundedly often in the iteration \(\P\) and, therefore, eventually
a \(\mathcal{D}\)-generic will be added for \(\Q\).

If we now allow \(\kappa\) to be singular we run into the problem that the dense sets
in \(\mathcal{D}\) might not appear at any initial stage of the iteration \(\P\).
We solve this issue by using lemma~\ref{lemma:FSIterationFactorsAsProduct} to
factor a suitable copy of \(\Q\) out of the iteration \(\P\) and see it as coming
after the forcing that added \(\mathcal{D}\).

Let \(\Q,\mathcal{D},D,\lambda\) and \(\dot{D}\) be as before. As mentioned, it may
no longer be true that \(\mathcal{D}\) appears at some initial stage of the iteration.
Instead, note that, since \(\dot{D}\) has size \(\lambda\), there are at most
\(\lambda\) many indices \(\alpha\) such that some condition 
appearing in \(\dot{D}\) has a nontrivial \(\alpha\)-th coordinate.
It follows that there is a \(\delta<\kappa\) such that no 
condition appearing in \(\dot{D}\) has a nontrivial \(\delta\)-th coordinate
and the poset considered at stage \(\delta\) is isomorphic to \(\Q\).
Additionally, if we fix a condition \(p\in G\) forcing that
\(\Q\) is ccc in \(V[G]\), we can find such a \(\delta\) beyond the support of \(p\).
Now argue for a moment in \(V[G_\delta]\). 
In this intermediate extension the quotient iteration
\(\P^\delta=\P\rest [\delta,\kappa)\) is of the type considered in 
lemma~\ref{lemma:FSIterationFactorsAsProduct} and, since we chose \(\delta\) beyond
the support of \(p\), we also get that \(\P^\delta\) forces that \(\check{\Q}\) is ccc.
The lemma now implies that \(\P^\delta\) is equivalent to \(\overline{\P}\times\Q\)
for some \(\overline{\P}\). Moving back to \(V\), we can conclude that
\(\P\rest p\) factors as \(\P_\delta\rest p *(\overline{\P}\times\check{\Q})\equiv
(\P_\delta\rest p*\overline{\P})\times\Q\) and obtain the corresponding
generic \((G_\delta*\overline{G})\times H\). Furthermore, the name \(\dot{D}\) is
essentially a \(\P_\delta*\overline{\P}\)-name, since no condition in \(\dot{D}\) has
a nontrivial \(\delta\)-th coordinate. It follows that the set \(D\) appears
already in \(V[G_\delta*\overline{G}]\) and that the final generic
\(H\in V[G_\delta*\overline{G}][H]=V[G]\) is \(\mathcal{D}\)-generic for \(\Q\).
We have thus shown that \((*)\) from lemma~\ref{lemma:grMASmallPosets} holds in
\(V[G]\), which implies that the grounded Martin's axiom also holds.
\end{proof}

We should reflect briefly on the preceding proof.
If we would have been satisfied with obtaining a model with a regular continuum,
the usual techniques would apply. Specifically, if \(\kappa\) were regular then all the 
dense sets in \(\mathcal{D}\) would have appeared by
some stage of the iteration, after which we would have forced with (a poset isomorphic 
to) \(\Q\), yielding the desired \(\mathcal{D}\)-generic. This approach, however, fails
if \(\kappa\) is singular, as small sets might not appear before the end of the 
iteration. 
Lemma~\ref{lemma:FSIterationFactorsAsProduct} was key in resolving this issue, allowing 
us to factor the iteration \(\P\) as a product and seeing the forcing \(\Q\) as happening
at the last stage, after the dense sets had already appeared.
The lemma implies that the iteration factors at any stage where
we considered an absolutely ccc poset (for example, we can factor out any Knaster
poset from the ground model). However, somewhat fortuitously, we can also factor out
any poset to which we might apply the grounded Martin's axiom, at least below
some condition. There is no surrogate for lemma~\ref{lemma:FSIterationFactorsAsProduct}
for the usual Solovay-Tennenbaum iteration and indeed, Martin's axiom implies that
the continuum is regular.

\begin{corollary}
\label{cor:grMAConsistentWithSuslinTrees}
The grounded Martin's axiom is consistent with the existence of a Suslin tree.
\end{corollary}

\begin{proof}
We saw that the model of \grma constructed in the above proof was obtained by
adding a Cohen real to an intermediate extension. 
Adding that Cohen real also adds a Suslin tree by a result of
Shelah~\cite{Shelah1984:TakeSolovayInaccessibleAway}.
\end{proof}

A further observation we can make is that the cofinality of \(\kappa\) plays no
role in the proof of theorem~\ref{thm:CanonicalgrMA} beyond the obvious König's 
inequality requirement on the
value of the continuum. This allows us to obtain models of the grounded Martin's
axiom with a singular continuum and violate cardinal arithmetic properties
which must hold in the presence of Martin's axiom.

\begin{corollary}
The grounded Martin's axiom is consistent with \(2^{<\c}>\c\).
\end{corollary}

\begin{proof}
Starting from some model, perform the construction of theorem~\ref{thm:CanonicalgrMA}
with \(\kappa\) singular. In the extension the continuum equals \(\kappa\).
But the desired inequality is true in any model where the continuum is singular:
of course \(\c=2^\omega\leq 2^{\cf(\c)}\leq 2^{<\c}\) is true but equalities cannot hold 
since the middle two cardinals have different cofinalities by König's inequality.
\end{proof}

On the other hand, assuming we start with a model satisfying GCH, the model of 
theorem~\ref{thm:CanonicalgrMA} will satisfy the best possible alternative
to \(2^{<\c}=\c\), namely \(2^{<\cf(\c)}=\c\). Whether this always happens remains open.

\begin{question}[open]
\label{q:grMAWeakLuzin}
Does the grounded Martin's axiom imply that \(2^{<\cf(\c)}=\c\)?
\end{question}

%
%

\section{The axiom's relation to other fragments of Martin's axiom}

Let us now compare some of the combinatorial consequences of the grounded Martin's 
axiom with those of the usual Martin's axiom. We first make an easy observation.

\begin{proposition}
\label{prop:grMAImpliesMACountable}
The local grounded Martin's axiom implies \mac.
\end{proposition}

\begin{proof}
Fix the cardinal \(\kappa\geq\c\) and the \(\mathrm{ZFC}^-\) ground model 
\(M\subseteq H_{\kappa^+}\) witnessing local \grma. 
Observe that the the model \(M\) contains the poset \(\mathrm{Add}(\omega,1)\), since
its elements are effectively coded by the natural numbers.
This poset is therefore always a valid target for local \grma.
\end{proof}

It follows from the above proposition that the (local) grounded Martin's axiom
will have some nontrivial effects on the cardinal characteristics of the continuum.
In particular, we obtain the following.

\begin{corollary}
\label{cor:grMALargeCharacts}
The local grounded Martin's axiom implies that the cardinals on the right side of 
Cichoń's diagram equal the continuum.
In particular, this holds for both the covering number for category 
\(\mathbf{cov}(\mathcal{B})\) and the reaping number \(\mathfrak{r}\).
\end{corollary}

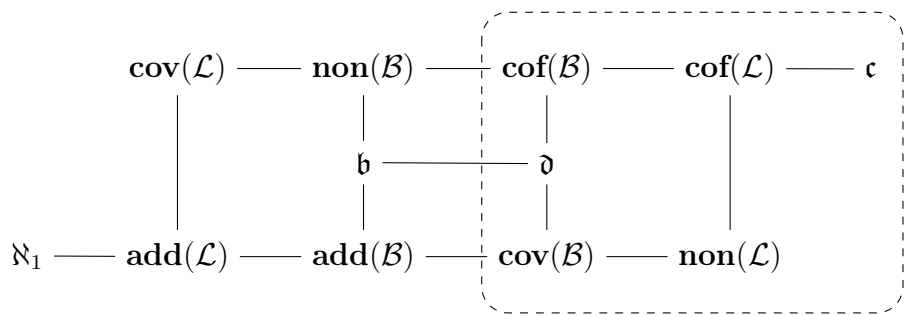
\begin{figure}[ht]
\centering
\begin{tikzpicture}
\matrix (m) [matrix of math nodes,column sep=2em, row sep=1.5em]
{ & \textbf{cov}(\mathcal{L}) & \textbf{non}(\mathcal{B}) & \textbf{cof}(\mathcal{B}) & \textbf{cof}(\mathcal{L}) & \c\\
& & \mathfrak{b} & \mathfrak{d} & & \\
\aleph_1 & \textbf{add}(\mathcal{L}) & \textbf{add}(\mathcal{B}) & \textbf{cov}(\mathcal{B}) & \textbf{non}(\mathcal{L})\\};
\draw (m-1-2)--(m-1-3)--(m-1-4)--(m-1-5)--(m-1-6);
\draw (m-2-3)--(m-2-4);
\draw (m-3-1)--(m-3-2)--(m-3-3)--(m-3-4)--(m-3-5);
\draw (m-1-2)--(m-3-2);
\draw (m-1-3)--(m-2-3)--(m-3-3);
\draw (m-1-4)--(m-2-4)--(m-3-4);
\draw (m-1-5)--(m-3-5);
\draw [rounded corners=10pt, dashed, black] (0.5,2) rectangle (6.1,-2);
\end{tikzpicture}
\caption{Cichoń's diagram; the indicated region is pushed up to \(\c\) under \grma.}
\end{figure}

\begin{proof}
All of the given equalities follow already from \mac; we briefly summarize
the arguments from~\cite{Blass2010:CardinalCharacteristicsHandbook}.

The complement of any nowhere dense subsets of the real line is dense. It follows that, given
fewer than continuum many nowhere dense sets, we can apply \mac to obtain a real number
not contained in any of them. Therefore the real line cannot be covered by fewer
than continuum many nowhere dense sets and, consequently, also not by fewer than
continuum many meagre sets.

To see that the reaping number must be large, observe that, given any infinite
\(x\subseteq\omega\), there are densely many conditions in \(\mathrm{Add}(\omega,1)\)
having arbitrarily large intersection with both \(x\) and \(\omega\setminus x\).
It follows that a Cohen real will split \(x\). Starting with fewer than continuum many
reals and applying \mac, we can therefore find a real splitting all of them, which
means that the original family was not a reaping family.
\end{proof}

But where Martin's axiom strictly prescribes the size of all cardinal characteristics 
of the continuum, the grounded Martin's axiom allows for more leeway in some cases. 
Observe that, since \(\kappa>\omega_1\), the iteration
\(\P\) of theorem~\ref{thm:CanonicalgrMA} contains \(\mathrm{Add}(\omega,\omega_1)\)
as an iterand. 
Thus, by lemma~\ref{lemma:FSIterationFactorsAsProduct}, there is a 
poset \(\overline{\P}\) such that \(\P\) is equivalent 
to \(\overline{\P}\times\mathrm{Add}(\omega,\omega_1)\).

\begin{theorem}
\label{thm:grMASmallCharacts}
It is consistent that the grounded Martin's axiom holds, \textup{CH} fails and the cardinal
characteristics on the left side of Cichoń's diagram, as well as the splitting
number \(\mathfrak{s}\) are equal to \(\aleph_1\).
\end{theorem}

\begin{proof}
Consider a model \(V[G]\) of \grma satisfying \(\c>\aleph_1\) which was obtained by 
forcing  with the iteration \(\P\) from theorem~\ref{thm:CanonicalgrMA} over a model of 
GCH. We  have argued that this model is obtained by adding \(\aleph_1\) many Cohen reals 
to some intermediate extension. We again briefly summarize the standard arguments
for the smallness of the indicated cardinal characteristics in such an extension
(see~\cite{Blass2010:CardinalCharacteristicsHandbook} for details).

Let \(X\) be the set of \(\omega_1\) many Cohen reals added by the final stage of 
forcing. We claim it is both nonmeager and splitting. Note that any real in
\(V[G]\) appears before all of the Cohen reals in \(X\) have appeared. It follows
that every real in \(V[G]\) is split by some real in \(X\). Furthermore, if \(X\) were
meager, it would be contained in a meager Borel set, whose Borel code also
appears before all of the reals in \(X\) do. But this leads to contradiction, since
any Cohen real will avoid any meager set coded in the ground model.
\end{proof}

To summarize, while the grounded Martin's axiom implies that the right side
of Cichoń's diagram is pushed up to \(\c\), it is consistent with the left side dropping
to \(\aleph_1\) (while CH fails, of course). This is the most extreme
way in which the effect of the grounded Martin's axiom on Cichoń's diagram can differ 
from that of Martin's axiom. The
precise relationships under \grma between the cardinal characteristics on the left
warrant further exploration in the future. 

We can consider further the position of the grounded Martin's axiom within the 
hierarchy of the more well-known fragments of Martin's axiom. As we have already 
mentioned, (local) \grma implies \mac. We can strengthen this slightly. Let \macoh
denote Martin's axiom restricted to posets of the form \(\Add(\omega,\lambda)\) for
some \(\lambda\). It will turn out that local \grma also implies \macoh.

\begin{lemma}
\label{lemma:MACohenSmallPosets}
The axiom \macoh is equivalent to its restriction to posets of the form
\(\Add(\omega,\lambda)\) for \(\lambda<\c\).
\end{lemma}

\begin{proof}[Proof\;\footnotemark]
\footnotetext{The proof of the key claim was suggested by Noah Schweber.}
Let \(\P=\Add(\omega,\kappa)\) and fix a collection \(\mathcal{D}\) of \(\lambda<\c\)
many dense subsets of \(\P\). As usual, let \(\Q\) be an elementary substructure of
\(\langle\P,D\rangle_{D\in\mathcal{D}}\) of size \(\lambda\). We shall show that
\(\Q\) is isomorphic to \(\Add(\omega,\lambda)\). The lemma then follows easily.

To demonstrate the desired isomorphism we shall show that \(\Q\) is determined
by the single-bit conditions it contains. More precisely, \(\Q\) contains precisely those
conditions which are meets of finitely many single-bit conditions in \(\Q\).

First note that being a single-bit conditions is definable in \(\P\): 
these are precisely
the coatoms of the order. Furthermore, given a coatom \(p\), its complementary 
coatom \(\bar{p}\) with the single bit flipped is definable from \(p\) as the unique
coatom such that any condition is compatible with either \(p\) or \(\bar{p}\). It follows
by elementarity that the coatoms of \(\Q\) are precisely the single-bit conditions 
contained in \(\Q\) and that \(\Q\) is closed under the operation \(p\mapsto\bar{p}\).
In \(\P\) any finite collection of pairwise compatible coatoms has a meet, therefore
the same holds in \(\Q\) and the meets agree. Conversely, any given condition in \(\P\)
uniquely determines the finitely many coatoms it strengthens and therefore all the
coatoms determined by conditions in \(\Q\) are also in \(\Q\). Taken together, 
this proves the claim.

It follows immediately from the claim that \(\Q\) is isomorphic to \(\Add(\omega,|X|)\)
where \(X\) is the set of coatoms of \(\Q\), and also that \(|X|=|\Q|\).
\end{proof}

\begin{proposition}
\label{prop:grMAImpliesMACohen}
The local grounded Martin's axiom implies \macoh.
\end{proposition}

\begin{proof}
Suppose the local grounded Martin's axiom holds, witnessed by \(\kappa\geq\c\) and a
\(\mathrm{ZFC}^-\) model \(M\subseteq H_{\kappa^+}\). 
In particular, the height of \(M\) is \(\kappa^+\) and
\(M\) contains all of the posets \(\Add(\omega,\lambda)\) for \(\lambda<\kappa^+\).
But this means that Martin's axiom holds for all the posets \(\Add(\omega,\lambda)\)
where \(\lambda<\c\) and lemma~\ref{lemma:MACohenSmallPosets} now implies that
\macoh holds.
\end{proof}

As we have seen, the local grounded Martin's axiom implies some of the weakest
fragments of Martin's axiom. The following corollary of theorem~\ref{thm:grMASmallCharacts} 
tells us, however, that this behaviour stops quite quickly.

\begin{corollary}
\label{cor:grMANotImpliesMAsigma}
The grounded Martin's axiom does not imply \mas.
\end{corollary}
\begin{proof}
By theorem~\ref{thm:grMASmallCharacts} there is a model of the grounded Martin's axiom
where the bounding number is strictly smaller than the continuum. But this is impossible
under \mas, since applying the axiom to Hechler forcing yields, for any family of
fewer than continuum many reals, a real dominating them all.
\end{proof}

As mentioned earlier, Reitz has shown that we can perform class forcing over
any model in such a way that the resulting extension has the same \(H_\c\)
and is also not a set-forcing extension of any ground model. Performing this
construction over a model of \mas (or really any of the standard fragments of
Martin's axiom) shows that \mas does not imply \grma, for the
disappointing reason that the
final model is not a ccc forcing extension of anything.
However, it turns out that already local \grma is independent of \mas, and even
of \mak. 
This places the grounded Martin's axiom, as well as its local version, outside the usual 
hierarchy of fragments of Martin's axiom.

\begin{theorem}
\label{thm:grMASeparateFragment}
Assume \(V=L\) and let \(\kappa>\omega_1\) be a regular cardinal.
Then there is a ccc forcing extension which satisfies
\(\mak+\mathfrak{c}=\kappa\) and in which the local grounded Martin's axiom fails.
\end{theorem}

\begin{proof}
Let \(\P\) be the usual finite-support iteration forcing 
\(\mak+\mathfrak{c}=\kappa\). More precisely, we consider
the names for posets in \(H_\kappa\) using appropriate bookkeeping and append them
to the iteration if they, at that stage, name a Knaster poset.
Let \(G\subseteq \P\) be generic. We claim that the local grounded Martin's axiom 
fails in the extension \(L[G]\).

Notice first that \(\P\), being a finite-support iteration of Knaster
posets, is Knaster. It follows that the product of \(\P\) with any ccc poset is
still ccc. In particular, forcing with \(\P\) preserves the Suslin trees of \(L\).

Now fix a \(\lambda\geq\kappa\) and let \(M\in L[G]\) be a transitive \(\mathrm{ZFC}^-\)
model of height \(\lambda^+\). It is straightforward to see that \(M\) builds
its constructible hierarchy correctly so that, in particular, \(L_{\omega_2}\subseteq
M\). This implies that \(M\) has all of the Suslin trees of \(L\).
Since these trees are still Suslin in \(L[G]\), partially generic filters do not
exist for them and the model \(M\) does not witness
local \grma in \(L[G]\). As \(\lambda\) and \(M\) were completely arbitrary,
local \grma must fail in \(L[G]\).
%
\end{proof}

We summarize the relationships between \grma and the other fragments of Martin's axiom
in the following diagram.

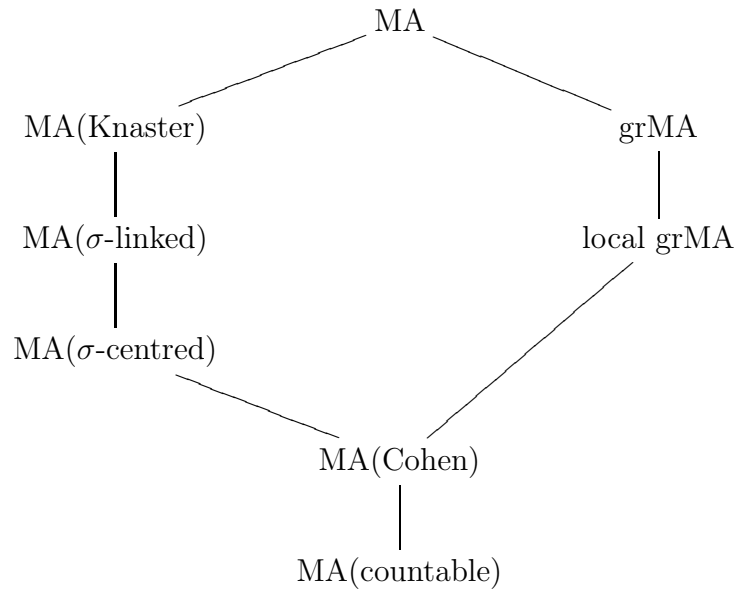
\begin{figure}[ht]
\[
\xymatrix{
&\ma\ar@{-}[dl]\ar@{-}[dr]&\\
\mak\ar@{-}[d]&&\grma\ar@{-}[d]\\
\textup{MA(\(\sigma\)-linked)}\ar@{-}[d]&&\text{local \grma}\ar@{-}[ddl]\\
\mas\ar@{-}[dr]\\
&\mathrm{MA(Cohen)}\ar@{-}[d]\\
&\mathrm{MA(countable)}}
\]
\caption{Diagram of implications between fragments of Martin's axiom}
\end{figure}

Let us mention that it is quite
easy to perform ccc forcing over any model and have \grma fail in the extension.

\begin{corollary}
\label{cor:grMAFailsAfterHechler}
Given any model \(V\) there is a ccc forcing extension \(V[G]\) in which
the local grounded Martin's axiom fails.
\end{corollary}

\begin{proof}
We may assume that CH fails in \(V\). If \(\P\) is the length \(\omega_1\) finite-support
iteration of Hechler forcing and \(G\subseteq \P\) is generic then it is easily seen
that \(G\) is a dominating family in \(V[G]\) and therefore 
the dominating number of \(V[G]\) equals \(\aleph_1\).
It now follows from 
corollary~\ref{cor:grMALargeCharacts} that the local \grma fails in \(V[G]\).
\end{proof}

In the following two sections we shall explore the other side of the coin:
\grma is preserved by certain kinds of ccc forcing.

\section{Adding a Cohen real to a model of the grounded Martin's axiom}

An interesting question when studying fragments of Martin's axiom is what effect 
adding various kinds of generic reals has on it. It was shown by 
Roitman~\cite{Roitman1979:AddingRandomOrCohenRealEffectMA} that
\(\ma_{\aleph_1}\) is destroyed after adding a Cohen or a random real. At the same 
time, it was shown that adding a Cohen real preserves a certain fragment, \mas. 
In this section we follow the spirit of Roitman's arguments to show that
the grounded Martin's axiom is preserved, even with respect to the same ground model, 
after adding a Cohen real.

It is well known that \(\ma+\lnot\mathrm{CH}\) implies that any ccc poset is Knaster 
(recall that a poset \(\P\) is Knaster if any uncountable subset of \(\P\) has in turn an 
uncountable subset of pairwise compatible elements). We start this section by transposing 
this fact to the \grma setting.

\begin{lemma}
\label{lemma:grMACccKnaster}
Let \(V\) satisfy the local grounded Martin's axiom over the ground model \(M\subseteq
H_{\kappa^+}\) and suppose CH fails in \(V\).
Then any poset \(\P\in M\) which is ccc in \(V\) is Knaster in \(V\).
\end{lemma}

\begin{proof}
Let \(\P\) be as in the statement of the lemma and let 
\(A=\set{p_\alpha}{\alpha<\omega_1}\in V\)
be an uncountable subset of \(\P\). We first claim that there is a \(p^*\in \P\) such that
any \(q\leq p^*\) is compatible with uncountably many elements of \(A\). For suppose not.
Then there would be for any \(\alpha<\omega_1\) some \(q_\alpha\leq p_\alpha\) which was
compatible with only countably many elements of \(A\). We could thus choose 
\(\beta(\alpha)<\omega_1\) in such a way that \(q_\alpha\) would be incompatible with any
\(p_\beta\) for \(\beta(\alpha)\leq\beta\). Setting \(\beta_\alpha=\beta^\alpha(0)\)
(meaning the \(\alpha\)-th iterate of \(\beta\)), 
this would mean that \(\set{q_{\beta_\alpha}}{\alpha<\omega_1}\) is an uncountable antichain
in \(\P\), contradicting the fact that \(\P\) was ccc in \(V\).

By replacing \(\P\) with the cone below \(p^*\) and modifying \(A\) appropriately, we may 
assume that in fact every element of \(\P\) is compatible with uncountably many elements of
\(A\). We now let
\(D_\alpha=\bigcup_{\beta\leq\alpha}\P\rest p_\beta\)
for \(\alpha<\omega_1\). The sets \(D_\alpha\) are dense in \(\P\) and by 
\(\grma+\lnot\mathrm{CH}\), we can find, in \(V\), a filter \(H\subseteq \P\) which 
intersects every \(D_\alpha\). But then \(H\cap A\) is an uncountable set of pairwise 
compatible elements.
\end{proof}

We now introduce the main technical device we will use in showing that the grounded
Martin's axiom is preserved when adding a Cohen real. In the proof we will be 
dealing with a two step extension \(W\subseteq W[G]\subseteq
W[G][c]\) where the first step is some ccc extension, the second adds a Cohen
real and \(W[G]\) satisfies the grounded Martin's axiom over \(W\).
To utilize the forcing axiom in \(W[G]\) in verifying it in \(W[G][c]\), we need
to find a way of dealing with (names for) dense sets from \(W[G][c]\) in \(W[G]\).
The termspace forcing construction (due to Laver and possibly independently also 
Abraham, Baumgartner, Woodin and others) comes to mind (for more information
on this construction we point the reader to~\cite{Foreman1983:SaturatedIdeals}), 
however the
posets arising from this construction are usually quite far from being ccc and
are thus unsuitable for our context. We attempt to rectify the situation by
radically thinning out the full termspace poset and keeping only the simplest
conditions.

\begin{definition}
Given a poset \(\P\), a \(\P\)-name \(\tau\) will be called a 
\emph{finite \(\P\)-mixture} if there exists a finite
maximal antichain \(A\subseteq \P\) such that for every \(p\in A\) there is some 
\(x\) satisfying \(p\forces_\P \tau=\check{x}\). The antichain \(A\) is called a 
\emph{resolving antichain} for \(\tau\) and we denote the value \(x\) of \(\tau\)
at \(p\) by \(\tau^p\).
\end{definition}


\begin{definition}
Let \(\P\) and \(\Q\) be posets. The
\emph{finite-mixture term\-space poset} for \(\Q\) over \(\P\) is
\[\Termfin(\P,\Q)=\set{\tau}{\textup{\(\tau\) is a finite \(\P\)-mixture and
\(1\forces_\P\tau\in\check{\Q}\)}}\]
ordered by letting \(\tau\leq\sigma\) iff \(1\forces_\P\tau\leq_{\Q}\sigma\).
\end{definition}


As a side remark, let us point out that in all interesting cases the finite-mixture 
termspace poset is not a regular suborder of the full termspace poset and we can expect 
genuinely different properties. In fact, this occurs as soon as \(\P\) and \(\Q\) are 
nontrivial. To see this, suppose \(\set{p_n}{n<\omega}\) and \(\set{q_n}{n<\omega}\) are 
infinite antichain in \(\P\) and \(\Q\), respectively. By mixing we can find a
\(\tau\in\Term(\P,\Q)\) such that \(p_n\forces\tau=q_n\); we claim that this \(\tau\)
does not have a reduction to \(\Termfin(\P,\Q)\). 
Suppose \(\sigma\in\Termfin(\P,\Q)\) were such a reduction. Then there is 
a condition \(p\) in its resolving antichain that
is compatible with at least two conditions \(p_i\), say \(p_0\) and \(p_1\).
Since \(\sigma\) is a reduction of \(\tau\), it and all stronger conditions are
compatible with \(\tau\), and this means that \(\sigma^p\) is compatible with
\(q_0\) and \(q_1\). Let \(q'\leq\sigma^p,q_1\) and define
a strengthening \(\sigma'\leq\sigma\) by setting \(\sigma'^p=q'\) and keeping the rest of
the mixture the same as in \(\sigma\). But now \(\sigma'\) and \(\tau\) are clearly
incompatible.

In what follows, let us write \(\C=\funcs{<\omega}{2}\).
We should mention two key issues with the finite-mixture termspace poset 
construction. 
Firstly, the construction is very sensitive to
the concrete posets being used. For example, the forthcoming 
lemma~\ref{lemma:CohenFiniteMixturesKnaster} will show that \(\Termfin(\C,\C)\) is
Knaster, but it is not difficult to see that \(\Termfin(\mathrm{ro}(\C),\C)\) already
has antichains of size continuum. Therefore we cannot freely substitute forcing
equivalent posets in the construction. In fact, if \(\B\) is a complete Boolean
algebra then one easily sees that \(\Termfin(\B,\Q)\) consists of exactly those names
that have only finitely many interpretations and, if \(\Q\) is nontrivial
and \(\B\) has no atoms, this poset will have antichains of size continuum.
The second issue is that
it is quite rare for a poset to have a large variety of finite maximal antichains.
Some, such as \(\funcs{<\omega}{\omega}\) or various collapsing posets, have none
at all except the trivial one-element maximal antichain, while others, such as 
\(\funcs{<\omega_1}{2}\), have a few, but they do not capture the structure of
the poset very well. In all of these cases we do not expect the finite-mixture
termspace poset to be of much help. Nevertheless, in the case of Cohen forcing
\(\C\) it turns out to be a useful tool.

\begin{lemma}
\label{lemma:CohenFiniteMixturesKnaster}
If \(\Q\) is a Knaster poset then \(\Termfin(\C,\Q)\) is also Knaster.
\end{lemma}

\begin{proof}
Let \(T=\set{\tau_\alpha}{\alpha<\omega_1}\) be an uncountable subset of 
\(\Termfin(\C,\Q)\) and
choose resolving antichains \(A_\alpha\) for \(\tau_\alpha\). By refining the 
\(A_\alpha\)
we may assume that each of them is a level of the tree \(\C\) and, by thinning out \(T\)
if necessary, that they are all in fact the same level \(A\). Let us enumerate 
\(A=\{s_0,\dotsc,s_k\}\) and write \(\tau_\alpha^i\) instead of \(\tau_\alpha^{s_i}\). 

Since \(\Q\) is Knaster, an uncountable subset \(Z_0\) of \(\omega_1\) such that
the set \(\set{\tau_\alpha^0}{\alpha\in Z_0}\subseteq \Q\) 
consists of pairwise compatible elements. Proceeding
recursively, we can find an uncountable \(Z\subseteq\omega_1\) such that for every
\(i\leq k\) the set \(\set{\tau_\alpha^i}{\alpha\in Z}\subseteq \Q\) 
consists of pairwise compatible elements. 
We can mix the lower bounds of \(\tau_\alpha^i\) and \(\tau_\beta^i\) over the antichain
\(A\) to produce a name \(\sigma_{\alpha\beta}\in \Termfin(\C,\Q)\) such that
\(\sigma_{\alpha\beta}\) is a lower bound for \(\tau_\alpha\) and \(\tau_\beta\).
Thus \(\set{\tau_\alpha}{\alpha\in Z}\) is an uncountable subset of \(T\) 
consisting of pairwise compatible elements, which proves that \(\Termfin(\C,\Q)\) is 
Knaster. 
\end{proof}

The following lemma is somewhat awkward, but it serves to give us a way of
transforming a name for a dense subset of \(\Q\) into a closely related 
actual dense subset of \(\Termfin(\C,\Q)\). With the usual termspace forcing
construction simply taking \(E=\set{\tau}{\forces\tau\in\dot{D}}\) would have sufficed,
but this set is not dense in \(\Termfin(\C,\Q)\), so modifications are necessary.

\begin{lemma}
\label{lemma:CohenFiniteMixturesDenseTranslation}
Let \(\Q\) be poset and \(\dot{D}\) a \(\C\)-name for a
dense subset of \(\Q\). Then for any \(n<\omega\) the set
\begin{align*}
E_n=\{\,\tau\in \Termfin(\C,\Q)\,;\,& \exists A \textup{ a resolving antichain for \(\tau\)}\,
\forall s\in A\colon \\&
n\leq |s| \land \exists s'\leq s\colon s'\forces\tau\in\dot{D}\,\}
\end{align*}
is a dense subset of \(\Termfin(\C,\Q)\).
\end{lemma}

One can think of the set \(E_n\) as the set of those \(\tau\) that have a sufficiently
deep resolving antichain, none of whose elements force \(\tau\) to not be in
\(\dot{D}\).

\begin{proof}
Let \(\sigma\in \Termfin(\C,\Q)\) and let \(A\) be a resolving antichain for it. Any finite refinement
of a resolving antichain is, of course, another resolving antichain, so we may assume that
we already have \(n\leq |s|\) for all \(s\in A\). By fullness
we can find a name \(\rho\) for an element of \(\Q\) such that
\(1\forces_{\C}(\rho\leq \sigma\land \rho\in\dot{D})\). For each \(s\in A\) we can find
an \(s'\leq s\) such that \(s'\forces_{\C} \rho=\check{q}_s\)
for some \(q_s\in\Q\). By mixing the \(q_s\) over the antichain 
\(A\), we get a name \(\tau\in E_n\) such that \(\tau\leq\sigma\), 
which shows that \(E\) is dense in \(\Termfin(\C,\Q)\).
\end{proof}


\begin{theorem}
\label{thm:CohenPreservesgrMA}
Assume the local grounded Martin's axiom holds in \(V\) over the ground model 
\(M\subseteq H_{\kappa^+}\) and let 
\(V[c]\) be obtained by adding a Cohen real to \(V\). Then \(V[c]\) also satisfies 
the local grounded Martin's axiom over the ground model \(M\subseteq H_{\kappa^+}^{V[c]}=
H_{\kappa^+}[c]\).
\end{theorem}

\begin{proof}
By assumption there is a ccc poset \(\P\in M\) such that \(H_{\kappa^+}^V=M[G]\) for 
an \(M\)-generic \(G\subseteq \P\).
We may assume that \(\mathrm{CH}\) fails in \(V\), for otherwise it would also hold in
the final extension \(V[c]\), which would then
satisfy the full Martin's axiom. Consider a poset \(\Q\in M\) which is ccc in \(V[c]\). 
Since \(\C\) is ccc, \(\Q\) must also
be ccc in \(V\) and by lemma~\ref{lemma:grMACccKnaster} is in fact Knaster in \(V\).

Let \(\lambda<\c^{V[c]}\) be a cardinal and let 
\(\mathcal{D}=\set{D_\alpha}{\alpha<\lambda}\in V[c]\) be a collection of dense 
subsets of \(\Q\). Pick names \(\dot{D}_\alpha\in V\) for these such that 
\(1\forces_{\C} \text{``\(\dot{D}_\alpha\subseteq \check{\Q}\) is dense''}\).

Consider \(\R=\Termfin(\C,\Q)\in M\). Note that \(\R\) is computed the same in
\(M\) and in \(V\). It now follows from lemma~\ref{lemma:CohenFiniteMixturesKnaster} 
that \(\R\) is Knaster in \(V\) (although not necessarily in \(M\)).


Let \(E_{\alpha,n}\subseteq \R\) be the dense sets associated to the
\(\dot{D}_\alpha\) as in lemma~\ref{lemma:CohenFiniteMixturesDenseTranslation}.
Write \(\mathcal{E}=\set{E_{\alpha,n}}{\alpha<\lambda,n<\omega}\). Applying
the grounded Martin's axiom in \(V\), 
we get an \(\mathcal{E}\)-generic filter \(H\subseteq \R\).
We will show that the filter generated by the set
\(H^c=\set{\tau^c}{\tau\in H}\) is \(\mathcal{D}\)-generic. Pick
a \(\dot{D}_\alpha\) and consider the set
\begin{align*}
B_\alpha=\{\,s'\in \C\,;\,&\exists \tau\in H\exists A \text{ a resolving antichain for \(\tau\)}
\,\exists s\in A \colon\\& s'\leq s\land s'\forces_{\C}\tau\in \dot{D}_\alpha\,\}
\end{align*}
We will show that \(B_\alpha\) is dense in \(\C\). To that end, pick a \(t\in \C\). Since
\(H\) is \(\mathcal{E}\)-generic, there is some \(\tau\in H\cap E_{\alpha,|t|}\). Let
\(A\) be a resolving antichain for \(\tau\). Since \(A\) is maximal, \(t\) must be compatible
with some \(s\in A\), and, since \(|t|\leq |s|\), we must in fact have \(s\leq t\). But then,
by the definition of \(E_{\alpha,|t|}\), there exists a \(s'\leq s\) such that
\(s'\forces_{\C}\tau\in \dot{D}\). This exactly says that \(s'\in B_\alpha\)
and also \(s'\leq s\leq t\). Thus \(B_\alpha\) really is dense in \(\C\).

By genericity we can find an \(s'\in B_\alpha\cap c\). If \(\tau\in H\) is the corresponding
name, the definition of \(B_\alpha\) now implies that \(\tau^c\in D_\alpha\cap H^c\). Thus
\(H^c\) really does generate a \(\mathcal{D}\)-generic filter.
\end{proof}

The proof is easily adapted to show that, starting from the full grounded Martin's axiom 
in \(V\) over a ground model \(W\), we obtain the full grounded Martin's axiom in
\(V[c]\) over the same ground model \(W\).

\section{Adding a random real to a model of the grounded Martin's axiom}

Our next goal is to prove a preservation theorem for adding a random real. The machinery
of the proof in the Cohen case will be slightly modified to take advantage of
the measure theoretic structure in this context. 

Recall that a measure algebra is a pair \((\B,m)\) where \(\B\) is a complete Boolean
algebra and \(m\colon\B\to[0,1]\) is a countably additive map such that
\(m(b)=1\) iff \(b=1\).

\begin{definition}
Let \((\B,m)\) be a measure algebra and \(0<\varepsilon<1\). A \(\B\)-name \(\tau\) will be
called an \emph{\(\varepsilon\)-deficient finite \(\B\)-mixture} if there is a finite antichain
\(A\subseteq \B\) such that \(m(\sup A)>1-\varepsilon\) and for every \(w\in A\)
there exists some \(x\) such that \(w\forces_\B\tau=\check{x}\).
The antichain \(A\) is called a \emph{resolving antichain} and we denote the value
\(x\) of \(\tau\) at \(w\) by \(\tau^w\).
\end{definition}

\begin{definition}
Let \((\B,m)\) be a measure algebra, \(\Q\) a poset and \(0<\varepsilon<1\).
The \emph{\(\varepsilon\)-deficient
finite mixture termspace poset} for \(\Q\) over \((\B,m)\) is
\[\Termfin^\varepsilon(\B,\Q)=\set{\tau}{\textup{\(\tau\) is an
\(\varepsilon\)-deficient finite \(\B\)-mixture and \(1\forces_\B\tau\in\check{\Q}\)}}\]
ordered by letting \(\tau\leq\sigma\) iff there are resolving antichains \(A_\tau\)
and \(A_\sigma\) such that \(A_\tau\) refines \(A_\sigma\) and 
\(\sup A_\tau\forces_\B\tau\leq \sigma\).
\end{definition}

The following lemma is the analogue of 
lemma~\ref{lemma:CohenFiniteMixturesDenseTranslation} for \(\varepsilon\)-deficient
finite mixtures.

\begin{lemma}
\label{lemma:DeficientFiniteMixDenseTranslation}
Let \((\B,m)\) be a measure algebra, \(\Q\) a poset and \(0<\varepsilon<1\).
If \(\dot{D}\) is a \(\B\)-name for a dense subset of \(\Q\) then 
\[E=\set{\tau\in\Termfin^\varepsilon(\B,\Q)}{\exists A \textup{ a resolving antichain for \(\tau\)}
\colon\sup A\forces_\B\tau\in\dot{D}}\]
is dense in \(\Termfin^\varepsilon(\B,\Q)\).
\end{lemma}

\begin{proof}
Let \(\sigma\in\Termfin^\varepsilon(\B,\Q)\) and pick a resolving antichain 
\(A=\{w_0,\dotsc,w_n\}\) for it. Let \(\delta=m(\sup A)-(1-\varepsilon)\). 
By fullness there are \(\B\)-names \(\rho_i\) for elements of \(\Q\) such that
\(w_i\forces \rho_i\leq \sigma \land \rho\in \dot{D}\). There are
maximal antichains \(A_i\) below \(w_i\) such that each element of
\(A_i\) decides the value of \(\rho_i\). We now choose finite subsets
\(A_i'\subseteq A_i\) such that \(m(\sup A_i)-m(\sup A_i')<\frac{\delta}{n}\).
Write \(A'=\bigcup_i A_i'\). We then have \(m(\sup A')>1-\varepsilon\).
By mixing we can find a \(\B\)-name \(\tau\) for an element of \(\Q\) which is forced 
by each element of \(A'\) to be equal to the appropriate \(\rho_i\). Thus \(A'\)
is a resolving antichain for \(\tau\) and we have
ensured that \(\tau\) is in \(E\) and \(\sigma\leq \tau\).
\end{proof}

In what follows we let \(\Bnull\) be the random Boolean algebra with the induced
Lebesgue measure \(\mu\). The next lemma is the analogue of
lemma~\ref{lemma:CohenFiniteMixturesKnaster}. 

\begin{lemma}
\label{lemma:RandomDeficientMixKnaster}
Let \(\Q\) be a Knaster poset and \(0<\varepsilon<1\). Then 
\(\Termfin^\varepsilon(\Bnull,\Q)\) is Knaster as well.
\end{lemma}

\begin{proof}
Let \(\set{\tau_\alpha}{\alpha<\omega_1}\) be an uncountable subset of 
\(\Termfin^\varepsilon(\Bnull,\Q)\).
Choose resolving antichains \(A_\alpha\) for the \(\tau_\alpha\).
We may assume that there is a fixed \(\delta\) such that
\(1-\varepsilon<\delta<\mu(\sup A_\alpha)\) for all \(\alpha\).
We may also assume that all of the \(A_\alpha\) have the same size \(n\) and enumerate 
them as \(A_\alpha=\{w_\alpha^0,\dotsc,w_\alpha^{n-1}\}\); we shall write
\(\tau_\alpha^i\) instead of \(\tau_\alpha^{w_\alpha^i}\). 
By inner regularity of the measure
we may assume further that the elements of each \(A_\alpha\) are compact.
Using this and the outer regularity of the measure we can find open neighbourhoods
\(w_\alpha^i\subseteq U_\alpha^i\) such that \(U_\alpha^i\) and \(U_\alpha^j\) are
disjoint for all \(\alpha\) and distinct \(i\) and \(j\) and additionally satisfy
\[\mu(U_\alpha^i\setminus w_\alpha^i)<\frac{\delta-(1-\varepsilon)}{n}\]

Fix a countable basis for the topology. Since the \(w^i_\alpha\) are compact,
we may take the \(U^i_\alpha\) to be finite unions of basic opens. Since there
are only countably many such finite unions, we can assume that there are fixed
\(U^i\) such that \(U^i_\alpha=U^i\) for all \(\alpha\). 

We now obtain 
\begin{align*}
\mu(w^i_\alpha\cap w^i_\beta)&=
\mu(w^i_\alpha)-\mu(w_\alpha^i\cap(U^i\setminus w^i_\beta))\\
&\geq \mu(w^i_\alpha)-\mu(U^i\setminus w^i_\beta)
>\mu(w^i_\alpha) - \frac{\delta-(1-\varepsilon)}{n}
\end{align*}
In particular, this gives that \(\sum_i\mu(w^i_\alpha\cap w^i_\beta)>1-\varepsilon\).

Since \(\Q\) is Knaster we may assume that the elements of
\(\set{\tau_\alpha^i}{\alpha<\omega_1}\) are pairwise compatible and
that this holds for any \(i\). Pick lower bounds \(q^i_{\alpha\beta}\)
for the \(\tau_\alpha^i\) and \(\tau_\beta^i\). By mixing we can
construct \(\Bnull\)-names \(\sigma_{\alpha\beta}\) for elements of \(\Q\) such that  
\(w_\alpha^i\cap w_\beta^i\forces \sigma_{\alpha\beta}=q^i_{\alpha\beta}\)
for all \(i\). By construction the \(\sigma_{\alpha\beta}\) are 
\(\varepsilon\)-deficient finite \(\Bnull\)-mixtures and are lower bounds for
\(\tau_\alpha\) and \(\tau_\beta\).
\end{proof}

While the concept of \(\varepsilon\)-deficient finite mixtures makes sense for any measure
algebra, finding a good analogue of the preceding proposition for algebras of
uncountable weight has proven difficult.

\begin{lemma}
\label{lemma:MeasureOfSupOfDirectedInfs}
Let \((\B,m)\) be a measure algebra and \(0<\varepsilon<1\).
Suppose \(\mathcal{A}\) is a family of finite antichains in \(\B\), downward
directed under refinement, such that \(m(\sup A)>1-\varepsilon\) for any 
\(A\in\mathcal{A}\). If we let \(d_{\mathcal{A}}=\inf\set{\sup A}{A\in\mathcal{A}}\)
then \(m(d_{\mathcal{A}})\geq 1-\varepsilon\).
\end{lemma}

\begin{proof}
By passing to complements it suffices to prove the following statement: if \(I\)
is an upward directed subset of \(\B\) all of whose elements have measure less than
\(\varepsilon\) then \(\sup I\) has measure at most \(\varepsilon\).

Using the fact that \(\B\) is complete, we can refine \(I\) to an antichain \(Z\) that
satisfies \(\sup I=\sup Z\). Since \(\B\) is ccc, \(Z\) must be countable. Applying the
upward directedness of \(I\) and the countable additivity of the measure, we can conclude
that \(m(\sup Z)\leq \varepsilon\).\qedhere
\end{proof}

We are finally ready to state and prove the preservation theorem we have been building
towards.

\begin{theorem}
\label{thm:RandomGivesgrMA}
Assume Martin's axiom holds in \(V\) and let \(V[r]\) be obtained by adding a random 
real to \(V\). Then \(V[r]\) satisfies the grounded Martin's axiom over the ground 
model \(V\).
\end{theorem}

\begin{proof}
If \(\mathrm{CH}\) holds in \(V\) then it holds in \(V[r]\) as well, implying that
\(V[r]\) satisfies the full Martin's axiom. We may therefore assume without loss of 
generality that \(V\) satisfies \(\ma+\lnot\mathrm{CH}\).

Assume toward a contradiction that \(V[r]\) does not satisfy the grounded Martin's
axiom over \(V\). Then there exist
a poset \(\Q\in V\) which is ccc in \(V[r]\), a cardinal \(\kappa<\c\) and a collection
\(\mathcal{D}=\set{D_\alpha}{\alpha<\kappa}\in V[r]\) of dense subsets of \(\Q\) such that
\(V[r]\) has no \(\mathcal{D}\)-generic filters on \(\Q\). There must be a condition
\(b_0\in\Bnull\) forcing this. Let \(\varepsilon<\mu(b_0)\).

Since \(\Bnull\) is ccc, \(\Q\) must be ccc in \(V\) and, since \(\ma+\lnot\mathrm{CH}\) 
holds there, is also Knaster there. 
Thus \(\Termfin^\varepsilon(\Bnull,\Q)\in V\) is Knaster by 
lemma~\ref{lemma:RandomDeficientMixKnaster}. We now choose names 
\(\dot{D}_\alpha\)
for the dense sets \(D_\alpha\) such that \(\Bnull\) forces that the \(\dot{D}_\alpha\) are dense
subsets of \(\Q\). Then, by lemma~\ref{lemma:DeficientFiniteMixDenseTranslation}, the
\(E_\alpha\) are dense in \(\Termfin^\varepsilon(\Bnull,\Q)\), where \(E_\alpha\) is defined
from \(\dot{D}_\alpha\) as in that lemma. We can thus obtain, using Martin's axiom 
in \(V\), a filter
\(H\) on \(\Termfin^\varepsilon(\Bnull,\Q)\) which meets all of the \(E_\alpha\).

Pick a resolving antichain \(A_\tau\) for each \(\tau\in H\) and consider 
\(\mathcal{A}=\set{A_\tau}{\tau\in H}\). This family satisfies the hypotheses of
lemma~\ref{lemma:MeasureOfSupOfDirectedInfs}, whence we can conclude that
\(\mu(d_{\mathcal{A}})\geq 1-\varepsilon\), where \(d_{\mathcal{A}}\) is defined as in
that lemma.
Interpreting \(H\) as a \(\Bnull\)-name for a subset of \(\Q\), we now observe that 
\[d_{\mathcal{A}}\forces_{\Bnull}
\text{``\(H\) generates a \(\dot{\mathcal{D}}\)-generic filter on \(\Q\)''}\]
Now, crucially, since we have chosen \(\varepsilon<\mu(b_0)\), the conditions \(b_0\) and
\(d_H\) must be compatible in \(\Bnull\). But this is a contradiction, since they
force opposing statements. Therefore \(V[r]\) really does satisfy
the grounded Martin's axiom over \(V\).
\end{proof}

\begin{corollary}
\label{cor:grMAConsistentWithNoSuslinTrees}
The grounded Martin's axiom is consistent with there being no Suslin trees.
\end{corollary}

\begin{proof}
If Martin's axiom holds in \(V\) and \(r\) is random over \(V\) then \(V[r]\) satisfies
the grounded Martin's axiom by the above theorem and also has no Suslin trees by a 
theorem of Laver~\cite{Laver1987:RandomRealsSuslinTrees}.
\end{proof}

Unfortunately, the employed techniques do not seem to yield the full preservation result 
as in theorem~\ref{thm:CohenPreservesgrMA}. If \(V\) satisfied merely the grounded
Martin's axiom over a ground model \(W\) we would have to argue that the
poset \(\Termfin^\varepsilon(\Bnull,\Q)\) as computed in \(V\) was actually an element of \(W\), so
that we could apply \grma to it. But we cannot expect this to be true if passing
from \(W\) to \(V\) added reals; not only will the termspace posets be computed
differently in \(W\) and in \(V\), even the random Boolean algebras of these two
models will be different. Still, these considerations lead us to the following
improvement to the theorem above.

\begin{theorem}
\label{thm:RandomPreservesDistributivegrMA}
Assume the grounded Martin's axiom holds in \(V\) over the ground model \(W\)
via a forcing which is countably distributive (or, equivalently, does not add reals),
and let \(V[r]\) be obtained by adding a random real to \(V\). Then \(V[r]\) also
satisfies the grounded Martin's axiom over the ground model \(W\).
\end{theorem}

\begin{proof}
By assumption there is a ccc countably distributive poset \(\P\in W\) such that
\(V=W[G]\) for some \(W\)-generic \(G\subseteq \P\). Since \(W\) and \(V\) thus have
the same reals, they must also have the same Borel sets. Furthermore, since the
measure is inner regular, a Borel set having positive measure is witnessed by a
positive measure compact (i.e.\ closed) subset, which means that \(W\) and \(V\)
agree on which Borel sets are null. It follows that the random Boolean algebras
as computed in \(W\) and in \(V\) are the same.

Now let \(0<\varepsilon<1\) and let \(\Q\in W\) be a poset which is ccc in \(V\).
We claim that \(V\) and \(W\) compute the poset \(\Termfin^\varepsilon(\Bnull,\Q)\)
the same. Clearly any \(\varepsilon\)-deficient finite mixture in \(W\)
is also such in \(V\), so we really only need to see that \(V\) has no new such
elements. But \(\Bnull\) is ccc, which means that elements of \(\Q\) have countable
nice names and these could not have been added by \(G\). So \(V\) and \(W\) in fact
agree on the whole termspace poset \(\Term(\Bnull,\Q)\), and therefore also
on the \(\varepsilon\)-deficient finite mixtures.

The rest of the proof proceeds as in theorem~\ref{thm:RandomGivesgrMA}. The key
step there, where we apply Martin's axiom to the poset 
\(\Termfin^\varepsilon(\Bnull,\Q)\), goes through, since we have shown that
this poset is in \(W\) and we may therefore apply \grma to it.
\end{proof}
\noindent Just as in theorem~\ref{thm:CohenPreservesgrMA} we may replace the grounded Martin's
axiom in the above theorem with its local version.

It is not immediately obvious that the hypothesis of the above theorem is
ever satisfied in a nontrivial way, that is, whether \grma can ever hold via a
\emph{nontrivial} countably distributive extension. The following theorem,
due to Larson, shows
that this does happen and gives yet another construction of a model of \grma.
For the purposes of this theorem we shall call a Suslin tree \(T\) \emph{homogeneous}
if for any two nodes \(p,q\in T\) of the same height, the cones below them are
isomorphic. Note that homogeneous Suslin trees may be constructed from \(\diamondsuit\).

\begin{theorem}[Larson~\cite{Larson1999:SmaxVariationForOneSuslinTree}]
\label{thm:grMAAfterSuslinTree}
Let \(\kappa>\omega_1\) be a regular cardinal satisfying \(\kappa^{<\kappa}=\kappa\)
and let \(T\) be a homogeneous Suslin tree. Then there is a ccc poset \(\P\) such that, given a
\(V\)-generic \(G\subseteq\P\), the tree \(T\) remains Suslin in \(V[G]\) and, if
\(b\) is a generic branch through \(T\), the extension \(V[G][b]\) satisfies
\(\c=\kappa\) and the grounded Martin's axiom over the ground model \(V[G]\).
\end{theorem}

\begin{proof}
The idea is to attempt to force \(\ma +\c=\kappa\), but only using posets that
preserve the Suslin tree \(T\). More precisely, fix a well-order \(\triangleleft\)
of \(H_\kappa\) of length \(\kappa\) and define \(\P\) as the length \(\kappa\) finite
support iteration which forces at stage \(\alpha\) with the next \(\P_\alpha\)-name
for a poset \(\dot{\Q}_\alpha\) such that \(\P_\alpha\) forces that 
\(T\times\dot{\Q}_\alpha\) is ccc.

Let \(G\subseteq\P\) be \(V\)-generic. It is easy to see by induction
that \(\P_\alpha\times T\) is ccc for all \(\alpha\leq\kappa\); 
the successor case is clear from the definition of the iteration \(\P\)
and the limit case follows by a \(\Delta\)-system argument. We can thus conclude that
\(T\) remains a Suslin tree in \(V[G]\). Furthermore, standard arguments show that
there are exactly \(\kappa\) many reals in \(V[G]\) and that this extension satisfies
Martin's axiom for small posets which preserve \(T\), i.e.\ those \(\Q\) such that
\(\Q\in H_\kappa^{V[G]}\) and \(\Q\times T\) is ccc.

Finally, let us see that adding a branch
\(b\) through \(T\) over \(V[G]\) yields a model of the grounded Martin's axiom over
\(V[G]\). Thus let \(\Q\in V[G]\) be a poset which is ccc in \(V[G][b]\) and has size
less than \(\kappa\) there. There is a condition in \(T\) forcing that \(\Q\) is ccc,
so by our homogeneity assumption \(T\) forces this, meaning that \(\Q\times T\) is ccc
in \(V[G]\).
The key point now is that, since \(T\) is countably
distributive, all of the maximal antichains (and open dense subsets) of \(\Q\) in 
\(V[G][b]\) are already in \(V[G]\). Furthermore, any collection \(\mathcal{D}\) 
of less than \(\kappa\)
many of these in \(V[G][b]\) can be covered by some \(\widetilde{\mathcal{D}}\)
in \(V[G]\) of the same size. Our observation from the previous paragraph then
yields a \(\widetilde{\mathcal{D}}\)-generic filter for \(\Q\) in \(V[G]\)
and therefore \(V[G][b]\) satisfies \grma over \(V[G]\) by lemma~\ref{lemma:grMASmallPosets}.
\end{proof}

Starting from a Suslin tree with a stronger homogeneity property, Larson
also shows that there are no Suslin trees in the extension \(V[G][b]\) above.
This gives an alternative proof of corollary~\ref{cor:grMAConsistentWithNoSuslinTrees}.

From the argument of theorem~\ref{thm:grMAAfterSuslinTree} we can actually extract 
another preservation result for \grma.

\begin{theorem}
\label{thm:SuslinPreservesgrMA}
Assume the grounded Martin's axiom holds in \(V\) over the ground model \(W\)
and let \(T\in V\) be a Suslin tree. If \(b\subseteq T\) is a generic branch
then \(V[b]\) also satisfies the grounded Martin's axiom over the ground model \(W\).
\end{theorem}

\begin{proof}
The point is that, just as in the proof of theorem~\ref{thm:grMAAfterSuslinTree},
forcing with \(T\) does not add any new maximal antichains to posets from \(W\)
that remain ccc in \(V[b]\) and any collection of these antichains in \(V[b]\) can 
be covered by a collection of the same size in \(V\).
\end{proof}

If \grma holds over a ground model that reals have been added to, it seems harder
to say anything about preservation after adding a further random real. Nevertheless,
we fully expect the answers to the following question to be positive.

\begin{question}[open]
Does adding a random real to a model of \(\grma\) preserve \(\grma\)? Does it preserve
it with the same witnessing ground model?
\end{question}

Generalizations of theorems~\ref{thm:CohenPreservesgrMA} and~\ref{thm:RandomGivesgrMA}
to larger numbers of reals added seem the natural next step in the exploration of
the preservation phenomena of the grounded Martin's axiom. 
Such preservation results would also help in determining the compatibility
of \grma with various configurations of the cardinal characteristics on the left
side of Cichoń's diagram.
The constructions \(\Termfin\) and \(\Termfin^\varepsilon\)
seem promising, but obtaining a good chain condition in any case at all, except those shown,
has proven difficult.

\section{The grounded proper forcing axiom}

We can, of course, also consider grounded versions of other forcing axioms. We
define one and note that similar definitions can be made for
\(\mathrm{grSPFA},\mathrm{grMM}\) and so on.

\begin{definition}
The \emph{grounded proper forcing axiom} (\grpfa) asserts that \(V\) is a forcing
extension of some ground model \(W\) by a proper poset and \(V\) satisfies the
conclusion of the proper forcing axiom for posets \(\Q\in W\) which are still
proper in \(V\).
\end{definition}

\begin{theorem}
\label{thm:CanonicalgrPFA}
Let \(\kappa\) be supercompact. Then there is a proper forcing extension that
satisfies \(\grpfa+\c=\kappa=\omega_2\) and in which \pfa, and even \ma, fails.
\end{theorem}

The idea of the proof is similar to the one used in the proof of 
theorem~\ref{thm:CanonicalgrMA}. Specifically, we modify the Baumgartner
PFA iteration, starting from a supercompact cardinal, to only use posets coming
from the ground model. Since the iteration uses countable support, a product analysis
like the one given in lemma~\ref{lemma:FSIterationFactorsAsProduct} is not possible,
but this turns out not to matter for this particular result.

\begin{proof}
Start with a Laver function \(\ell\) for \(\kappa\) and build a countable-support 
forcing iteration \(\P\) of length \(\kappa\) which forces at stage \(\alpha\)
with \(\dot{\Q}\), some full name for the poset \(\Q=\ell(\alpha)\) if it is proper at 
that stage and with trivial forcing otherwise. Note that \(\P\) is proper. 
Now let \(V[G][H]\) be a forcing extension by \(\R=\Add(\omega,1)\times\P\).
We claim that \(V[G][H]\) is the required model.

Since the Laver function \(\ell\) will quite often output the poset
\(\Add(\omega,1)\) and this will always be proper, the iteration \(\P\) will 
add reals unboundedly often. 
Furthermore, since \(\R\) is \(\kappa\)-cc, we will obtain \(\c=\kappa\) in \(V[G][H]\).

Next we wish to see that \(\kappa=\omega_2^{V[G][H]}\). For this it suffices to see
that any \(\omega_1<\lambda<\kappa\) is collapsed at some point during the iteration.
Recall the well-known fact that any countable-support iteration of nontrivial posets 
adds a Cohen subset of \(\omega_1\) at stages of cofinality \(\omega_1\) and therefore
collapses the continuum to \(\omega_1\) at those stages. Now fix some \(\omega_1<\lambda
<\kappa\). Since \(\P\) ultimately adds \(\kappa\) many reals and is \(\kappa\)-cc,
there is some stage \(\alpha\) of the iteration such that \(\P_\alpha\) has already
added \(\lambda\) many reals and therefore \(\c^{V[H_\alpha]}\geq\lambda\).
Since \(\P_\alpha\) is proper, the rest of the iteration \(\P\rest[\alpha,\kappa)\)
is a countable-support iteration in \(V[H_\alpha]\) and the fact mentioned above
implies that \(\lambda\) is collapsed to \(\omega_1\) by this tail of the iteration.

Note that \(\R\) is proper, since \(\P*\Add(\omega,1)\) is proper and has a dense subset
isomorphic to \(\R\). 
To verify that \grpfa holds in \(V[G][H]\) let \(\Q\in V\) be a poset that is
proper in \(V[G][H]\) and let \(\mathcal{D}=\set{D_\alpha}{\alpha<\omega_1}\in V[G][H]\)
be a family of dense subsets of \(\Q\). In \(V\) we can fix (for some large enough 
\(\theta\)) a \(\theta\)-supercompactness embedding \(j\colon V\to M\) such that
\(j(\ell)(\kappa)=\Q\). Since the Cohen real forcing is small, the embedding \(j\)
lifts to a \(\theta\)-supercompactness embedding \(j\colon V[G]\to M[G]\).
We can factor \(j(\P)\) in \(M[G]\) as \(j(\P)=\P*\Q*\Ptail\). Let
\(h*\Htail\subseteq \Q*\Ptail\) be \(V[G][H]\)-generic. As usual, we can now lift
the embedding \(j\) in \(V[G][H*h*\Htail]\) to 
\(j\colon V[G][H]\to \overline{M}=M[G][H*h*\Htail]\).
Note that the closure of this embedding implies that \(j[h]\in \overline{M}\).
But \(j[h]\) is a \(j(\mathcal{D})\)-generic filter on \(j(\Q)\) in
\(\overline{M}\) and so, by elementarity, there is a \(\mathcal{D}\)-generic filter
on \(\Q\) in \(V[G][H]\).

Finally, since we can see \(V[G][H]\) as obtained by adding a Cohen real to an 
intermediate extension and since CH fails there, PFA and even Martin's axiom must fail 
there by Roitman's~\cite{Roitman1979:AddingRandomOrCohenRealEffectMA}.
\end{proof}

With regard to the above proof, we should mention that one usually argues that
\(\kappa\) becomes \(\omega_2\) after an iteration similar to ours because
at many stages the poset forced with was explicitly a collapse poset
\(\Coll(\omega_1,\lambda)\).
In our case, however, the situation is different. It turns out that
a significant number of proper posets from \(V\) (the collapse posets among them)
cease to be proper as soon as we add the initial Cohen real. 
Therefore the possibility of choosing
\(\Coll(\omega_1,\lambda)\) never arises in the construction of the iteration
\(\P\) and a different argument is needed. We recount a proof of this
fact below. The argument is essentially due to Shelah, as communicated
by Goldstern in~\cite{MO:Goldstern2015:PreservationOfProperness}.

\begin{theorem}[Shelah]
\label{thm:NotProperAfterAddingReal}
Let \(\P\) be a ccc poset and let \(\Q\) be a countably distributive poset which
collapses \(\omega_2\). Let \(G\subseteq\P\) be \(V\)-generic. If \(V[G]\)
has a new real then \(\Q\) is not proper in \(V[G]\).
\end{theorem}

\begin{proof}
Fix at the beginning a \(\Q\)-name \(\dot{f}\), forced to be a bijection between 
\(\omega_1\) and \(\omega_2^V\). Let \(\theta\) be a sufficiently large regular cardinal.
By claim XV.2.12 of~\cite{Shelah1998:ProperImproperForcing} we can label the nodes
\(s\in\funcs{<\omega}{2}\) with countable models \(M_s\prec H_\theta\) such that:
\begin{itemize}
\item the \(M_s\) are increasing along each branch of the tree \(\funcs{<\omega}{2}\);
\item \(\P,\Q,\dot{f}\in M_\emptyset\);
\item there is an ordinal \(\delta\) such that \(M_s\cap\omega_1=\delta\) for all \(s\);
\item for any \(s\) there are ordinals \(\alpha_s<\beta_s<\omega_2\) such that
\(\alpha_s\in M_{s\concat 0}\) and \(\alpha_s\notin M_{s\concat 1}\) and 
\(\beta_s \in M_{s\concat 1}\).
\end{itemize}

Now consider, in \(V[G]\), the tree of models \(M_s[G]\). By the argument given in
the proof of lemma~\ref{lemma:grMASmallPosets}, the models \(M_s[G]\) are elementary
in \(H_\theta^{V[G]}\) and, since \(\P\) is ccc, we still have
\(M_s[G]\cap\omega_1=\delta\). Let \(M=M_r[G]\) be the branch model determined by
the new real \(r\in V[G]\). We shall show that there are no \(M\)-generic
conditions in \(\Q\).

Suppose that \(q\) were such a generic condition. We claim that \(q\) forces that
\(\dot{f}\rest\delta\) maps onto \(M\cap\omega_2^V\) (note that \(\dot{f}\) still names
a bijection \(\omega_1\to\omega_2^V\) over \(V[G]\)). First, suppose that \(q\) does not
force that \(\dot{f}[\delta]\subseteq M\). Then we can find \(q'\leq q\) and an
\(\alpha<\delta\) such that \(q'\forces\dot{f}(\alpha)\notin M\).
But if \(q'\in H\subseteq\Q\) is generic then \(M[H]\) is an elementary substructure of
\(H_\theta^{V[G][H]}\) and, of course, \(f,\alpha\in M[H]\), leading to a contradiction.

Conversely, suppose that \(q\) does not force that
\(M\cap\omega_2^V\subseteq \dot{f}[\delta]\). We can again find \(q'\leq q\) and an
\(\alpha\in M\cap\omega_2\) such that \(q'\forces\alpha\notin\dot{f}[\delta]\).
Let \(q'\in H\subseteq\Q\) be generic. As before, \(M[H]\) is an elementary substructure
of \(H_\theta^{V[G][H]}\) and \(f,\alpha\in M[H]\). Since 
\(f\colon\omega_1\to\omega_2^V\) is a bijection, we must have \(f^{-1}(\alpha)\in M[H]\).
But by construction \(f^{-1}(\alpha)\) is an ordinal greater than \(\delta\)
while simultaneously \(M[H]\cap\omega_1=\delta\) by the \(M\)-genericity of \(q\),
giving a contradiction.

Fixing our putative generic condition \(q\), we can use the countable distributivity of 
\(\Q\) in \(V\) to see that \(\dot{f}\rest\delta\), and consequently \(M\cap\omega_2\),
exist already in \(V\). But we can extract \(r\) from \(M\cap\omega_2\).

Notice that, given a model \(M_{s\concat 1}\) in our original tree, no elementary
extension \(M_{s\concat 1}\prec X\) satisfying \(X\cap\omega_1=\delta\)
can contain \(\alpha_s\). This is because \(M_{s\concat 1}\) contains a bijection
\(g\colon \omega_1\to\beta_s\) and, by elementarity, \(g\) must restrict to a bijection
between \(\delta\) and \(M_{s\concat 1}\cap \beta_s\). 
But seen from the viewpoint of \(X\), that same function \(g\) must restrict to a
bijection between \(\delta\) and \(X\cap \beta_s\) and so \(X\cap \beta_s=
M_{s\concat 1}\cap\beta_s\).

Using this fact, we can now extract \(r\) from \(M\cap\omega_2\) in \(V\). Specifically,
we can decide at each stage whether the branch determined by \(r\) went left or
right depending on whether \(\alpha_s\in M\cap\omega_2\) or not. We conclude that
\(r\) appears already in \(V\), contradicting our original assumption. Therefore
there is no generic condition \(q\) as above and \(\Q\) is not proper in \(V[G]\).
\end{proof}


Ultimately, one hopes that by grounding the forcing axiom we lower its consistency 
strength while still being able to carry out some of the usual arguments and 
obtain some of the standard consequences. However, 
theorem~\ref{thm:NotProperAfterAddingReal} severely limits the kind of arguments we
can carry out under \grpfa. Many arguments involving \pfa use, among other things,
collapsing posets such as \(\Coll(\omega_1,2^\omega)\).
In contrast, if the poset witnessing \grpfa in a model is
any kind of iteration that at some stage added, say, a Cohen real, the theorem
prevents us from applying the forcing axiom to any of these collapsing posets.
It is thus unclear exactly how much strength of \pfa can be recovered from \grpfa.
In particular, while \grpfa implies that CH fails, the following key question 
remains open:

\begin{question}[open]
\label{q:grPFAContinuum}
Does \grpfa imply that the continuum equals \(\omega_2\)?
\end{question}

Regarding the relation of \grpfa to other forcing axioms, a lot remains unknown.
Theorem~\ref{thm:CanonicalgrPFA} shows that \grpfa does not imply \ma. Beyond this
a few more things can be said.

\begin{proposition}
Martin's axiom does not imply \grpfa.
\end{proposition}

\begin{proof}
Starting over some model, force with the Solovay-Tennenbaum iteration to produce
a model \(V[G]\) satisfying \ma and \(\c=\omega_3\). Now perform Reitz's ground
axiom forcing (cf.~\cite{Reitz2007:GroundAxiom}) above \(\omega_3\) to produce a model
\(V[G][H]\) still satisfying \ma and \(\c=\omega_3\) but which is not a
set-forcing extension of any model (note that \(H\) is added by class-sized forcing).
Therefore the only way \(V[G][H]\) could satisfy \grpfa is if it actually
satisfied \pfa in full. But that cannot be the case since \pfa implies that
the continuum equals \(\omega_2\).
\end{proof}

\begin{proposition}
The grounded proper forcing axiom does not imply \mas (and not even \(\mathrm{MA}_{\aleph_1}(\textup{\(\sigma\)-centred})\)).
\end{proposition}

\begin{proof}
We could have replaced the forcing \(\Add(\omega,1)\times\P\)
in the proof of theorem~\ref{thm:CanonicalgrPFA} with \(\Add(\omega,\omega_1)\times\P\)
without issue. 
As in the proof of theorem~\ref{thm:grMASmallCharacts} we get a model
whose bounding number equals \(\aleph_1\), but this contradicts
\(\mathrm{MA}_{\aleph_1}(\text{\(\sigma\)-centred})\) as in the proof of
corollary~\ref{cor:grMANotImpliesMAsigma}.
\end{proof}


\begin{figure}[h!t]
\[
\xymatrix{
&\pfa\ar@{-}[d]\ar@{-}[dr]\\
&\ma\ar@{-}[dl]\ar@{-}[dr]& \grpfa\ar@{.}[d]^{?}\\
\mak\ar@{-}[d]&&\grma\ar@{-}[d]\\
\textup{MA(\(\sigma\)-linked)}\ar@{-}[d]&&\text{local \grma}\ar@{-}[ddl]\\
\mas\ar@{-}[dr]\\
&\mathrm{MA(Cohen)}\ar@{-}[d]\\
&\mathrm{MA(countable)}}
\]
\caption{Expanded diagram of forcing axioms}
\end{figure}
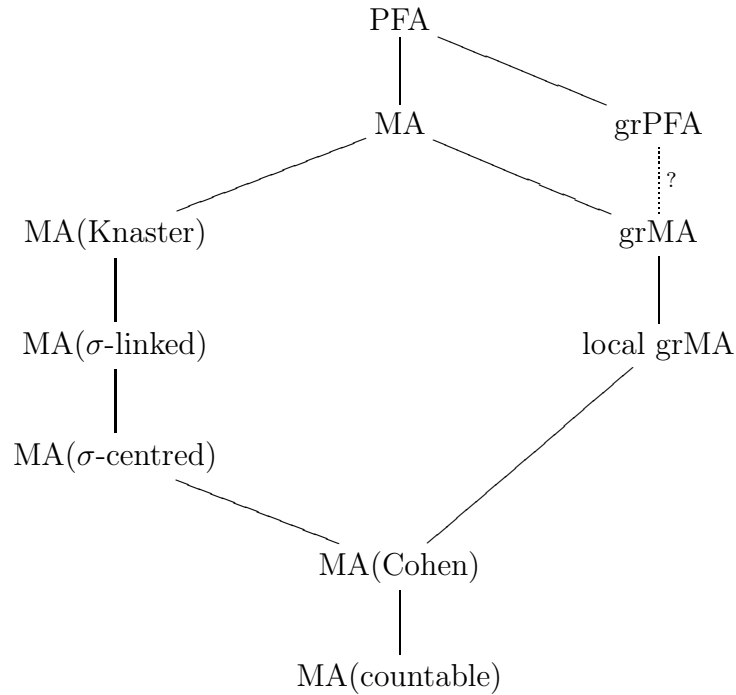

While it is easy to see that \grpfa implies \(\mathrm{MA}_{\aleph_1}(\mathrm{Cohen})\),
whether or not it even implies \mac is unclear (a large part of the problem being that
we do not have an answer to question~\ref{q:grPFAContinuum}). But an even more
pressing question concerns the relationship between \grpfa and \grma:

\begin{question}[open]
Does \grpfa imply \grma?
\end{question}

Even if the answer to question~\ref{q:grPFAContinuum} turns out to be positive, we
conjecture that the answer to this last question is negative. We expect that
it is possible to use the methods of \cite{Reitz2007:GroundAxiom} or, more
generally, \cite{FuchsHamkinsReitz2015:SetTheoreticGeology} in combination
with the forcing construction of theorem~\ref{thm:CanonicalgrPFA} to produce
a model of \grpfa which has no ccc ground models and in which \ma (and consequently
also \grma) fails.

\singlespacing
\bibliographystyle{amsplain}
\bibliography{bibbase}

\end{document}